\documentclass[twoside,a4paper,reqno]{amsart}

\usepackage{amsmath}
\usepackage{amsthm}
\usepackage{amssymb}
\usepackage{amsfonts}
\usepackage{amsxtra}
\usepackage{stmaryrd}
\usepackage{float}
\usepackage{microtype}
\usepackage{calc}
\usepackage{subfig}
\usepackage{mathtools} 
\usepackage{caption}
\usepackage{upgreek}

\usepackage{enumitem,ulem}
\setenumerate{label={\rm (\alph{*})}}

\usepackage[rose,frak,hyperref,operator,comment]{paper_diening}

\usepackage{color}

\usepackage{graphicx}

\newtheorem{proposition}[equation]{Proposition}
\newtheorem{theorem}[equation]{Theorem}
\newtheorem{lemma}[equation]{Lemma}
\newtheorem{corollary}[equation]{Corollary}

\newtheorem{assumption}[equation]{Assumption}

\newtheorem{remark}[equation]{Remark}

\numberwithin{equation}{section}

\providecommand{\meantmp}[2]{#1\langle{#2}#1\rangle}
\providecommand{\mean}[1]{\meantmp{}{#1}}

\providecommand{\jumptmp}[2]{#1\llbracket{#2}#1\rrbracket}
\providecommand{\jump}[1]{\jumptmp{}{#1}}

\providecommand{\avgtmp}[2]{#1\{{#2}#1\}}
\providecommand{\avg}[1]{\avgtmp{}{#1}}
\providecommand{\bigavg}[1]{\avgtmp{\big}{#1}}

\providecommand{\Rdn}{{\setR^{d \times n}}}

\providecommand{\flux}[1]{{\widehat{#1}}}

\providecommand{\nablaDG}{\boldsymbol{\mathcal{G}}_h^k}
\providecommand{\DG}{{\rm DG}}
\providecommand{\SZ}{{\rm SZ}} 
\providecommand{\PiDG}{{\Uppi_{h}^{k}}}
\providecommand{\Pia}{{\Uppi_{h}^{0}}}
\providecommand{\Pian}{{\Uppi_{h_n}^{0}}}
\providecommand{\PiSZ}{{\Uppi^k_{\SZ}}}

\providecommand{\Wz}{\ensuremath{W_{\Gamma_D}}}

\providecommand{\Xhk}{X_h^k}
\providecommand{\Vhk}{U_h^k}

\providecommand\AAA{\boldsymbol{\mathcal{A}}}
\providecommand{\aaa}{\avg{\abs{\Pia \bL_h}}}
\providecommand{\aaan}{\avg{\abs{\Pian \bL_{h_n}}}}
\providecommand{\aaal}{\avg{\abs{\Pia (\nablaDG \bfu_h +\Rhk\bfu_D^*)}}}

\providecommand{\Lh}{{\nablaDG \bfu_h +\Rhk\bfu_D^*}}
\providecommand{\divo}{\divergence}

\newcommand{\Ghk}{\boldsymbol{\mathcal{G}}_h^k}

\newcommand{\Rhk}{\boldsymbol{\mathcal{R}}_h^k}

\newcommand{\WDG}{W^{1,p}(\mathcal{T}_h)}
\newcommand{\WDGpsi}{W^{1,\psi}(\mathcal{T}_h)}
\newcommand{\WDGphi}{W^{1,\phi}(\mathcal{T}_h)}
\newcommand{\WDGpsiD}{W^{1,\psi}_{\Gamma_D}(\mathcal{T}_h)}

\begin{document}

\title[LDG approximation for systems with Orlicz-structure]{{ Convergence
  analysis of a Local Discontinuous Galerkin approximation for
  nonlinear systems with balanced Orlicz-structure}}

\author{A.~Kaltenbach,  M.~\Ruzicka}

\begin{abstract}
 In this paper, we investigate a Local Discontinuous Galerkin~(LDG) approximation for systems with balanced Orlicz-structure. We propose a~new~numerical flux, which yields optimal convergence
rates for linear~ansatz~functions. In {particular}, our approach yields a
unified treatment for problems with $(p,\delta)$-structure for arbitrary $p
\in (1,\infty)$ and $\delta\ge 0$.
\end{abstract}


\maketitle

\section{Introduction}\label{intro}
\thispagestyle{empty}
We consider the numerical approximation of the nonlinear 
system
\begin{equation}
  \label{eq:p-lap}
  \begin{aligned}
    -\divo \AAA(\nabla\bu)&=\bfg -\divo \bfG \qquad&&\text{in }\Omega\,,
    \\[-0.5mm]
    \bu&=\bfu_D \qquad&&\text{on }\Gamma_D\,,
    \\[-0.5mm]
    \AAA(\nabla\bu) \bfn &= \bfa_N &&\text{on } \Gamma_N\,,
  \end{aligned}
\end{equation}
using a Local Discontinuous Galerkin (LDG) scheme.  More~precisely,~for~given~data $\bfg$, $\bfG$, $\bfu_D$ and $\bfa_N$, we seek a
vector field $\bu=(u_1,\ldots, u_d)^\top:\Omega\to \setR^d$ solving \eqref{eq:p-lap}. Here,
$\Omega\subseteq\setR^n$, $n \ge 2$, is a polyhedral, bounded
domain~with~Lipschitz~continuous~boundary $\partial \Omega$, \hspace*{-0.1mm}which \hspace*{-0.1mm}is
\hspace*{-0.1mm}disjointly \hspace*{-0.1mm}divided \hspace*{-0.1mm}into \hspace*{-0.1mm}a \hspace*{-0.1mm}Dirichlet~\hspace*{-0.1mm}part~\hspace*{-0.1mm}$\Gamma_D$,~\hspace*{-0.1mm}where~\hspace*{-0.1mm}we~\hspace*{-0.1mm}\mbox{assume}~\hspace*{-0.1mm}that~\hspace*{-0.1mm}${\abs{\Gamma_D}\!>\! 0}$, and a Neumann
part~$\Gamma_N$, i.e., $\Gamma_D \cup \Gamma_N=\partial \Omega$ and
$\Gamma_D \cap \Gamma_N=\emptyset$.  By
$\bn :\partial\Omega\to \mathbb{S}^{n-1}$, we denote the unit normal
vector field to $\partial \Omega$ pointing
outward. We consider the case that
$\AAA\colon \setR^{d\times n}\to \setR^{d\times n}$ is a nonlinear
operator having $\phi$-structure for some balanced N-function $\phi$ (cf.~Section~\ref{sec:stress_tensor}). The relevant example falling into~this~class~is
\begin{equation*}
  \AAA(\nabla \bu) = \frac {\phi'\big (\abs{\nabla \bu}\big )}{\abs{\nabla \bu}}\, \nabla\bu \,.
\end{equation*}
Introducing the additional unknowns $\bL:\Omega\to \setR^{d\times n}$
and $\bA:\Omega\to \setR^{d\times n}$, the system \eqref{eq:p-lap} can
be re-written as a "first order"
system,~i.e., 
\begin{align}
  \begin{alignedat}{2}
    \bfL = \nabla \bfu\,,\quad \bfA &= \AAA(\bfL)\,, \quad
    -\divergence \bfA  = \bfg -\divo \bfG&\qquad&\text{in $\Omega$}\,,
    \\[-0.5mm]
    \bfu &= \bfu_D &\qquad&\text{on $\Gamma_D$}\,,
    \\[-0.5mm]
    \bfA \bfn &= \bfa_N &\qquad&\text{on $\Gamma_N$}\,.
  \end{alignedat}\label{eq:dg-p-lap}
\end{align}
Discontinuous Galerkin (DG) methods for elliptic problems have been
introduced in the late 90's. They are by now well-understood and
rigorously analyzed~in~the~context of linear~elliptic~problems
(cf.~\cite{arnold-brezzi} for the Poisson
problem). In~contrast~to~this,~only~few papers treat $p$-Laplace type
problems with DG methods
(cf.~\cite{ern-p-laplace,BufOrt09,dkrt-ldg,CS16,sip,QS19}), or
non-conforming methods (cf.~\cite{Bar21}). There exists even fewer
numerical~\mbox{investigations} of problems with
Orlicz-structure. Steady problems with Orlicz-structure are treated,
using finite element methods (FEM) in \cite{dr-interpol,DKS13,BC15}
and~non-conforming~\mbox{methods}~in~\cite{Bar21}. Unsteady problems
are investigated in \cite{EW13,Ruf17}. To the best of
the~author's~\mbox{knowledge}, there are no studies of steady problems
with Orlicz-structure using DG methods, except for \cite[Remark
2.3]{dkrt-ldg}, where it is mentioned that the results for the
$p$-Laplace can be extended to balanced N-functions. Note that the
convergence rates in \cite{dkrt-ldg} are sub-optimal (cf.~Remark
\ref{rem:dis}) since the
continuous solution $\bu$ of \eqref{eq:p-lap} satisfies the flux
formulation of \eqref{eq:p-lap}  with an additional term defined on
interior and Dirichlet~faces (cf.~\eqref{eq:cont}). The main purpose
of this paper is to overcome~this~drawback.~Moreover, we state our
results in the context of balanced N-functions to make clear that the
usual distinction between the cases $p\ge 2$ and $p\le 2$~for~\mbox{$p$-Laplace}~problems~is~not needed.

To this end, we introduce a new numerical flux
(cf.~\eqref{def:flux-A}), which allows us to establish convergence of
DG-solutions to a weak solution of the~system \eqref{eq:p-lap} for
general right-hand sides $\bg - \divo \bG$ (cf.~Theorem
\ref{thm:minty}), and error estimates if $\bfG=\bfzero$ (cf.~Theorem
\ref{thm:error}, Corollary~\ref{cor:error}). The~convergence rates are
optimal for linear ansatz functions.~Further,~our~approach yields a
unified treatment of problems with
$(p,\delta)$-structure~(cf.~\cite{dkrt-ldg}), ${p \in (1,\infty)}$,
${\delta \ge 0}$ and recovers in the DG setting the results in
\cite{eb-liu,dr-interpol} (using~FEM)~and~\cite{Bar21}~(using~special
non-conforming~methods). The presence of the shift in the new flux is
analogous to the gradient shift in the natural distance. It takes into
account the structure of the nonlinear problem \eqref{eq:p-lap}
(cf.~Proposition~\ref{lem:hammer}, Remark \ref{rem:natural_dist},
Remark \ref{rem:dis}).

\textit{This \hspace*{-0.1mm}paper \hspace*{-0.1mm}is
  \hspace*{-0.1mm}organized \hspace*{-0.1mm}as
  \hspace*{-0.1mm}follows:} \!In \hspace*{-0.1mm}Section
\hspace*{-0.1mm}\ref{sec:preliminaries}, \hspace*{-0.1mm}we
\hspace*{-0.1mm}introduce~\hspace*{-0.1mm}the~\hspace*{-0.1mm}employed~\hspace*{-0.1mm}\mbox{notation},
define the relevant function spaces,~the~basic assumptions on the
nonlinear operator and its consequences, introduce discrete operators
and discuss their properties.  In Section \ref{sec:ldg}, we define our
numerical fluxes and derive the flux and
the \mbox{primal} \mbox{formulation} of our problem.  In Section
\ref{ssec:primalapriori}, we prove the existence of DG solutions
(cf.~Proposition~\ref{prop:exist}), the \hspace*{-0.15mm}stability
\hspace*{-0.15mm}of \hspace*{-0.15mm}the \hspace*{-0.15mm}method,
\hspace*{-0.15mm}i.e., \hspace*{-0.15mm}a \hspace*{-0.15mm}priori
\hspace*{-0.15mm}estimates
\hspace*{-0.15mm}(cf.~\hspace*{-0.15mm}Proposition~\ref{prop:stab},~\hspace*{-0.15mm}\mbox{Corollary}~\hspace*{-0.15mm}\ref{cor:stab}),
and~the~convergence of DG solutions (cf.~Theorem~\ref{thm:minty}). In
Section~\ref{ssec:primalerror},~we~\mbox{derive} error estimates for
our problem (cf.~Theorem~\ref{thm:error}, Corollary~\ref{cor:error}).
These~are~the~first convergence rates for an LDG method for systems
with balanced Orlicz-structure.~In~\mbox{Section}~\ref{sec:experiments},
we~confirm~our~theoretical findings via numerical experiments. For the
convenience of the reader, we collect in the~\mbox{Appendix}~\ref{sec:aux} known results in the~\mbox{DG Orlicz setting}, needed in
the paper,~and~prove~some~new results.

\section{Preliminaries}
\label{sec:preliminaries}

\subsection{Function spaces}
\label{ssec:functionspaces}

We use $c, C\!>\!0$ to denote generic constants,~that~may~change from line
to line, but are not depending on the crucial quantities. Moreover,
we~write ${f\sim g}$ if and only if there exists constants $c,C>0$ such
that $c\, f \le g\le C\, f$.

For $k\in \setN$ and $p\in [1,\infty]$, we will employ the customary Lebesgue spaces~$L^p(\Omega)$~and Sobolev
spaces $W^{k,p}(\Omega)$, where $\Omega \subseteq \setR^n$, $n\ge 2$, is a bounded,
polyhedral~domain~ha-ving a Lipschitz continuous boundary $\partial \Omega$,
which is  disjointly divided into a Dirichlet part $\Gamma_D\subseteq\partial\Omega$, where we assume that $\abs{\Gamma_D}> 0$,  and a Neumann part $\Gamma_N\subseteq\partial\Omega$, i.e., $\Gamma_D \cup \Gamma_N\!=\!\partial \Omega$ and 
$\Gamma_D \cap \Gamma_N\!=\!\emptyset$.  We denote by
$\smash{\norm{\,\cdot\,}_p}$, the norm in $L^p(\Omega)$~and~by~$\smash{\norm{\,\cdot\,}_{k,p}}$,
the norm in $W^{k,p}(\Omega)$. The space $\smash{W^{1,p}_{\Gamma_D}(\Omega)}$
is defined as those functions from $W^{1,p}(\Omega)$ whose trace vanishes on $\Gamma_D$. We equip
$\smash{W^{1,p}_{\Gamma_D}(\Omega)}$ 
with the gradient norm $\smash{\norm{\nabla\,\cdot\,}_p}$. 

For a normed vector space $X$, we denote its (topological) dual space by $X^*$.
We do not distinguish between function spaces for scalar, vector- or
tensor-valued~functions. However, we will denote vector-valued
functions by boldface letters~and~\mbox{tensor-valued} functions by capital
boldface letters. For $d\in \setN$, the standard scalar product between two vectors $\bu\hspace*{-0.1em}=\hspace*{-0.1em}(u_1,\dots,u_d)^\top\hspace*{-0.1em}, \bv\hspace*{-0.1em}=\hspace*{-0.1em}(v_1,\dots,v_d)^\top\hspace*{-0.1em}\in\hspace*{-0.1em} \setR^d$
is~denoted~by~${\bu \cdot \bv\hspace*{-0.1em}=\hspace*{-0.1em}\sum_{i=1}^d{u_iv_i}}$, and we use the notation $\vert \bu\vert \hspace*{-0.1em}=\hspace*{-0.1em}\sqrt{\bu\cdot\bu}$ for all $\bu  \hspace*{-0.1em}\in\hspace*{-0.1em} \setR^d$. For $d,n\hspace*{-0.1em}\in\hspace*{-0.1em} \setN$,~the~Frobenius~scalar product between two tensors
${\bP\!=\!(P_{ij})_{i=1,\dots,d,j=1,\dots,n},\! \bQ\!=\!(Q_{ij})_{i=1,\dots,d,j=1,\dots,n}\!\in\! \setR^{d\times n}}$ is~denoted by $\bP: \bQ\!=\!\sum_{i=1}^d{\sum_{j=1}^n{P_{ij}Q_{ij}}}$, and we write~$\abs{\bP}\!=\!\sqrt{\bP : \bP}$~for~all~${\bP\!\in\!\setR^{d\times n}}$. 
We denote by $\abs{M}$, the
$n$-- or $(n-1)$--dimensional~Lebesgue measure of a
(Lebesgue) measurable set $M\subseteq \mathbb{R}^n$, $n\in \mathbb{N}$.  The mean value of a locally integrable function
$f$ over a measurable set $M\!\subseteq\! \Omega$ is denoted by
${\mean{f}_M\!\vcentcolon=\!\smash{\dashint_M f \,\textup{d}x} \!\vcentcolon=\!\smash{\frac 1 {|M|}\int_M f \,\textup{d}x}}$.~\mbox{Moreover}, we use the notation $\hskp{f}{g}\vcentcolon=\int_\Omega f g\,\textup{d}x$,
whenever the right-hand~side~is~\mbox{well-defined}.

We will also use Orlicz and Sobolev--Orlicz spaces
(cf.~\cite{ren-rao}).~A~real~convex~\mbox{function}
$\psi : \setR^{\geq 0} \to \setR^{\geq 0}$ is said to be an
\textbf{N-function}, if $\psi(0)=0$, $\psi(t)>0$~for~all~${t>0}$,
$\lim_{t\rightarrow0} \psi(t)/t=0$, and
$\lim_{t\rightarrow\infty} \psi(t)/t=\infty$. We call $\psi$ a \textbf{regular
  N-function}, if it, in addition,
belongs to $C^1(\setR^{\ge 0})\cap C^2(\setR^{> 0})$ and satisfies
${\psi''(t)>0}$~for~all~${t>0}$. For a regular N-function, we have $\psi (0)=\psi'(0)=0$,
$\psi'$ is increasing and $\lim _{t\to \infty} \psi'(t)=\infty$.
We~define the (convex) \textbf{conjugate N-function}~$\psi^*$ by
${\psi^*(t)\hspace*{-0.1em}\vcentcolon=\hspace*{-0.1em} \sup_{s \geq 0} (st
  \hspace*{-0.1em}- \hspace*{-0.1em}\psi(s))}$~for all
${t \!\geq \!0}$, which~satisfies~$(\psi^*)'\! =\!
(\psi')^{-1}\!\!$. A
given N-function $\psi$~satisfies~the~\mbox{$\Delta_2$-condition}
(in short, $\phi \hspace*{-0.1em}\in\hspace*{-0.1em}\Delta_2$), if there exists $K\hspace*{-0.1em}>\hspace*{-0.1em} 2$ such that for
all $t \hspace*{-0.1em}\ge\hspace*{-0.1em} 0$,~it~holds~${\psi(2\,t) \hspace*{-0.1em}\leq\hspace*{-0.1em} K\, \psi(t)}$. We denote the
smallest such constant by $\Delta_2(\psi)\hspace*{-0.1em}>\hspace*{-0.1em}0$.  If one assumes that
both~$\psi$~and~$\psi^*$ satisfy the $\Delta_2$-condition, there holds
\begin{align} 
  \label{eq:phi*phi'}
  \smash{\psi^*\circ\psi' \sim \psi\,.}
\end{align}
We need the following refined version of the $\varepsilon$-Young inequality: for every
${\varepsilon\!>\! 0}$,~there is a $c_\epsilon>0 $, depending only on
$\Delta_2(\psi),\Delta_2( \psi ^*)<\infty$, such that~for~all~${s,t\geq0}$,~it~holds
\begin{align}
  \label{ineq:young}
  \begin{split}
    t\,s&\leq \epsilon \, \psi(t)+ c_\epsilon \,\psi^*(s)\,,
    \\
    t\, \psi'(s) + \psi'(t)\, s &\le \epsilon \, \psi(t)+ c_\epsilon
    \,\psi(s)\,.
  \end{split}
\end{align}

We denote by $L^\psi(\Omega)$ and $W^{1,\psi}(\Omega)$, the classical
Orlicz and Sobolev--Orlicz~spaces, i.e., we have that $f \hspace*{-0.1em}\in\hspace*{-0.1em} L^\psi(\Omega)$ if the
\textbf{modular}
$\rho_\psi(f)\hspace*{-0.1em}=\hspace*{-0.1em}\rho_{\psi,\Omega}(f)\hspace*{-0.1em}\vcentcolon=\hspace*{-0.1em}\int_\Omega \psi(\abs{f})\,\textup{d}x $~is~finite and ${f \in W^{1,\psi}(\Omega)}$ if
$f, \nabla f \in L^\psi(\Omega)$.  Equipped with the induced
Luxembourg~norm, i.e.,
$\smash{\norm {f}_{\psi}}\vcentcolon= \smash{\inf \set{\lambda >0\mid \int_\Omega
  \psi(\abs{f}/\lambda)\,\textup{d}x \le 1}}$, the space $L^\psi(\Omega)$~forms~a~Banach~space. The same holds for the space
$W^{1,\psi}(\Omega)$ if it is equipped with the norm
$\smash{\norm {\,\cdot\, }_{\psi} +\norm {\nabla\, \cdot\,}_{\psi}} $.  Note
that the dual space $(L^\psi(\Omega))^*$ of  $L^\psi(\Omega)$ can be identified with~the~space~$L^{\psi^*}(\Omega)$. The space $\smash{W^{1,\psi}_{\Gamma_D}(\Omega)}$ is
defined as those functions from $W^{1,\psi}(\Omega)$ whose trace~vanishes on $\Gamma_D$. We equip
$\smash{W^{1,\psi}_{\Gamma_D}(\Omega)}$ with the gradient norm
$\smash{\norm{\nabla\, \cdot\,}_\psi}$.

\subsection{Basic properties of the nonlinear operator} 
\label{sec:stress_tensor}
In the whole paper, we always assume that the nonlinear operator
$\AAA:\setR^{d\times n}\to \setR^{d\times n}$ has $\phi$-structure,
which will be defined now.  A detailed discussion and thorough proofs
can be found~in~\cite{die-ett,dr-nafsa,br-multiple-approx}. A regular
N-function $\psi$ is called \textbf{balanced}, if there exist
constants~${\gamma_1\hspace*{-0.1em}\in\hspace*{-0.1em} (0,1]}$~and~${\gamma_2\hspace*{-0.1em} \ge \hspace*{-0.1em} 1}$ such that for all
$t> 0$, there holds
\begin{align}
  \label{eq:equi1}
  \smash{\gamma_1\,\psi'(t)\le t\,\psi''(t)\le \gamma_2\,\psi'(t)}
  \,.
\end{align}
The constants $\gamma_1$ and $\gamma_2$ are called
\textbf{characteristics} of the balanced N-function $\psi$, and
will be denoted as $(\gamma_1,\gamma_2)$. The basic properties of
balanced N-functions are collected in the following lemma
(cf.~\cite{die-kacanow} for related results of a slightly different approach). 
\begin{lemma}
  \label{lem:bal}
  Let $\psi$ be a balanced N-function with characteristics
  $(\gamma_1,\gamma_2)$.  Then, the following statements apply:
  \begin{itemize}[noitemsep,topsep=1pt,leftmargin=\widthof{(iiiv)},font=\normalfont\upshape,labelwidth=\widthof{(iii)}] 
  \item [(i)] The \hspace*{-0.1mm}conjugate \hspace*{-0.1mm}N-function
  $\psi^\ast\!$ \hspace*{-0.1mm}is \hspace*{-0.1mm}a
  \hspace*{-0.1mm}balanced \hspace*{-0.1mm}N-function
  \hspace*{-0.1mm}with~\hspace*{-0.1mm}characteristics~\hspace*{-0.1mm}$(\frac{1}{\gamma_1},\!\frac{1}{\gamma_2})$. 
\item [(ii)] The N-functions $\psi$ and $ \psi^\ast$
  satisfy the $\Delta_2$-condition, and the $\Delta_2$-constants of
  $\psi$ and $\psi^*$ possess an upper bound depending only on the characteristics~of~$\psi$.
\item [(iii)] Uniformly with respect to $t>0$, we have that $\psi(t)\sim\psi'(t)\,t \sim  \psi''(t)\,t^{2}$,
  with constants of equivalence depending only on the characteristics
  of $\psi$.
\end{itemize}
\end{lemma}
\begin{proof}
  The assertion (i) is proved in \cite[Lemma 6.4]{dr-nafsa}.  The
  assertion (ii) is proved~in~\cite{BDK12} (cf.~\cite[Lemma
  2.10]{br-multiple-approx}). Assertion (iii) follows from (ii), since
  for N-functions satisfying the $\Delta_2$-condition, uniformly with respect to $t>0$, there holds
  \begin{align}
    \label{eq:equi2}
    \smash{\psi'(t)\, t \sim \psi(t)\,,}
  \end{align}
  and the fact that $\psi$ is balanced.
\end{proof}
\begin{remark}
  {\rm
    It is well-known  that the
   N-function $\psi$ for which $\psi'(t)= (\delta +t)^{p-2}t$ for all $t\ge 0$,
    $p \in (1,\infty)$, $\delta\ge 0$, is balanced (cf.~\cite{dr-nafsa,dr-interpol}). Other examples are
    given via
    $\psi'(t)\!=\! (
    t^\alpha(\delta +t)^{1-\alpha} )^{p-2}t$ for all $t\!\ge\! 0$,~or~${ \psi'(t)\!=\! (t\!+\!\delta)^ {p-2}t (\ln
    (1\!+\!\delta+ \! t))^\alpha}$~for~all~${t\!\ge\! 0}$, where ${\alpha\ge 1}$, $p\! \in\! (1,\infty)$,~${\delta\!\ge\! 0}$.
  }%
\end{remark}
\begin{assumption}[Nonlinear operator]\label{ass:1}
  We assume that the nonlinear operator \linebreak
  ${\AAA \colon \setR^{d \times n} \to \setR^{d \times n}}$ belongs to
  $C^0(\setR^{d \times n},\setR^{d \times n} )\cap C^1(\setR^{d \times
    n}\setminus \{\bfzero\}, \setR^{d \times n} ) $ and satisfies
  $\AAA (\mathbf 0)=\mathbf 0$. 
  Moreover, we assume that the operator $\AAA $ has
  \textbf{$\phi$-structure},~i.e., there exist a regular N-function
  $\varphi$ and constants $\gamma_3 \in (0,1]$, $\gamma_4 >1$ such
  that 
   \begin{subequations}
     \label{eq:ass_S}
     \begin{align}
       \sum\nolimits_{j,l=1}^n \sum\nolimits_{i,k=1}^d \partial_{kl}
       \mathcal{A}_{ij} (\bP) Q_{ij}Q_{kl} &\ge \gamma_3 \,\frac {\varphi'(|\bP|)}{|\bP|} |\bQ |^2\,,\label{1.4b}
       \\
       \big |\partial_{kl} \mathcal{A}_{ij}({\bP})\big | &\le \gamma_4 \,\frac {\varphi'(|\bP|)}{|\bP|}\,,\label{1.5b}
     \end{align}
   \end{subequations}
   are satisfied for all $\bP,\bQ \in \setR^{d\times n} $ with
   $\bP \neq \bfzero$ and all $i,k=1,\ldots, d$, $j,l=1,\ldots, n$.
   The constants $\gamma_3$, $\gamma_4$, and $\Delta_2(\varphi)$ are
   called the \textbf{characteristics} of $\AAA$ and will be denoted by
   $(\gamma_3,\gamma_4, \Delta_2(\varphi))$.
\end{assumption}

Closely related to the nonlinear operator
$\AAA\colon\setR^{d \times n} \to \setR^{d \times n}$ with
$\phi$-structure are the functions
$\bF,\bF^*\colon\setR^{d \times n} \to \setR^{d \times n}$, for every
$\bP\in \setR^{d\times n} $ defined via
\begin{align}
  \begin{aligned}
    \bF(\bP)
    \vcentcolon=\sqrt{\frac{\varphi'(\abs{\bfP})}{\abs{\bP}}} \,\bP\,,\qquad
    \bF^*(\bP)
    \vcentcolon= \sqrt{\frac{(\varphi^*)'(\abs{\bfP})}{\abs{\bP}}}\,\bP  \,.
  \end{aligned}  \label{eq:def_F}
\end{align}
\begin{remark}\label{rem:pot}
  {\rm
    Note that $\bF$ and $\bF^*$ are derived from the potentials
    \begin{align*}
   \kappa(t)\vcentcolon=\int_0^t \sqrt{\phi'(s)\,s}\, \mathrm{d}s \quad
      \textrm{ and } \quad \kappa^*(t)\vcentcolon=\int_0^t
      \sqrt{(\phi^*)'(s)\,s}\, \mathrm{d}s\quad\textup{ for all }t\ge 0\,,
    \end{align*}
    respectively. One easily sees that these potentials are balanced N-functions.
  }  
\end{remark}

Another important tool are the \textbf{shifted N-functions}
(cf.~\cite{die-ett,dr-nafsa,r-cetraro}). For~a~given N-func\-tion
$\psi$, we define
the {\rm shifted N--functions}
${\psi_a:\setR^{\ge
    0}\to \setR^{\ge
    0}}$,~${a\ge 0}$,~via
\begin{align}
  \label{eq:phi_shifted}
  \psi_a(t)\vcentcolon= \int _0^t \psi_a'(s)\,\textup{d}s\qquad\text{with }\quad
  \psi'_a(t)\vcentcolon=\psi'(a+t)\frac {t}{a+t}\quad\textup{ for all }t\ge 0\,.
\end{align}

The following properties of shifted N-functions are proved in
\cite{die-ett,dr-nafsa}.  
\begin{lemma}\label{lem:shift}
  Let the N-functions $\psi, \psi^*$ satisfy the
  $\Delta_2$-condition.~Then,~it~holds:
  \begin{itemize}[noitemsep,topsep=1pt,leftmargin=\widthof{(iiv)},font=\normalfont\upshape,labelwidth=\widthof{(ii)}]
  \item [(i)] The family of shifted N-functions ${\psi_a:\setR^{\ge
    0}\to\setR^{\ge
    0}}$,~${a\ge 0}$, satisfy the
$\Delta_2$-con-dition uniformly with respect to $a\ge 0$ with
$\Delta_2(\psi_a)$ depending~only~on~$\Delta_2(\psi)$.
  \item [(ii)] The conjugate
  function satisfies $\smash{(\psi_a)^*(t)\sim  (\psi^*)_{\psi'(a)}(t)}$ uniformly with~respect~to ${t ,a\ge 0}$ 
  with constants depending only on
  $\Delta_2(\psi), \Delta_2(\psi^*)$.
  \end{itemize}
\end{lemma}
\begin{lemma}
  Let $\psi$ be an N--function satisfying the
  $\Delta_2$-condition. Then, there exists a constant $c>0$  such that for every
  ${\bP,\bQ\in \mathbb{R}^{d\times n}}$ and $t\ge 0$, we have that
  \begin{align}
    \smash{\psi_{\abs{\bP}}'(t) \le c\, \psi_{\abs{\bQ}}'(t) +
    c\, \psi_{\abs{\bfP}}'(\abs{\bfP-\bfQ})\,.}\label{eq:phi_prime_shift}
  \end{align}
  Moreover,
  $\psi_\abs{\bQ}'(\abs{\bP\hspace*{-0.1em}-\hspace*{-0.1em}\bQ})\hspace*{-0.1em}\sim \hspace*{-0.1em}\psi_\abs{\bP}'
  (\abs{\bP\hspace*{-0.1em}-\hspace*{-0.1em}\bQ} )$ holds  uniformly with~respect~to~${\bP,\bQ \hspace*{-0.1em}\in\hspace*{-0.1em} \setR^{d\times n}}$.      
\end{lemma}
\begin{lemma}[Change of shift]
  \label{lem:change2}
  Let $\psi$ be an N--function such that $\psi$ and $\psi^*$ satisfy
  the $\Delta_2$--condition.  Then, for every $\delta \in (0,1)$, there
  exists $c_\delta>0$ such that for every $\bP,\bQ \in \setR^{d\times n}$ and
  $t\ge 0$, we have that
  \begin{align}
    \begin{aligned}\label{eq:34c1}
      \psi_{\abs{\bP}}(t) &\le c_\delta\, \psi_{\abs{\bQ}}(t) +
      \delta\, \psi_{\abs{\bfP}}(\abs{\bfP-\bfQ})\,,\\
      \big (\psi_{\abs{\bP}}\big )^*(t) &\le c_\delta\, \big
      (\psi_{\abs{\bQ}}\big )^*(t) + \delta\,
      \psi_{\abs{\bfP}}(\abs{\bfP-\bfQ})\,.
    \end{aligned}
  \end{align}
   Moreover, $\psi_\abs{\bQ}(\abs{\bP\hspace*{-0.1em}-\hspace*{-0.1em}\bQ})\hspace*{-0.1em}\sim\hspace*{-0.1em}
   \psi_\abs{\bP}(\abs{\bP\hspace*{-0.1em}-\hspace*{-0.1em}\bQ})$ holds uniformly with~respect~to~${\bP,\bQ\hspace*{-0.1em} \in\hspace*{-0.1em} \setR^{d\times n}}$.      
\end{lemma}

The connection between
$\AAA,\bfF,\bfF^*\colon\hspace*{-0.1em}\setR^{d \times n}
\hspace*{-0.1em}\hspace*{-0.1em}\to\hspace*{-0.1em} \setR^{d \times
  n}$ and
$\phi_a,(\phi^*)_a\hspace*{-0.1em}:\hspace*{-0.1em}\setR^{\ge
  0}\hspace*{-0.1em}\to\hspace*{-0.1em} \setR^{\ge
  0}$,~${a\hspace*{-0.1em}\ge\hspace*{-0.1em} 0}$, is best explained
by the following proposition (cf.~\cite{die-ett,dr-nafsa}).
\begin{proposition}
  \label{lem:hammer}
  Let $\AAA$ satisfy Assumption~\ref{ass:1} for a balanced
  N-function ${\phi}$. Then, uniformly with respect to 
  $\bfP, \bfQ \in \setR^{d \times n}$, we have that
    \begin{align}
      \label{eq:hammera}
      \big(\AAA(\bfP) - \AAA(\bfQ)\big)
      :(\bfP-\bfQ ) &\sim  \smash{\abs{ \bfF(\bfP) - \bfF(\bfQ)}^2}
      \sim \phi_{\abs{\bfP}}(\abs{\bfP - \bfQ})\,,
      \\
      \smash{\abs{ \bfF^*(\bfP) - \bfF^*(\bfQ)}^2}
      \label{eq:hammerf}
      &\sim  \smash{\smash{\big(\phi^*\big)}_{\smash{\abs{\bfP}}}(\abs{\bfP - \bfQ})}\,,
      \\
       \label{eq:hammere}  
      \abs{\AAA(\bfP) - \AAA(\bfQ)} &\sim
       \smash{\phi'_{\abs{\bfP}}(\abs{\bfP -
      	\bfQ})}\,.\\[-2mm]
      \intertext{Moreover,  uniformly with respect to $\smash{\bfQ \in \setR^{d \times n}}$, we have that\vspace*{-2mm}} 
      \label{eq:hammerd}
      \AAA(\bfQ) \cdot \bfQ \sim  \smash{\abs{\bfF(\bfQ)}^2} &\sim
      \phi(\abs{\bfQ})\,,\\
       \label{eq:hammere2}  
      \abs{\AAA(\bfP)} &\sim  \smash{\phi'(\abs{\bfP})}\,.
    \end{align}
  The constants in \eqref{eq:hammera}--\eqref{eq:hammere2} depend only on the characteristics of ${\AAA}$ and $\phi$.
\end{proposition}

Combining \eqref{eq:phi*phi'} and \eqref{eq:hammerf}--\eqref{eq:hammerd}, we obtain, uniformly~with~respect~to~${\bfQ \in \Rdn}$,
\begin{align}
  \label{eq:hammerg}  
   \smash{\abs{\bF^*(\AAA(\bfQ))}^2 \sim
  \varphi^*(\abs{\AAA(\bfQ)}) \sim
  \abs{\bfF(\bQ)}^2\sim\varphi(\abs{\bfQ})\,.}
\end{align}
In the same way as in \cite[Lemma 2.8, Corollary 2.9, Lemma
2.10]{dkrt-ldg}, one~can~show~that
\begin{lemma}\label{lem:F-F^*-diff}
  Let $\AAA$ satisfy Assumption~\ref{ass:1} for a balanced N-function
  ${\phi}$. Then, uniformly with respect to $t \geq 0$ and
  $\bQ, \bP \in\Rdn $, we  have that
  \begin{align}
    \label{eq:F-F*1}
    (\phi^*)_{\abs{\AAA(\bfP)}}(t) &\sim (\phi_{\abs{\bfP}})^*(t)\,,
    \\
    \label{eq:F-F*2}
    (\vp^*) _{|\AAA ( \bfP)|} (|\AAA (\bQ) - \AAA ( \bfP) |)&\sim
    \vp_{|\bfP|} (|\bQ - \bfP|)\,,
    \\
    \label{eq:F-F*3}
     \smash{\abs{\bF^*(\AAA(\bfQ))-\bF^*(\AAA(\bfP))}^2}
    &\sim  \smash{\abs{\bfF(\bQ)-\bfF(\bP)}^2}\,
  \end{align}
  with constants depending only on the characteristics of ${\AAA}$ and $\phi$.
\end{lemma}

\begin{lemma}
  \label{cor:change3}
  Let $\AAA$ satisfy Assumption~\ref{ass:1} for a balanced N-function
  ${\phi}$. Then, uniformly with respect to  $t \geq 0$ and
  $\bQ, \bP \in\Rdn $, we  have that
  \begin{align*}
    \big(\phi_{\abs{\bfP}})(t) &\le c\, \big(\phi_{\abs{\bfQ}}\big)(t)
    + c\,\abs{\bfF(\bfP) - \bfF(\bfQ)}^2\,,
    \\
    \big(\phi_{\abs{\bfP}})^*(t) &\le c\,
    \big(\phi_{\abs{\bfQ}}\big)^*(t) + c\,\abs{\bfF(\bfP) -
      \bfF(\bfQ)}^2\,,
    \\
    \smash{ \smash{\big(\phi^*\big)}_{\smash{\abs{\AAA(\bfP)}}}(t)} &\le c\,
     \smash{\smash{\big(\phi^*\big)}_{\smash{\abs{\AAA(\bfQ)}}}(t) }+ c\,\abs{\bfF(\bfP) -
      \bfF(\bfQ)}^2\,,
    \\
    \smash{ \smash{\big(\phi^*\big)}_{\smash{\abs{\bfP}}}(t) }&\le c\,
    \smash{ \smash{\big(\phi^*\big)}_{\smash{\abs{\bfQ}}}(t) }+ c\,\abs{\bfF^*(\bfP) -
      \bfF^*(\bfQ)}^2\,
  \end{align*}
  with constants depending only on the characteristics of ${\AAA}$ and $\phi$.
\end{lemma}

\begin{remark}[Natural energy spaces]\label{rem:nfs}
  {\rm \!If $\AAA$ satisfies Assumption~\ref{ass:1}~for~a~\mbox{balanced} N-function
  ${\phi}$, we see from
    \eqref{eq:hammerd} and \eqref{eq:hammerg} that  $\bu\in
    W^{1,\phi}(\Omega)$, ${\bL=\nabla \bu\in L^{\phi}(\Omega)}$~and
    ${\bfA =\AAA(\bL)\in L^{\phi^*}(\Omega)}$.} 
\end{remark}

\begin{remark}[Natural distance]
  \label{rem:natural_dist}
  {\rm
    If $\AAA$ satisfies Assumption~\ref{ass:1} for a balanced
    N-function ${\phi}$, we see from the previous results that for all $\bfu, \bfw \in W^{1,\phi}(\Omega)$
    \begin{align*}
      \hskp{\AAA(\nabla \bfu) \!-\!
      \AAA(\nabla\bfw)}{\nabla\bfu \!-\! \nabla \bfw}
      &\sim
        \norm{\bfF(\nabla\bfu) \!-\! \bfF(\nabla\bfw)}_2^2 \,\sim
        \int_\Omega\!  \phi_{\abs{\nabla\bfu}}(\abs{\nabla\bfu \!-\!
        \nabla\bfw}) \,\mathrm{d}x\,,
    \end{align*}
    where the constants depend only on the characteristics of $\AAA$ and
    $\phi$. In the context of $p$-Laplace problems, the quantity $\bF$
    was introduced in \cite{acerbi-fusco}, while the last expression
    equals the quasi-norm introduced in~\cite{barliu}
    raised~to~the~power~${\rho = \max \set{p,2}}$. We refer to all
    three equivalent quantities as the \textbf{natural distance}. This
    name expresses the fact, that the  natural distance provides the
    appropriate error measure for problem~\eqref{eq:p-lap} rather than the $W^{1,p}(\Omega)$-semi-norm. 
  }
\end{remark}

\subsection{DG spaces, jumps and averages}\label{sec:dg-space}

Let $(\mathcal{T}_h)_{h>0}$ be a family of triangulations of our
domain~$\Omega$ consisting of \mbox{$n$-dimensional} simplices $K$.
Here, the parameter $h>0$, refers to the maximal mesh-size, i.e., if we set $h_K\vcentcolon=\textup{diam}(K)$ for all $K\in \mathcal{T}_h$, then $h\vcentcolon=\max_{K\in \mathcal{T}_h}{h_K}$.
For simplicity, we always assume, in the
paper, that  $h \le 1$. For a simplex $K\! \in\! \mathcal{T}_h$,
we denote by $\rho_K\!>\!0$, the supremum of diameters~of~inscribed~balls. We assume that there exists a constant $\omega_0>0$, independent
of $h>0$, such that ${h_K}{\rho_K^{-1}}\le
\omega_0$ for all $K \in \mathcal{T}_h$. The smallest such constant~is~called~the~chunkiness~of $(\mathcal{T}_h)_{h>0}$. Note that, in the following, all
constants may depend on the chunkiness~$\omega_0$, but are independent
of~$h>0$.  For $K \in \mathcal{T}_h$, let $S_K$ denote the neighborhood of~$K$,~i.e., the patch
$S_K$ is the union of all simplices of~$\mathcal{T}_h$ touching~$K$.
We assume further for our triangulation that the interior of
each~$S_K$ is connected. 
Under~these~assumptions,  $\abs{K} \sim
\abs {S_K}$ uniformly in  $K\in \mathcal{T}_h$ and $h>0$, and the
number of simplices in $S_K$ and  patches to which a simplex belongs
to are uniformly bounded with respect to $K \in \mathcal{T}_h$~and~$h>0$. We define the faces of
$\mathcal{T}_h$ as follows: an interior face of $\mathcal{T}_h$ is the
non-empty interior of $\partial K \cap \partial K'$, where $K, K'$
are two adjacent elements of $\mathcal{T}_h$. For the face $\gamma\vcentcolon=
\partial K \cap \partial K'$, we use the notation $S_\gamma\vcentcolon= K \cup
K'$. A boundary face of $\mathcal{T}_h$ is the non-empty interior of
$\partial K \cap \partial \Omega$, where $K$ is a boundary element of
$\mathcal{T}_h$.  For the face $\gamma\vcentcolon= \partial K \cap \partial
\Omega$, we use the notation $S_\gamma\vcentcolon= K $. By $\smash{\Gamma_h^{i}}$, $\Gamma_D$, and
$\Gamma_N$, we denote the~interior,~the~Dirichlet~and the Neumann faces,
resp., and put $\Gamma_h\vcentcolon= \smash{\Gamma_h^{i}}\cup \Gamma_D \cup \Gamma_N$.
We assume~that~each~${K \in \mathcal{T}_h}$ has at most~one~face~from~$\Gamma_D \cup \Gamma_N$.  We introduce the following scalar products~on~$\Gamma_h$
\begin{align*}
  \skp{f}{g}_{\Gamma_h} \vcentcolon= \sum_{\gamma \in \Gamma_h} {\langle f, g\rangle_\gamma},\quad\text{ where }\quad\langle f, g\rangle_\gamma\vcentcolon=\smash{\int_\gamma f g \,\textup{d}s}\quad\text{ for all }\gamma\in \Gamma_h,
\end{align*}
if all the integrals are well-defined. Similarly, we define
$\skp{\cdot}{\cdot}_{\Gamma_D}$, $\skp{\cdot}{\cdot}_{\Gamma_N}$~and~$\skp{\cdot}{\cdot}_{\smash{\Gamma_h^{i}}}$.  We extend the notation of
modulars to the sets $\smash{\Gamma_h^{i}}$, 
$\Gamma_D$, and ${\smash{\Gamma_h^{i}}\cup \Gamma_D}$, i.e., we
define ${\rho_{\psi,B}(f)\!\vcentcolon=\! \smash{\int _B
\psi(\abs{f})\,\textup{d}s}}$ for every $f \!\in\! \smash{L^\psi(B)}$, where $B\!=\! \smash{\Gamma_h^{i}}$ or
${B\!=\! \Gamma_D}$~or~${B\!= \!\smash{\Gamma_h^{i}}\cup \Gamma_D}$.

For $m \in \setN_0$ and $K\in \mathcal{T}_h$,
we denote by ${\mathcal P}_m(K)$, the space of
scalar, vector-valued or tensor-valued polynomials of degree at most $m$ on $K$.
Given a triangulation~of~$\Omega$ with the above properties, given an
N-function $\psi$, and given $k \in \setN_0$,
we define
\begin{align}
  \begin{split}
    U_h^k&\vcentcolon=\big\{\bfu_h\in L^1(\Omega)^d\,\fdg \bfu_h|_K\in \mathcal{P}_k(K)^d\text{ for all }K\in \mathcal{T}_h\big\}\,,\\
    X_h^k&\vcentcolon=\big\{\bfX_h\in L^1(\Omega)^{d\times n}\fdg \bfX_h|_K\in \mathcal{P}_k(K)^{d\times n}\text{ for all }K\in \mathcal{T}_h\big\}\,,\\
        W^{1,\psi}(\mathcal T_h)&\vcentcolon=\big\{\bw_h\in L^1(\Omega)\fdg \bw_h|_K\in W^{1,\psi}(K)\text{ for all }K\in \mathcal{T}_h\big\}\,.
  \end{split}\label{eq:2.19}
\end{align}
Note that  $W^{1,\psi}(\Omega)\!\subseteq\! \WDGpsi$ and
$\Vhk\!\subseteq \!\WDGpsi$. We denote by ${\PiDG\!:\!L^1(\Omega)^d\!\to\! U_h^k}$, the (local)
$L^2$-projection into $U_h^k$, which for every $\bfu \hspace*{-0.18em}\in\hspace*{-0.18em}
L^1(\Omega)$ and $\bz_h\hspace*{-0.18em}
\in\hspace*{-0.18em} \Vhk$ is~defined~via
\begin{align}
  \label{eq:PiDG}
  \bighskp{\PiDG \bfu}{\bz_h}=\hskp{\bu}{\bz_h}\,.
\end{align}
Analogously, we define the (local)
$L^2$--projection into $X_h^k$, i.e., ${\PiDG: L^1(\Omega)^{d\times n} \to \Xhk}$. 

For $\bw_h\in \WDGpsi$, we denote by $\nb_h \bw_h\in L^\psi(\Omega)$
the \textbf{local gradient}, defined via
$(\nb_h \bw_h)|_K\!\vcentcolon=\!\nb(\bw_h|_K)$ for~all~${K\!\in\!\mathcal{T}_h}$.  For
every $K\!\in \!\mathcal{T}_h$, ${\bw_h\!\in\! \WDGpsi}$~admits~an
interior~trace~${\text{tr}^K(\bw_h)\in L^{\psi}(\pa K)}$. For each
face $\gamma\in \Gamma_h$ of a~given~simplex~${ T\in \mathcal{T}_h}$, we
define this interior trace by
$\smash{\textup{tr}^K_\gamma(\bw_h)\in L^{\psi}(\gamma)}$. Then, for
$\bw_h\in \WDGpsi$ and interior faces $\gamma\in \Gamma_{I}$ shared by
adjacent elements $K^-_\gamma, K^+_\gamma\in \mathcal{T}_h$,
we~denote~by
\begin{align}
  \{\bw_h\}_\gamma&\vcentcolon=\smash{\frac{1}{2}}\big(\textup{tr}_\gamma^{K^+}(\bw_h)+
  \textup{tr}_\gamma^{K^-}(\bw_h)\big)\in
  L^{\psi}(\gamma)\,, \label{2.20}
  \\
  \llbracket\bw_h\otimes\bfn\rrbracket_\gamma
  &\vcentcolon=\textup{tr}_\gamma^{K^+}(\bw_h)\otimes\bfn^+_\gamma+
    \textup{tr}_\gamma^{K^-}(\bw_h)\otimes\bfn_\gamma^- 
    \in L^\psi(\gamma)\,,\label{eq:2.21} 
\end{align}
the \textbf{average} and the \textbf{normal jump}, resp., of $\bw_h$ on $\gamma$.
Moreover,~for~every~$\bfw_h\in \WDGpsi$ and boundary faces $\gamma\in\Gamma_D$, we define boundary averages  and boundary jumps, resp., via
       \begin{align}
  \{\bfw_h\}_\gamma&\vcentcolon=\textup{tr}^\Omega_\gamma(\bfw_h) \in L^\psi(\gamma)\,,\label{eq:2.23a} \\
  \llbracket \bfw_h\otimes\bfn\rrbracket_\gamma&\vcentcolon=
  \textup{tr}^\Omega_\gamma(\bfw_h)\otimes\bfn \in L^\psi(\gamma)\,,\label{eq:2.23} 
\end{align}
where $\bfn:\partial\Omega\to \mathbb{S}^{d-1}$ denotes the unit
normal vector field to $\Omega$ pointing outward.  Analogously, we
define $\{\bfX_h\}_\gamma$ and
$ \llbracket\bfX_h\bfn\rrbracket_\gamma $~for every $\bfX_h \in \Xhk$
and $\gamma\in \smash{\Gamma_h^{i}} \cup \Gamma_D$. We omit the index $\gamma$ if
there is no danger~of~confusion.
For every $k\in \mathbb{N}_0$, we define for a given face $\gamma\in \smash{\Gamma_h^{i}}\cup\Gamma_D$, the
\textbf{(local)~jump~\mbox{operator}}
${\boldsymbol{\mathcal{R}}_{h,\gamma}^k :\WDGpsi \to X_h^k}$ for every $\bw_h\in \smash{\WDGpsi}$ (using
Riesz representation)~via
\begin{align}
  \big(\boldsymbol{\mathcal{R}}_{h,\gamma}^k(\bw_h),\bfX_h\big)
  \vcentcolon=\big\langle \llbracket\bw_h\otimes\bfn\rrbracket_\gamma,\{\bfX_h\}_\gamma\big\rangle_\gamma
  \quad\text{ for all }\bfX_h\in X_h^k\,,\label{eq:2.25}
\end{align}
and the \textbf{(global) jump operator} via
\begin{align}\label{def:Rhk}
  \Rhk=\boldsymbol{\mathcal R}^k_{h, \Gamma_D}\vcentcolon=\sum_{\gamma\in \smash{\Gamma_h^{i}} \cup
  \Gamma_D}{\boldsymbol{\mathcal{R}}_{\gamma,h}^k}:\WDGpsi \to X_h^k\,,
\end{align}
which, by definition, for every $\bw_h\in \smash{\WDGpsi}$ and  $\bfX_h\in X_h^k$ satisfies
\begin{align}
  \big(\Rhk\bw_h,\bfX_h\big)=\big\langle
  \llbracket\bw_h\otimes\bfn\rrbracket,\{\bfX_h\}\big\rangle_{\smash{\Gamma_h^{i}}
  \cup \Gamma_D}\,.\label{eq:2.25.1}
\end{align}
Furthermore, for every $k\in \mathbb{N}_0$, we define the
\textbf{discrete gradient operator} or \textbf{DG gradient operator}
$  \Ghk \hspace*{-0.1em}=\hspace*{-0.1em}\boldsymbol{\mathcal G}^k_{h, \Gamma_D}\hspace*{-0.1em}:\hspace*{-0.1em}\WDGpsi\hspace*{-0.1em}\to\hspace*{-0.1em} L^\psi(\Omega)$,
for every $\bw_h\hspace*{-0.1em}\in\hspace*{-0.1em}
\smash{\WDGpsi}$~via\footnote{Note that we use for discrete gradients the notation from
  \cite{ern-book}, which is different from the one used in \cite{dkrt-ldg}.\vspace*{-20mm}}
\begin{align}
  \boldsymbol{\mathcal G}^k_{h}\bw_h=
  \boldsymbol{\mathcal G}^k_{h,\Gamma_D}\bw_h\vcentcolon=
  \nb_h\bw_h-\boldsymbol{\mathcal R}^k_{h,\Gamma_D}\bw_h
  \quad\text{ in }L^\psi(\Omega)\,.\label{eq:DGnablaR} 
\end{align}
In particular, for every $\bw_h\in \smash{\WDGpsi}$ and $\bfX_h\in X_h^k$, we have that 
\begin{align}
\big(\Ghk\bw_h,\bfX_h\big)=(\nb_h\bw_h,\bfX_h)
  -\big\langle \llbracket
  \bw_h\otimes\bfn_h\rrbracket,\{\bfX_h\}\big\rangle_{\smash{\Gamma_h^{i}} \cup\Gamma_D}
\,.  \label{eq:DGnablaR1}
\end{align}
Note that for every $\bfu\in W^{1,\psi}_{\Gamma_D}(\Omega)$, it holds $\Ghk \bfu=\nb\bfu
$ in $L^\psi(\Omega)$.

 We define the pseudo-modular
$\smash{m_{\psi,h}=m_{\psi,h,\Gamma_D}}$ and the modular
$\smash{M_{\psi,h}=M_{\psi,h,\Gamma_D}}$ on $\WDGpsi$ and $\WDGpsiD$~via
\begin{align}
  \begin{aligned}
    m_{\psi,h}(\bw_h)=m_{\psi,h,\Gamma_D}(\bw_h)&\vcentcolon= h\,\rho_{\psi,\smash{\Gamma_h^{i}}\cup \Gamma_D
    }\big(h^{-1}\jump{\bw_h\otimes \bn}\big)\,,
    \\
    M_{\psi,h}(\bw_h)=M_{\psi,h,\Gamma_D}(\bw_h)&\vcentcolon= \rho_{\psi,\Omega}(\nabla_h \bw_h) +
    m_{\psi,h,\Gamma_D}(\bw_h)\,.
  \end{aligned}\label{def:mh}
\end{align}
The induced Luxembourg norm of the modular $M_{\psi,h}$ is
denoted by $\smash{\norm{\cdot}_{M_{\psi,h}}}$.~Note~that for every $\bfu \!\in\! \smash{\Wz^{1,\psi}(\Omega)}$, it holds
$m_{\psi,h}(\bfu)=0$ and
$M_{\psi,h}(\bfu) = \rho_{\psi,\Omega}(\nabla \bfu)$,~so~$M_{\psi,h}$~forms an extension of the modular
$\rho_{\psi,\Omega}(\nabla \,\cdot\,)$ on
$\smash{W^{1,\psi}_{\Gamma_D}(\Omega)}$ to the \DG~setting.~In~most~cases, we will
omit the index $\Gamma_D$ in $\smash{\boldsymbol{\mathcal R}^k_{h, \Gamma_D}}$,
$\smash{\boldsymbol{\mathcal G}^k_{h, \Gamma_D}}$, $\smash{m^k_{\psi, h, \Gamma_D}}$,~and~$\smash{M^k_{\psi,h, \Gamma_D}}$,~and simply write $\smash{\boldsymbol{\mathcal R}^k_{h}}$,
$\smash{\boldsymbol{\mathcal G}^k_{h}}$, $\smash{m^k_{\psi, h}}$, and
$\smash{M^k_{\psi,h}}$, respectively.

\begin{remark}
  In the case $\psi=\phi$, due to \eqref{eq:hammerd}, for every $\bw_h\in \WDGphi$, it holds
  \begin{align*}
      m_{\phi,h}(\bw_h)&\sim h\, \bignorm{
        \bF\big(h^{-1}\jump{\bw_h\otimes \bn}\big)}_{2,\smash{\Gamma_h^{i}}\cup
        \Gamma_D}^2\,,
      \\
      M_{\phi,h}(\bw_h)&\sim \norm{\bF(\nabla_h\bw_h)}^2_{2,\Omega}+
      h\,\bignorm{\bF\big(h^{-1}\jump{\bw_h\otimes
          \bn}\big)}_{2,\smash{\Gamma_h^{i}}\cup \Gamma_D}^2\,.
  \end{align*}
\end{remark}

\section{Fluxes and LDG formulations}
\label{sec:ldg}

In order to obtain the LDG formulation of \eqref{eq:p-lap} for given
$k \in \setN$, we multiply the equations in \eqref{eq:dg-p-lap}$_1$ by
$\bfX_h \in \Xhk$, $\bfY_h\in \Xhk$ and $\bfz_h\in \Vhk$, resp., use~partial~integration, replace in the volume integrals the fields $\bu$,
$\bL$, $\bA$ and $\bfG$ by the discrete objects $\bu_h$, $\bL_h$, $\bA_h$ and $\PiDG \bfG$,
resp., and in the surface integrals $\bu$, $\bA$ and $\bfG$ by~the~numerical~fluxes $\smash{\widehat{\bfu}_h\vcentcolon=\widehat{\bfu}(\bu_h)}$, 
$\smash{\widehat{\bfA}_h\vcentcolon=\widehat{\bfA}(\bu_h,\bA_h,\bL_h)}$ and $\smash{\widehat{\bfG}_h\vcentcolon=\widehat{\bfG}(\PiDG\bfG)}$, and obtain
\begin{align}
    \int_K \bfL_h : \bfX_h\,\textup{d}x &= -\int_K \bfu_h
    \cdot\divergence \bfX_h\,\textup{d}x + \int_{\partial K} \widehat{\bfu}_h
    \cdot (\bfX_h \bfn)\,\textup{d}s\,,\notag
    \\
    \int_K \bfA_h : \bfY_h \,\textup{d}x &= \int_K \AAA(\bfL_h) : \bfY_h\,\textup{d}x\,,\label{eq:flux-K}
    \\
    \int_K \bfA_h : \nabla_h \bfz_h\,\textup{d}x &= \int_K \bfg \cdot
    \bfz_h +\PiDG\bfG :\nabla _h \bfz_h\,\textup{d}x +\int_{\partial K} \bfz_h \cdot
    \big (\widehat{\bfA}_h \bfn-\widehat{\bfG}_h \bfn\big)
    \,\textup{d}s \,.\notag
\end{align}
For given boundary data $\bu_D\!:\!\Gamma_D\!\to\! \setR^d$ and
$\ba_N\!:\!\Gamma_N\!\to\! \setR^d$,
we~denote~by~${\bu_D^*\! \in\! W^{1,\phi}(\Omega)}$ an extension of
$\bu_D$, which exists if $\bfu_D$ belongs to the trace space of
$ W^{\smash{1,\phi}}(\Omega)$, that is characterized in
\cite{Pal79}. \!For~the~\mbox{nonlinear}~\mbox{operator}
$\AAA\!:\!\setR^{d\times n}\!\to\! \setR^{d\times n}$~\mbox{having}~\mbox{$\phi$-structure}, we define for every $a \ge 0$ and
$\bfP \in \setR^{d\times n}$
\begin{align}
  \label{eq:flux}
  \AAA_a(\bfP) \vcentcolon=  \frac{\phi_a'(\abs{\bfP})}{\abs{\bfP}}\, \bfP\quad\text{ in }\setR^{d\times n}\,,
\end{align}
where $\phi _a$, $a\ge 0$, is the
shifted N-function of the balanced N-function
$\phi$.  Using these notions, the
numerical~fluxes~are~defined~via\footnote{Due to
  $\smash{\jump{\bfu_D^*\otimes \bfn}=\bfzero}$ on $\smash{\smash{\Gamma_h^{i}}}$, the flux
  $\smash{\flux{\bA}}$ depends only on $\smash{\bu_D:\Gamma_D\to \mathbb{R}^d}$. We have chosen to formulate
  the flux in this form for a more compact notation.\vspace*{-0.8cm}}
\begin{align}
  \label{def:flux-u}
  \flux{\bfu}(\bfu_h) &\vcentcolon= 
  \begin{cases}
    \avg{\bfu_h} &\quad\text{on $\smash{\Gamma_h^{i}}$}\,,
  \\
    \bfu_D^* &\quad\text{on $\Gamma_D$}\,,
    \\
    \bfu_h &\quad\text{on $\Gamma_N$}\,,
  \end{cases}
\end{align}
\begin{align}\label{def:flux-A}
  \hspace*{-3mm}\flux{\bfA}(\bu_h, \bfA_h,\bL_h) &\hspace*{-0.1em}\vcentcolon=\hspace*{-0.1em}
  \begin{cases}
    \avg{\bfA_h} \hspace*{-0.1em}- \hspace*{-0.1em}\alpha\, \AAA_{\aaa}\big(h^{-1}\jump{(\bfu_h \hspace*{-0.1em}-\hspace*{-0.1em} \bfu_D^*)\hspace*{-0.1em}\otimes\hspace*{-0.1em} \bfn}\big)
    &  \hspace*{1.5mm}\text{on $\hspace*{-0.1em}\smash{\Gamma_h^{i}} \hspace*{-0.1em}\cup\hspace*{-0.1em} \Gamma_D$}\,,
    \\
    \bfa_N \hspace*{-0.1em}\otimes\hspace*{-0.1em} \bfn &  \hspace*{1.5mm}\text{on $\hspace*{-0.1em}\Gamma_N$}\,,
  \end{cases}  \hspace*{-3mm}
\end{align}
and
\begin{align}
  \label{def:flux-F}
  \flux{\bfG}(\PiDG\bfG) &\vcentcolon= 
  \begin{cases}
    \bigavg{\PiDG \bfG} &\quad\text{on $\smash{\Gamma_h^{i}} \cup \Gamma_D$}\,,
    \\
    \bfzero &\quad\text{on $\Gamma_N$}\,,
  \end{cases}
\end{align}
resp., where $\alpha>0$ is some constant.  Note that we actually would
like to use the~shift $\abs{\nabla\bfu}$ in the flux
\eqref{def:flux-A}, which apparently is not possible since
$\nabla\bfu$ is not~known~a~priori. Since the distance of the discrete
shift $\aaa$ to $\abs{\nabla\bfu}$ 
is controlled 
(cf.~Lemma~\ref{lem:e5}), we resort to the discrete shift $\aaa$.

The fluxes are conservative
since~they~are~single-valued. The fluxes $\smash{\flux{\bu}}$ and $\smash{\flux{\bA}}$
are consistent, since $\smash{\flux{\bfu}(\bu) \!= \!\bfu}$,
$\smash{\flux{\bfA}(\bu, \bA,\bL)\! =\! \bfA}$ for regular functions $\bfu$ and $\bfA$ satisfying ${\bu=\bu_D^*}$ on $\Gamma_D$ and $\AAA(\nabla\bu) \bfn = \bfa_N$ on $\Gamma_N$.

Note that for the flux $\widehat{\bfu}(\bfu_h)$ in~\eqref{def:flux-u}, we
have that $\bfL_h = \Ghk \bfu_h$ if
${\bfu_D=\bfzero}$~(cf.~\eqref{eq:Lh}). Thus, this choice of the flux
is the natural DG equivalent of $\bfL = \nabla \bfu$. Moreover,
\eqref{eq:Lh} also implies that the flux
$\smash{\flux{\bfA}}(\bu_h, \bfA_h,\bL_h) $ is actually only depending
on $\bA_h$~and~$\bu_h$. It is a natural generalization of the
corresponding fluxes for the Laplace
problem~(cf.~\cite{arnold-brezzi}) and the $p$-Laplace problem
(cf.~\cite{ern-p-laplace}) taking into account the
$\phi$-structure~of~\eqref{eq:p-lap}. Note that it
coincides for $\phi(t)=t^2$ with a standard flux~for~the~Laplace
problem. The usage of the operator $\AAA_{\smash{\aaa}}$ in
\eqref{def:flux-A} instead of  $\AAA_0$ in \mbox{\cite[(2.33)]{dkrt-ldg}} is the key to better approximation properties~(cf.~Theorem~\ref{thm:error},~Corollary~\ref{cor:error},~Remark~\ref{rem:dis}). The flux
$\smash{\flux{\bfG}}$ is designed such that it yields the weak
form in~\eqref{eq:DG}.

Proceeding as in \cite{dkrt-ldg} and, in addition, using that $\PiDG$
is self-adjoint,~we~arrive~at~the \textbf{flux formulation}
of~\eqref{eq:p-lap}: For given data
$\bu_D\hspace*{-0.1em}\in\hspace*{-0.1em}\trace W^{1,\phi}(\Omega)$,
${\bfg\hspace*{-0.1em}\in\hspace*{-0.1em}
  \smash{L^{\phi^*}(\Omega)}}$,~${\bfG\hspace*{-0.1em}\in\hspace*{-0.1em}
  \smash{L^{\phi^*}(\Omega)}}$ and $\ba_N\in \smash{L^{\phi^*}(\Gamma_N)}$
find $(\bfu_h, \bL_h,\bA_h)^\top \in \Vhk \times \Xhk \times \Xhk$ such
that for all $(\bfX_h,\bfY_h,\bfz_h)^\top$ $ \in \Xhk \times \Xhk \times \Vhk$, it holds
\begin{align}\label{eq:DG}
  \begin{aligned}
    \hskp{\bfL_h}{\bfX_h} &= \bighskp{\nablaDG \bfu_h + \Rhk\bfu_D^*}{
      \bfX_h}\,,
    \\[-0.5mm]
    \hskp{\bfA_h}{\bfY_h} &= \hskp{\AAA(\bfL_h)}{
      \bfY_h}\,, 
    \\[-0.5mm]
    \bighskp{\bfA_h}{\nablaDG \bfz_h} &=
    \hskp{\bfg}{\bfz_h}+\bighskp{\bfG}{\Ghk\bfz_h}+
    \skp{\bfa_N}{\bfz_h}_{\Gamma_N}
    \\[-0.5mm]
    &\quad - \alpha \bigskp{\AAA_{\aaa}(h^{-1} \jump{(\bfu_h -\bu_D^*)\otimes
        \bfn})}{ \jump{\bfz_h \otimes \bfn}}_{\smash{\Gamma_h^{i}}\cup
      \Gamma_D}\,.
  \end{aligned}
\end{align}

Next, we want to eliminate in the system~\eqref{eq:DG} the
variables~${\bfL_h\in \Xhk}$~and~${\bfA_h\in \Xhk}$ to derive a system only expressed in
terms of the single variable $\bfu_h\in \Vhk$.~To~this~end, first note that from~\eqref{eq:DG}, it 
follows that
\begin{alignat}{2}
  \label{eq:Lh}
  \bfL_h &= \nablaDG \bfu_h + \Rhk \bfu_D^*&&\quad\text{ in }\Xhk\,,
  \\
  \label{eq:Ah}
  \bfA_h&= \PiDG \AAA(\bfL_h)&&\quad\text{ in }\Xhk\,.
\end{alignat}
If we insert \eqref{eq:Lh} and \eqref{eq:Ah} into \eqref{eq:DG}$_3$,
we find that for every $\bfz_h \in \Vhk$, there holds
\begin{align}
  \label{eq:primal}
  \begin{aligned}
    &\bighskp{\AAA(\nablaDG \bfu_h + \Rhk\bfu_D^*)}{\nablaDG
      \bfz_h}
    \\
    &=     \hskp{\bfg}{\bfz_h}+\bighskp{\bfG}{\Ghk\bfz_h}+
    \skp{\bfa_N}{\bfz_h}_{\Gamma_N}
   \\
    &\quad - \alpha \bigskp{\AAA_{\aaal}(h^{-1} \jump{(\bfu_h -\bu_D^*)\otimes
        \bfn})}{ \jump{\bfz_h \otimes \bfn}}_{\smash{\Gamma_h^{i}}\cup
      \Gamma_D}\,.
  \end{aligned}
\end{align}
This is the \textbf{primal formulation} of
our system. It can be equivalently formulated~as
\begin{align}
  \label{eq:op}
  \hskp{\bB_h\bu_h}{\bz_h}= \hskp{\bfb_h}{\bz_h} \quad \textup{ for all } \bz_h
  \in \Vhk\,,
\end{align}
where the nonlinear operator $\bB_h\colon \Vhk \to (\Vhk)^*$ and the linear functional
$\bb_h\colon \Vhk \to \setR$  for every $\bu_h, \bz_h \in \Vhk$ are defined via
\begin{align}
  \label{eq:bB}
 \hskp{\bB_h\bu_h}{\bz_h}&\vcentcolon= \bighskp{\AAA(\nablaDG \bfu_h + \Rhk\bfu_D^*)}{\nablaDG
  \bfz_h}
  \\
  &\quad + \alpha \bigskp{\AAA_{\aaal}  (h^{-1} \jump{(\bfu_h -\bu_D^*)\otimes
  \bfn})}{ \jump{\bfz_h \otimes \bfn}}_{\smash{\Gamma_h^{i}}\cup
    \Gamma_D}\,,\notag
  \\
  \hskp{\bfb_h}{\bz_h} &\vcentcolon=\hskp{\bfg}{\bfz_h}+\bighskp{\bfG}{\Ghk\bfz_h}+
    \skp{\bfa_N}{\bfz_h}_{\Gamma_N}\,.\label{eq:bb}
\end{align}

\section{Well-posedness, stability and weak convergence}
\label{ssec:primalapriori}

In this section, for $k\!\in\! \mathbb{N}$ and $\alpha\! >\! 0$, we establish the existence of a solution~of~\eqref{eq:DG}, \eqref{eq:primal} and~\eqref{eq:op}~(i.e., 
well-posedness), resp., their stability (i.e., an a priori~estimate), and the weak convergence of the discrete solutions to a solution of problem \eqref{eq:p-lap}. Our approach extends the
treatment in \cite{dkrt-ldg}. Although the existence~of~\mbox{discrete} solution resorts to standard tools, the rigorous argument is
involved due to the presence of the shift
$\aaal$ in the
stabilization term. Hence,~we~present~a~detailed treatment. We start showing several estimates needed for the existence of
discrete solutions and their stability.
\begin{lemma}\label{lem:coer}
  Let $\AAA$ satisfy Assumption~\ref{ass:1} for a balanced N-function
  ${\phi}$. 
  Then, for every $\bu_h \in \Vhk$, we have that
  \begin{align}
      &\bighskp{\bB_h\bu_h}{\bu_h-\PiDG\bu_D^*}\label{eq:coer}\\
      &\ge c\, \min\{1,\alpha\}\,M_{\phi,h}(\bu_h\hspace*{-0.1em}
      -\hspace*{-0.1em}\bu_D^*) \hspace*{-0.1em}+\hspace*{-0.1em} c\,\alpha\, m_{\phi_{\smash{\aaal}}} (\bu_h \hspace*{-0.1em}-\hspace*{-0.1em}\bu_D^*) \hspace*{-0.1em}-\hspace*{-0.1em} c_\alpha\, \rho_{\phi, \Omega}(\nabla
      \bu_D^*) \notag
      \\
      &\ge c\, \min\{1,\alpha\}\, M_{\phi,h}(\bu_h) - c_\alpha\, \rho_{\phi, \Omega}(\nabla
      \bu_D^*) - c\, \min\{1,\alpha\}\, \rho_{\phi, \Omega}\big( h^{-1} \bu_D^*\big)     \notag
\end{align}
with a constant $c>0$ depending only on the
characteristics of ${\AAA }$ and $\phi$, and the chunkiness
$\omega_0>0$, and a constant $c_\alpha>0$ additionally depending on $\max\{1,\alpha\}$.
\end{lemma}

\begin{proof}
  In view of \eqref{eq:Lh}, we write $\smash{\AAA_{\smash{\aaa}}}$ instead of
  $\smash{\AAA_{\smash{\aaal}}}$~to~shorten the notation. 
  Resorting to \eqref{eq:DGnablaR}, for every $\bu_h \in \Vhk$, we
  find that
  \begin{align}\label{eq:decomp}
    \hspace*{-2mm}\nablaDG\big( \bfu_h\!-\! \PiDG\bu_D^*\big)=  \big(\nablaDG \bfu_h \!+\! \Rhk \bfu_D^*\big)\! -\!
    \big(\Rhk \bfu_D^* \!-\!\Rhk \PiDG\bfu_D^*\big) \!-\!\nabla _h \PiDG\bu^*_D\;\,\text{ in }\Xhk\,.\hspace{-2mm}
  \end{align}
	Using \eqref{eq:decomp}, we immediately deduce that for every  $\bu_h \in \Vhk$, we have that
  \begin{align}
    &\hskp{\bB_h\bu_h}{\bu_h-\PiDG\bu_D^*}\label{eq:apri}
    \\
    &=\bighskp{\AAA(\nablaDG \bfu_h + \Rhk \bfu_D^*)}{\nablaDG \bfu_h +
      \Rhk \bfu_D^*}\notag 
    \\
    &\quad + \alpha \bigskp{\AAA_{\smash{\aaa}}
      (h^{-1} \jump{(\bfu_h -\bu_D^*)\otimes \bfn})}
      { \jump{(\bfu_h -\bu_D^*) \otimes \bfn}}_{\smash{\Gamma_h^{i}}\cup \Gamma_D}\notag
    \\
    &\quad - \bighskp{\AAA(\nablaDG \bfu_h + \Rhk \bfu_D^*)}{\nabla_h\PiDG \bfu_D^*
    } -\bighskp{\AAA(\nablaDG \bfu_h + \Rhk \bfu_D^*)}{\Rhk(\bu_D^*-\PiDG
      \bfu_D^*) }\notag
    \\
    &\quad -\alpha\, \bigskp{\AAA_{\aaa}(h^{-1} \jump{(\bfu_h -\bu_D^*)\otimes \bfn})}
      { \jump{(\PiDG \bfu_D^* -\bu_D^*) \otimes \bfn} }_{\smash{\Gamma_h^{i}}\cup \Gamma_D}
    \notag
    \\
    &=: I_1+\alpha\,I_2+I_3+I_4+ I_5 \,.\notag
  \end{align}
  Before we estimate the terms $I_i$, $i=3,\ldots,5$, we collect the
  information~we~obtain from the terms $I_1$ and $I_2$.  From
  Proposition~\ref{lem:hammer}, for every $\bu_h \in \Vhk$, it follows that
  \begin{align}\label{eq:a0}
    \begin{aligned}
      I_1 \sim \rho_{\phi,\Omega}\big({\nablaDG \bfu_h + \Rhk
        \bfu_D^*}\big)\,,\quad
      I_2 \sim m_{\phi_{\aaa},h } (\bfu_h-\bu_D^*) \,.
    \end{aligned}
  \end{align}
  Resorting, again, to \eqref{eq:DGnablaR}, for every $\bu_h \in \Vhk$, we observe that
  \begin{align}\label{eq:id}
  \nablaDG \bfu_h + \Rhk \bfu_D^* =  \nablaDG (\bfu_h - \bfu_D^*) +  
      \nabla \bu^*_D\quad\text{ in }L^\phi(\Omega)\,,
  \end{align}
  which together with the properties of $\phi$ implies for every $\bu_h \in \Vhk$
  that
  \begin{align}\label{eq:a1}
     \rho_{\phi,\Omega}\big (\nablaDG (\bfu_h -\bu_D^*)\big ) &\le c\, I_1 +  c\,
          \rho_{\phi,\Omega}(\nabla  \bfu_D^*)\,.
  \end{align}
  The shift change from Lemma \ref{lem:change2}, \eqref{eq:Lh},
  \eqref{eq:a0}, the trace inequality \eqref{eq:PiDGapproxmglobal1}, \eqref{eq:id}, the
  $L^\phi$-stability of $\PiDG$ in \eqref{eq:PiDGLpsistable}  with
  $k=0$, 
  and \eqref{eq:a1}, for every
  $\bu_h \in \Vhk$, yield
  \begin{align}
    \hspace*{-.5mm}  m_{\phi,h } (\bfu_h-\bu_D^*)
      &\le c\, m_{\phi_{\aaa},h } (\bfu_h-\bu_D^*) \!+\!c \,h\,
      \rho_{\phi,\smash{\Gamma_h^{i}}\cup\Gamma_D } \big  (\aaal \big )
        \notag
      \\[-1mm]
      &\le c\, I_2 +  c\, \rho_{\phi,\Omega } \big  (\Pia  (\nablaDG (\bfu_h - \bfu_D^*) +  
      \nabla  \bu^*_D ) \big ) \notag
      \\[-0.25mm]
      &\le c\, I_2 +  c\, \rho_{\phi,\Omega } \big(\nablaDG (\bfu_h - \bfu_D^*) +  
      \nabla \bu^*_D\big) \label{eq:a2}
      \\[-0.25mm]
      &\le c\, I_2 +  c\,     \rho_{\phi,\Omega } \big  (\nablaDG (\bfu_h - \bfu_D^*) \big ) +  c\,
      \rho_{\phi,\Omega }( \nabla  \bu^*_D) \notag
      \\[-0.25mm]
      &\le c\, I_2 +   c\, I_1 +  c\,
      \rho_{\phi,\Omega }( \nabla  \bu^*_D)\,. \notag
  \end{align}
  From \eqref{eq:a1}, \eqref{eq:a2} and~the equivalent expression for $M_{\phi,h}$ in Lemma~\ref{lem:equiW1psi},~it~follows~that
  \begin{align*}
    M_{\phi,h } (\bfu_h-\bu_D^*)
    &\le c\, I_1 + c\,I_2 + c\, \rho_{\phi,\Omega }( \nabla  \bu^*_D) \,,
  \end{align*}
  which implies that
  \begin{align}
    \begin{aligned}
      & \alpha\, m_{\phi_{\aaa},h } (\bfu_h-\bu_D^*)
      +\min\{1,\alpha\}\, M_{\phi,h } (\bfu_h-\bu_D^*)
      \\[-0.25mm]
      &\quad \le c\, I_1 +c\,\alpha\, I_2 + c\,\rho_{\phi,\Omega }( \nabla
      \bu^*_D)\,.
    \end{aligned}
        \label{eq:a3}
  \end{align}
  Now we can estimate the remaining terms. Using~\eqref{eq:hammere2}, the $\varepsilon$-Young~inequality~\eqref{ineq:young}, \eqref{eq:phi*phi'}, the $L^\phi$-gradient stability
  of~$\PiDG$ in \eqref{eq:PiDGLW1psistable}, and \eqref{eq:a1}, we have that 
  \begin{align}\label{eq:a4}
    \begin{aligned}
      \abs{I_3}&\le \vep \,\rho_{\phi^*,\Omega}\big(\phi'(\abs{\nablaDG \bfu_h +
        \Rhk \bfu_D^*})\big)  +c_\vep\,
      \rho_{\phi,\Omega}\big(\nabla_h\PiDG\bfu_D^*\big)
      \\[-0.25mm]
      &\le \vep \,\rho_{\phi,\Omega}\big({\nablaDG \bfu_h + \Rhk
        \bfu_D^*}\big)  +c_\vep \,\rho_{\phi,\Omega}\big(\nabla\bfu_D^*\big)
      \\[-0.25mm]
      &\le \vep \, c\, I_1 +c_\vep \,\rho_{\phi,\Omega}\big(\nabla\bfu_D^*\big)\,,
    \end{aligned}
  \end{align}
  and with the stability of $\Rhk$ in
  \eqref{eq:Rgammaest} and the approximation property~of~$\PiDG$~in~\eqref{eq:PiDGapproxmglobal}
  \begin{align}\label{eq:a5}
    \begin{aligned}
      \abs{I_4}&\le \vep\, \rho_{\phi^*,\Omega}\big(\phi'(\abs{\nablaDG
        \bfu_h + \Rhk \bfu_D^*})\big) +c_\vep\,
      \rho_{\phi,\Omega}\big(\Rhk(\bu_D^*-\PiDG \bfu_D^*)\big)
      \\[-0.25mm]
      &\le \vep\, \rho_{\phi,\Omega}\big({\nablaDG \bfu_h + \Rhk
        \bfu_D^*}\big) +c_\vep \,m_{\phi,h}\big(\bu_D^*-\PiDG \bfu_D^*\big)
      \\[-0.25mm]
      &\le \vep \,c\, I_1 +c_\vep \,\rho_{\phi,\Omega}(\nabla\bfu_D^*)\,.
    \end{aligned}
  \end{align}
  Using the definition of $\AAA_{\aaa}$ in \eqref{eq:flux}, the  $\varepsilon$-Young
  inequality \eqref{ineq:young}~with~$\phi_{\aaa}$, \eqref{eq:phi*phi'}, the
  shift change in Lemma \ref{lem:change2}, \eqref{eq:Lh}, the
  approximation property~of~$\PiDG$~in \eqref{eq:PiDGapproxmglobal},
  \hspace*{-0.1mm}and \hspace*{-0.1mm}proceeding \hspace*{-0.1mm}as \hspace*{-0.1mm}in \hspace*{-0.1mm}\eqref{eq:a2} \hspace*{-0.1mm}to \hspace*{-0.1mm}handle
  \hspace*{-0.1mm}$h\rho_{\phi,\smash{\Gamma_h^{i}}\cup\Gamma_D } \!(\avg{\abs{\Pia (\nablaDG \bfu_h
  \!+\!\Rhk\bfu_D^*)}})$,~\hspace*{-0.1mm}we~\hspace*{-0.1mm}get
  \begin{align}
    \abs{I_5}&\le \alpha\,h\, \bigabs{ \bigskp{\phi_{\aaa}'(\abs{h^{-1} \jump{(\bfu_h -\bu_D^*)\otimes
               \bfn}})}{h^{-1} \jump{(\bu_D^*-\PiDG \bu_D^* ) \otimes \bfn}}_{\smash{\Gamma_h^{i}}\cup
               \Gamma_D}}\notag
    \\[-0.25mm]
             &\le \vep\,h \, \alpha\,\rho_{\phi_{\aaa}^*,\smash{\Gamma_h^{i}}\cup
               \Gamma_D}\big(\phi_{\aaa}'(\abs{h^{-1} \jump{(\bfu_h -\bu_D^*)\otimes
               \bfn}})\big) \notag
    \\[-1mm]
    &\quad +c_\vep \, \alpha\,m_{\phi_{\aaa},h}\big (\bfu_D^*
               -\PiDG\bu_D^*\big) \label{eq:a6}
    \\[-0.75mm]
             &\le \vep\,c\,\alpha\,m_{\phi_{\aaa},h}(\bfu_h -\bu_D^*)+c_\vep \,c_\kappa\,
               \alpha\,m_{\phi,h}\big(\bfu_D^* -\PiDG\bu_D^*\big)\notag
    \\[-1mm]
    &\quad + \kappa\, c_\vep \,\alpha\,h\,
      \rho_{\phi,\smash{\Gamma_h^{i}}\cup\Gamma_D } \big (\aaal \big) \notag
      \\[-0.25mm]
      &\le \vep\,c\, \alpha\,I_2 
        + \kappa\, c_\vep \, c\,  \max\{1,\alpha \}\,I_1
        + c_\kappa\, c_\vep\, \max\{1,\alpha \}\,
      \rho_{\phi,\Omega }\big ( \nabla  \bu^*_D \big)\,. \notag
  \end{align}
  \enlargethispage{3mm}
  Choosing first $\vep>0$  and, then, $\kappa>0$
  sufficiently small, we conclude from
  \eqref{eq:apri}, \eqref{eq:a3}--\eqref{eq:a6} that the first
  inequality~in~\eqref{eq:coer} applies. Then, the second one follows from the
  first one and 
  \begin{align}
    \label{eq:a7}
    \begin{aligned}
      M_{\phi,h}(\bu_h -\bu_D^*)&\ge c\, M_{\phi,h}(\bu_h) -c\,
      M_{\phi,h}(\bu_D^*)
      \\[-0.25mm]
      &\ge c\, M_{\phi,h}(\bu_h) -c\,\rho_{\phi,\Omega}(\nabla
      \bu_D^*) - c\,\rho_{\phi,\Omega}\big(h^{-1} \bu_D^*\big)\,,
    \end{aligned}
  \end{align}
  where we used the definition of $M_{\phi,h}$ in
  \eqref{def:mh} and the trace inequality \eqref{eq:PiDGapproxmglobal2}.\vspace*{-2.5mm}
\end{proof}
\begin{lemma}\label{lem:coer1}
  Let $\AAA$ satisfy Assumption~\ref{ass:1} for a balanced N-function
  ${\phi}$.  Then, for every $\vep >0$, there exists~a~\mbox{constant}
  ${c_\vep>0}$ such that for every $\bu_h \in \Vhk$,~we~have~that
  \begin{align*}
      \abs{\hskp{\bB_h\bu_h}{\bu_h}} &\ge c\,
      M_{\phi,h}(\bu_h) - c\,\big (\rho_{\phi, \Omega}(\nabla
      \bu_D^*) + \rho_{\phi, \Omega}(
      h^{-1} \bu_D^*)\big )
\end{align*}
with a constant $c>0$ depending only on $\alpha>0$, the
characteristics of $\AAA$ and $\phi$,
and the chunkiness $\omega_0>0$.
\end{lemma}
\begin{proof}
  In view of \eqref{eq:Lh}, we write $\smash{\AAA_{\smash{\aaa}}}$ instead of
  $\smash{\AAA_{\smash{\aaal}}}$. Then, for every $\bu_h \in \Vhk$, we have  that
  \begin{align}
    \hskp{\bB_h\bu_h}{\bu_h}
    &=\bighskp{\AAA(\nablaDG \bfu_h + \Rhk \bfu_D^*)}{\nablaDG \bfu_h +
      \Rhk \bfu_D^*}\notag
    \\[-0.5mm]
    &\quad + \alpha \bigskp{\AAA_{\aaa}
      (h^{-1} \jump{(\bfu_h -\bu_D^*)\otimes \bfn})}
      { \jump{(\bfu_h -\bu_D^*) \otimes \bfn}}_{\smash{\Gamma_h^{i}}\cup \Gamma_D}\notag
    \\[-0.5mm]
    &\quad -\bighskp{\AAA(\nablaDG \bfu_h + \Rhk \bfu_D^*)}{\Rhk \bu_D^* }\label{eq:apri1}
    \\[-0.5mm]
    &\quad +\alpha\, \bigskp{\AAA_{\aaa}(h^{-1} \jump{(\bfu_h -\bu_D^*)\otimes \bfn})}
      { \jump{\bu_D^* \otimes \bfn} }_{\smash{\Gamma_h^{i}}\cup \Gamma_D}\notag
    \\[-0.5mm]
    &=: I_1+\alpha\,I_2+I_3+I_4 \,.\notag
  \end{align}
  Proceeding as in the proof of Lemma \ref{lem:coer}, for every
  $\bu_h \in \Vhk$, we find that (cf.~\eqref{eq:a3})
  \begin{align*}
    I_1 + \alpha\,I_2 &\ge c\, \smash{m_{\phi_{\aaa},h } (\bfu_h-\bu_D^*)} +c\, M_{\phi,h } (\bfu_h-\bu_D^*)
                        -  c\, \rho_{\phi,\Omega }( \nabla  \bu^*_D)
  \end{align*}
  with \hspace{-0.1mm}$c\!>\!0$ \hspace{-0.1mm}depending \hspace{-0.1mm}on \hspace{-0.1mm}$\alpha\!>\!0$. \hspace{-1mm}Using \hspace{-0.1mm}the \hspace{-0.1mm}$\varepsilon$-Young
  \hspace{-0.1mm}inequality \hspace{-0.1mm}\eqref{ineq:young}, \hspace{-0.1mm}\eqref{eq:phi*phi'}, \hspace{-0.1mm}the \hspace{-0.1mm}stability~\hspace{-0.1mm}of $\Rhk$ in
  \eqref{eq:Rgammaest}, \eqref{eq:id}, the trace
  inequality~\eqref{eq:PiDGapproxmglobal2},~\eqref{eq:equi},~for~all~${\bu_h\!
    \in\! \Vhk}$,~we~get~(cf.~\eqref{eq:a5})
  \begin{align*}
      \abs{I_3}&\le \vep \,\rho_{\phi,\Omega}\big({\nablaDG \bfu_h + \Rhk
        \bfu_D^*}\big) +c_\vep \,m_{\phi,h}(\bu_D^*)
      \\
      &\le \vep \, c\,\rho_{\phi,\Omega}\big(\nablaDG (\bfu_h
      -\bu_D^*)\big ) +c_\vep \,\rho_{\phi,\Omega}(\nabla\bfu_D^*)
      +c_\vep \,\rho_{\phi,\Omega}\big(h^{-1}\bfu_D^*\big)
      \\
      &\le \vep \, c\,M_{\phi,h}(\bfu_h
      -\bu_D^* ) +c_\vep \,\rho_{\phi,\Omega}(\nabla\bfu_D^*)
      +c_\vep \,\rho_{\phi,\Omega}\big(h^{-1}\,\bfu_D^*\big)\,.
  \end{align*}
  Proceeding as in the estimate \eqref{eq:a6} and, in addition, using
  the equivalent expression for  $M_{\phi,h}$ in
  \eqref{eq:equi}, for~every~${\bu_h \in \Vhk}$, we obtain 
  \begin{align*}
    \abs{I_4}&\le \alpha\,h\, \bigabs{ \bigskp{\phi_{\aaa}'(\abs{h^{-1} \jump{(\bfu_h -\bu_D^*)\otimes
               \bfn}})}{h^{-1} \jump{\bu_D^* \otimes \bfn}}_{\smash{\Gamma_h^{i}}\cup
               \Gamma_D}}\notag
    \\[-0.5mm]
             &\le \vep\,m_{\phi_{\aaa},h}(\bfu_h -\bu_D^*)+c_\kappa\,c_\vep \,m_{\phi,h}(\bfu_D^* )\notag
       \\[-0.5mm]&\quad+ \kappa\, c_\vep \,h\,
      \rho_{\phi,\smash{\Gamma_h^{i}}\cup\Gamma_D } \big (\aaal \big ) \notag
      \\[-0.5mm]
      &\le \vep\,m_{\phi_{\aaa},h}(\bfu_h -\bu_D^*) +c_\kappa\,c_\vep \,\rho_{\phi,\Omega}(\nabla\bfu_D^*)
        +c_\kappa \,c_\vep \,\rho_{\phi,\Omega}\big(h^{-1}\,\bfu_D^*\big)
    \\[-0.5mm]
    &\quad +\kappa\, c_\vep \,
      \rho_{\phi,\Omega } \big(\nablaDG (\bfu_h - \bfu_D^*)\big )\notag
    \\[-0.5mm]
       &\le \vep\,m_{\phi_{\aaa},h}(\bfu_h -\bu_D^*) +c_\kappa\,c_\vep \,\rho_{\phi,\Omega}(\nabla\bfu_D^*)
         +c_\kappa\,c_\vep \,\rho_{\phi,\Omega}\big(h^{-1}\,\bfu_D^*\big)
    \\[-0.5mm]
    &\quad + \kappa\,c_\vep \,
      M_{\phi,h } (\bfu_h - \bfu_D^*)  \notag
 \end{align*}
 with constants $c_\kappa,c_\vep>0$ depending on $\alpha>0$. Choosing
 first $\vep>0$ and than $\kappa>0$ sufficiently~small, the last three estimates and \eqref{eq:a7} prove the assertion.
\end{proof}

\begin{lemma}\label{lem:funct}
  Let $\AAA$ satisfy Assumption~\ref{ass:1} for a balanced N-function
  ${\phi}$.  Then, for every $\vep \hspace*{-0.15em}>\hspace*{-0.15em}0$, there exists~a~constant
  $c_\vep\hspace*{-0.15em}>\hspace*{-0.15em}0$, depending on the
  characteristics~of~$\AAA$~and~$\phi$ and the chunkiness $\omega_0$, such that for every $\bu_h \in \Vhk$,~we~have~that
  \begin{align*}
    \bigabs{\bighskp{\bb_h}{\bu_h-\PiDG\bu_D^*}}
    &\le\vep\,  M_{\phi,h}(\bu_h-   \bu_D^*) +
      c_\vep\,\rho_{\phi^*,\Omega}(\bfg) +  c_\vep\,\rho_{\phi^*,\Omega}(\bfG)
    \\
    &\quad       +c_\vep\,\rho_{\phi^*,\Gamma_N }(\ba_N)+
      c_\vep\,\rho_{\phi, \Omega}(\nabla \bu_D^*)\,. 
\end{align*}
\end{lemma}
\begin{proof}\let\qed\relax
  Using the identities \eqref{eq:decomp} and \eqref{eq:id}, for every
  $\bu_h \in \Vhk$, we observe that
  \begin{align}
    \label{eq:c1}
    \begin{aligned}
      \bighskp{\bfb_h}{\bu_h-\PiDG \bu_D^*} &=\hskp{\bfg}{\bu_h-\bu_D^*}
      + \bighskp{\bfg}{\bu_D^*-\PiDG \bu_D^*}
      \\
      &\quad
      +\bighskp{\bfG}{\Ghk(\bfu_h-\bu_D^*)}+\bighskp{\bfG}{\Rhk(\bfu_D^*-\PiDG\bu_D^*)}
      \\
      &\quad+\skp{\bfa_N}{\bu_h-\bu_D^*}_{\Gamma_N}+
      \bigskp{\bfa_N}{\bu_D^*-\PiDG\bu_D^*}_{\Gamma_N}
      \\
      &=:J_1+ \ldots +J_6\,.
    \end{aligned}
  \end{align}
  Resorting to the $\varepsilon$-Young inequality \eqref{ineq:young}, \Poincare's inequality
  \eqref{eq:poincare} and the approximation property of $\PiDG$ in
  \eqref{eq:PiDGapprox1a} in doing so,~for~every~${\bu_h \in \Vhk}$,
  we find that
  \begin{align*}
    \abs{J_1 +J_2} &\leq \epsilon\, \rho_{\phi,\Omega}(\bfu_h - \bfu_D^*) +
                c_{\epsilon} \,\rho_{\phi^*,\Omega}(\bfg) + c\,
                \rho_{\phi,\Omega}\big(\bfu_D^* - \PiDG \bfu_D^*\big)
    \\
              &\leq \epsilon \,c\,M_{\phi,h}(\bfu_h - \bfu_D^*) + c_{\epsilon}\,
                \rho_{\phi^*,\Omega}(\bfg) + c\, \rho_{\phi,\Omega}(h \, \nabla
                \bfu_D^*)
    \\
              &\leq \epsilon \,c\,M_{\phi,h}(\bfu_h - \bfu_D^*) + c_{\epsilon}\,
                \rho_{\phi^*,\Omega}(\bfg) + c\, \rho_{\phi,\Omega}( \nabla
                \bfu_D^*) \,. 
  \end{align*}
  Resorting, again, to the $\varepsilon$-Young inequality \eqref{ineq:young}, using the equivalent
  expression for $M_{\phi,h}$ in \eqref{eq:equi}, the stability of
  $\Rhk$ in \eqref{eq:Rgammaest} and the approximation property~of~$\PiDG$ in~\eqref{eq:PiDGapproxmglobal}, for every
  $\bu_h \in \Vhk$, we deduce that
  \begin{align*}
    \abs{J_3 +J_4} &\leq \epsilon\, \rho_{\phi,\Omega}\big(\Ghk(\bfu_h - \bfu_D^*)\big) +
                c_{\epsilon} \,\rho_{\phi^*,\Omega}(\bfG) + c\,
                \rho_{\phi,\Omega}\big(\Rhk(\bfu_D^* - \PiDG \bfu_D^*)\big)
    \\
              &\leq \epsilon \,c\,M_{\phi,h}(\bfu_h - \bfu_D^*) + c_{\epsilon}\,
                \rho_{\phi^*,\Omega}(\bfG) + c\, m_{\phi,h}\big(
                \bu_D^* -\PiDG  \bfu_D^*\big) 
    \\
              &\leq \epsilon \,c\,M_{\phi,h}(\bfu_h - \bfu_D^*) + c_{\epsilon}\,
                \rho_{\phi^*,\Omega}(\bfG) + c\, \rho_{\phi,\Omega}( \nabla
                \bfu_D^*) \,. 
  \end{align*}
  Finally, we estimate, using the $\varepsilon$-Young inequality \eqref{ineq:young}, the trace inequality~\eqref{eq:trace},
the $L^\phi$-gradient stability of $\PiDG$~in~\eqref{eq:PiDGapprox2} and the
  approximation property~of~$\PiDG$~in \eqref{eq:PiDGapproxmglobal}
  to conclude for every $\bu_h \in \Vhk$ that
  \begin{align*}
    \abs{J_5+J_6} &\leq c_\varepsilon\, \rho_{\phi^*, \Gamma_N}(\ba_N) +
                \varepsilon \,\rho_{\phi, \Gamma_N}(\bfu_h-\bu_D^*) +c\,
                 \rho_{\phi, \Gamma_N}\big(\bfu_D^*-   \PiDG \bu_D^*\big)
    \\
              &\leq c_\varepsilon\, \rho_{\phi^*, \Gamma_N}(\ba_N) +
                \varepsilon\,c\, M_{\phi, h}(\bfu_h-\bu_D^*) + c\,
                M_{\phi, h}\big(\bfu_D^*-\PiDG \bu_D^*\big)
    \\
              &\leq c_\varepsilon\, \rho_{\phi^*, \Gamma_N}(\ba_N) +
                \varepsilon\,c\, M_{\phi,h}(\bfu_h-\bu_D^*) + c\,
                \rho_{\phi, \Omega}(\nabla \bfu_D^*)\,. \tag*{$\square$}
  \end{align*}
\end{proof}

\begin{lemma}\label{lem:funct1}
  Let $\AAA$ satisfy Assumption~\ref{ass:1} for a balanced N-function
  ${\phi}$.  Then, for every $\vep >0$, there exists~a~constant
  $c_\vep>0$ such that for every $\bu_h \in \Vhk$,~we~have~that
  \begin{align*}
    \abs{\hskp{\bb_h}{\bu_h}}
    \le\vep\,  M_{\phi,h}(\bu_h) +
      c_\vep\,\rho_{\phi^*,\Omega}(\bfg) +  c_\vep\,\rho_{\phi^*,\Omega}(\bfG)
      +c_\vep\,\rho_{\phi^*,\Gamma_N }(\ba_N)      \,. 
\end{align*}
\end{lemma}
\begin{proof}
The assertion follows, using the same tools as in the proof of Lemma
	\ref{lem:funct}. 
\end{proof}

Now we have everything at our disposal to prove the existence of
discrete solutions and their stability.
\begin{proposition}[Well-posedness]\label{prop:exist}
  Let $\AAA$ satisfy Assumption~\ref{ass:1} for a balanced N-function
  ${\phi}$. Then, for given data
  $\bu_D\hspace*{-0.1em}\in\hspace*{-0.1em}\trace W^{1,\phi}(\Omega)$,
  ${\bfg\hspace*{-0.1em}\in\hspace*{-0.1em}
    \smash{L^{\phi^*}(\Omega)}}$,~${\bfG\hspace*{-0.1em}\in\hspace*{-0.1em}
    \smash{L^{\phi^*}(\Omega)}}$~and
  $\ba_N\!\in\! \smash{L^{\phi^*} \!(\Gamma_N)}$ as well as given
  $k \!\in\! \setN$ and $\alpha \!> \!0$, $h >0$, there exist $\bu_h\! \in\! \Vhk$~solving~\eqref{eq:primal},~and   ${(\bL_h,\bA_h)^\top\in   \Xhk \times
  \Xhk}$ such that $(\bu_h,\bL_h,\bA_h)^\top $ solves~\eqref{eq:DG}.
\end{proposition}
\begin{proof}
  Lemma \ref{lem:coer1} and Lemma \ref{lem:funct1} with $\vep>0$ small
  enough yield~for~every~${\bu_h\in \Vhk}$
  \begin{align}
      \hskp{\bB_h\bu_h}{\bu_h}-\hskp{\bb_h}{\bu_h} &\ge c\,
      M_{\phi,h}(\bu_h) - c\,\rho_{\phi, \Omega}(\nabla \bu_D^*) -
      c\,\rho_{\phi, \Omega}(\bu_D^*)-c\, \rho_{\phi, \Omega}\big( h^{-1} \bu_D^*\big)\notag
      \\
      &\quad 
      -
      c_\vep\,\rho_{\phi^*,\Omega}(\bfg) -
      c_\vep\,\rho_{\phi^*,\Omega}(\bfG)
      -c_\vep\,\rho_{\phi^*,\Gamma_N }(\ba_N)\,. \label{eq:d1}
  \end{align}
  We equip $\Vhk$ with the norm $\smash{\norm{\cdot}_{\smash{M_{\phi,h}}}}\!$. Using
  that $\smash{\norm{\bu_h}_{\smash{M_{\phi,h}}}}\!\!\le\! \smash{\smash{M_{\phi,h}}}(\bu_h)$ if
  $\smash{\norm{\bu_h}_{\smash{M_{\phi,h}}}\!\!\ge\! 1}$, we find that the right-hand side in
  \eqref{eq:d1} converges to infinity for $\smash{\norm{\bu_h}_{M_{\phi,h}}}
  \to \infty$. Thus, Brouwer's fixed point theorem yields the existence
  of~$\bu_h\in \Vhk$~solving~\eqref{eq:primal}. In turn, we also conclude the existence of
  $\bL_h\in \Xhk$ and $\bA_h\in \Xhk$ from~\eqref{eq:Lh}~and~\eqref{eq:Ah},
  respectively, which shows the solvability of \eqref{eq:DG}.
\end{proof}
\begin{proposition}[Stability]\label{prop:stab}
    Let $\AAA$ satisfy Assumption~\ref{ass:1} for a balanced
    {N-func\-tion}
  ${\phi}$. Moreover, let $\bu_h \in \Vhk $ be a
  solution of \eqref{eq:primal} for $\alpha >0$, $h >0$,  and $k \in \setN$.~Then,~it~holds
     \begin{align}\label{eq:apri2}
    \begin{aligned}
      &M_{\phi,h}(\bu_h -\bu_D^*) +  m_{\phi_{\aaal},h} (\bu_h -\bu_D^*)
      \\
      &\quad\le  c\, \rho_{\phi, \Omega}(\nabla
      \bu_D^*)  + c\, \rho_{\phi^*,\Omega}(\bfg)
      + c\, \rho_{\phi^*,\Omega}(\bfG) + c\, \rho_{\phi^*,\Gamma_N
       }(\ba_N)
    \end{aligned}
     \end{align}
     with a constant $c>0$ depending only on $\alpha>0$, the characteristics of
     $\AAA$ and $\phi$, $k \in \setN$, and the chunkiness $\omega_0>0$.
\end{proposition}
\begin{proof}
  The assertion follows by combining Lemma \ref{lem:coer} and
  Lemma \ref{lem:funct}~for~${\vep> 0}$ sufficiently small, if we choose $\bz_h =\bu_h
  	-\PiDG\bu_D^*\in \Vhk$~in~\eqref{eq:op}. 
\end{proof}
\begin{corollary}\label{cor:stab}
     Let $\AAA$ satisfy Assumption~\ref{ass:1} for a balanced
    {N-func\-tion}
  ${\phi}$, and let $(\bu_h,\bL_h,\bA_h)^\top \in \Vhk \times \Xhk \times \Xhk$ be a solution
  of \eqref{eq:DG} for $\alpha>0$, $h >0$  and $k \in \setN$. Then, it holds
  \begin{align*}
    \begin{aligned}
      &\rho_{\phi,\Omega}(\bL_h ) + \rho_{\phi^*,\Omega}(\bA_h )+h\rho_{\phi^*,\smash{\Gamma_h^{i}}\cup\Gamma_D}\big(\AAA_{\aaa}( h^{-1} \llbracket(\bu_h-\bu_D^*)\otimes\bfn\rrbracket)\big)
     \\&\quad \le  c\, \rho_{\phi, \Omega}(\nabla
      \bu_D^*)  + c\, \rho_{\phi^*,\Omega}(\bfg)
      + c\, \rho_{\phi^*,\Omega}(\bfG) + c\, \rho_{\phi^*,\Gamma_N
       }(\ba_N)
    \end{aligned}
     \end{align*}
  with a constant $c>0$ depending only on $\alpha>0$, the characteristics
        of $\AAA$ and $\phi$,
  	and the chunkiness $\omega_0>0$.   
\end{corollary}

\begin{proof}
  From \eqref{eq:Lh} and \eqref{eq:id}, it  follows that
  $\bL_h=\Ghk(\bu_h-\bu_D^*)+\nabla\bu_D^*$ in $L^\phi(\Omega)$, which together with the
  properties of $\phi$ and the equivalent expression for the modular
  $M_{\phi,h}$~in Lemma~\ref{eq:equi}, implies that
  \begin{align}
    \begin{aligned}
      \rho_{\phi,\Omega}(\bL_h )& \leq
      c\,\rho_{\phi,\Omega}\big (\Ghk(\bu_h-\bu_D^*)\big )+c\,\rho_{\phi,\Omega}(\nabla
      \bu_D^*)
      \\
      &\leq c\,M_{\phi,h}(\bu_h -\bu_D^*)
      +c\,\rho_{\phi,\Omega}(\nabla \bu_D^*)\,.
    \end{aligned}\label{cor:stab.1}
  \end{align}
  From \eqref{eq:Ah} we know that $\bfA_h=\Pi_h^k\AAA(\bL_h)$ in $\Xhk$, so the
  stability of $\PiDG$ in \eqref{eq:PiDGLpsistable},
  \eqref{eq:hammere2} and \eqref{eq:phi*phi'} imply that
  \begin{align}
    \rho_{\phi^*,\Omega}(\bA_h )\leq c\,\rho_{\phi^*,\Omega}\big ( \phi'(
    \vert  \bL_h\vert ) \big)\leq c\,\rho_{\phi,\Omega}(\bL_h
    )\,.\label{cor:stab.2} 
  \end{align}
  In addition, using the shift change in \eqref{eq:phi_prime_shift}
  and \eqref{eq:hammere}, we find that
  \begin{align*}
    \big\vert\AAA_{\aaa}(h^{-1}
    \jump{(\bu_h-\bu_D^*)\otimes\bfn}\big)\big\vert
    &=\phi_{\aaa}'(h^{-1}\vert
      \llbracket(\bu_h-\bu_D^*)\otimes\bfn\rrbracket\vert )
    \\
    &\leq c\,\phi'(h^{-1}\vert \llbracket(\bu_h-\bu_D^*)\otimes\bfn
      \rrbracket\vert)+c\,\phi'(\aaa
    )\,, 
  \end{align*}
  which, using \eqref{eq:phi*phi'},  the
  discrete trace inequality  \eqref{eq:pol-trace}, and the $L^\phi$-stability~of~$\Pi_h^0$~in \eqref{eq:PiDGLpsistable}, implies that
  \begin{align}
    \begin{aligned}
      &h\rho_{\phi^*,\smash{\Gamma_h^{i}}\cup\Gamma_D}\big(\AAA_{\aaa}(
      h^{-1}
      \llbracket(\bu_h-\bu_D^*)\otimes\bfn\rrbracket)\big)
     \\
      &\leq  c\,h\,\rho_{\phi^*,\smash{\Gamma_h^{i}}\cup\Gamma_D}\big ( \phi'(h^{-1}\vert
      \llbracket(\bu_h-\bu_D^*)\otimes\bfn\rrbracket\vert)\big )\\&\quad+c\,
      h\,\rho_{\phi^*,\smash{\Gamma_h^{i}}\cup\Gamma_D}\big
      (\phi'(\aaa\big )
      \\
      &\leq c\,m_{\phi,h}(\bu_h-\bu_D^*)+c\,
      h\,\rho_{\phi,\smash{\Gamma_h^{i}}\cup\Gamma_D}\big
      (\avg{\Pi_{h}^0\bL_h}\big )
      \\
      &\leq c\,m_{\phi,h}(\bu_h-\bu_D^*)+c\,
      \rho_{\phi,\Omega}\big  (\Pi_{h}^0\bL_h\big )
      \\
      &\leq c\,M_{\phi,h}(\bu_h-\bu_D^*)+c\, \rho_{\phi,\Omega}(\bL_h)\,.
    \end{aligned}\label{cor:stab.3}
  \end{align}
  Resorting to Proposition \ref{prop:stab}, we conclude the assertion~from~\mbox{\eqref{cor:stab.1}--\eqref{cor:stab.3}}.
\end{proof}

\begin{theorem}[Convergence]\label{thm:minty}
  Let $\AAA$ satisfy Assumption~\ref{ass:1} for a balanced
  \linebreak {N-func\-tion} ${\phi}$. For every $h>0$, let 
  $(\bu_h,\bL_h,\bA_h)^\top \in \Vhk \times \Xhk \times \Xhk$ be a solution~of \eqref{eq:DG} for $\alpha\!>\!0$ and $k \!\in\! \setN$.
  Moreover, let $(h_n)_{n\in \mathbb{N}}\!\subseteq \!\setR^{>0}$ be
  any sequence such that $h_n\!\to\! 0$ $(n\!\to\! \infty)$~and define
  $(\bu_n,\bL_n,\bA_n)^\top\vcentcolon=(\bu_{h_n},\bL_{h_n},\bA_{h_n})^\top \in U_{h_n}^k
 \times X_{h_n}^k \times X_{h_n}^k $, $n\in \mathbb{N}$. Then,~it~holds
  \begin{align}
    \begin{aligned}
      \bu_n&\weakto \bu&&\quad\text{ in }L^{\phi}(\Omega)&&\quad(n\to \infty)\,,\\
      \bL_n&\weakto \nb \bu&&\quad\text{ in }L^{\phi}(\Omega)&&\quad(n\to \infty)\,,\\
      \bu_n&\weakto \bu&&\quad\text{ in
      }L^{\phi}(\Gamma_N)&&\quad(n\to \infty)\,,
    \end{aligned}\label{eq:minty0.1}
  \end{align}
  where $\bu\in W^{1,\phi}(\Omega)$ is the unique weak solution of
  \eqref{eq:p-lap}, i.e., $\bu=\bu_D$ in $L^{\phi}(\Gamma_D)$ and for
  every $\bz\in \smash{W^{1,\phi}_{\Gamma_D}(\Omega)}$, it holds
  \begin{align}
    (\AAA(\nb\bu),\nb\bz)=(\bfg,\bz)+(\bfG,\nb\bz)+\langle
    \ba_N,\bz\rangle_{\Gamma_N}.\label{eq:minty0.2} 
  \end{align}
\end{theorem}

\begin{proof}
  Resorting to the theory of monotone operators (cf.~\cite{zei-IIB}), we~observe~that~\eqref{eq:p-lap} admits a
  unique~weak~\mbox{solution}, since $\AAA$ is strictly
  monotone. \!Thus, to prove~the~convergences \eqref{eq:minty0.1} for
  the entire sequence to the unique weak~solution~${\bu\in
    W^{1,\phi}(\Omega)}$ of \eqref{eq:p-lap}, it is sufficient to
  show that each sequence has a subsequence that satisfies \eqref{eq:minty0.1}.  Then, the
  assertion~for the entire sequence follows from the standard
  convergence principle. To this end, we adapt Minty's trick to our
  situation. Appealing~to~Proposition~\ref{prop:stab}, it holds
  $\smash{\sup_{n\in
      \mathbb{N}}{M_{\phi,h_n}(\bu_n-\bu_D^*)}}<\infty$, which,
  resorting to Lemma~\ref{lem:wc}, yields a not relabeled
  subsequence~and~a~function
  ${\tilde{\bu} \in W^{1,\phi}_{\Gamma_D}(\Omega)}$~such~that
  \begin{align}
    \begin{aligned}
      \bu_n-\bu_D^*&\weakto \tilde{\bu}&&\quad\text{ in }L^{\phi}(\Omega)&&\quad(n\to \infty)\,,\\
      \boldsymbol{\mathcal{G}}_{h_n}^k(\bu_n-\bu_D^*)&\weakto
      \nabla\tilde{\bu}&&\quad\text{ in }L^{\phi}(\Omega)&&\quad(n\to
      \infty)\,,
      \\
      \bu_n-\bu_D^*&\weakto \tilde{\bu}&&\quad\text{ in
      }L^{\phi}(\Gamma_N)&&\quad(n\to \infty)\,.
    \end{aligned}\label{eq:minty1}
  \end{align}
  Therefore, introducing the notation
  $\bu\vcentcolon=\tilde{\bu}+\bu_D^*\in W^{1,\phi}(\Omega)$, in particular,
  exploiting that
  $\boldsymbol{\mathcal{G}}_{h_n}^k(\bu_n-\bu_D^*)=\bL_n-\nb
  \bu_D^*$ (cf.~\eqref{eq:id}), 
  \eqref{eq:minty1} is exactly \eqref{eq:minty0.1} and
  yields directly that $\bu=\bu_D$ in $L^{\phi}(\Gamma_D)$.
  Resorting to
  Corollary~\ref{cor:stab}~and~the~\mbox{reflexivity}~of~$L^{\phi^*}(\Omega)$,
  we obtain a further not relabeled subsequence~and~a~function
  $\bA\in \smash{L^{\phi^*}(\Omega)}$ such that
  \begin{align}
    \bA_n\weakto \bA\quad\text{ in }L^{\phi^*}(\Omega)\quad(n\to \infty).\label{eq:minty3}
  \end{align}
  Next, let $\bz\in \smash{W^{1,\phi}_{\Gamma_D}(\Omega)}$ be arbitrary and
  define $\bz_n\vcentcolon=\smash{\Uppi_{h_n}^k\bz\in V_{h_n}^k}$ for every
  $n\in \mathbb{N}$.  Then, appealing to the convergence properties
  for $\smash{\Uppi_{h_n}^k}$ in Lemma \ref{lem:conv},
  Corollary~\ref{cor:stabpi}, and the trace inequality in
  \eqref{eq:trace}, we obtain for $n\to \infty$ that
  \begin{align}
    \begin{aligned}
      \rho_{\phi, \Omega}\big
    (\boldsymbol{\mathcal{G}}_{h_n}^k(\bz_n-\bz)\big )+
      \rho_{\phi, \Omega}(\bz_n-\bz)+M_{\phi,h_n}(\bz_n-\bz)+\rho_{\phi,
        \Gamma_N}(\bz_n-\bz)&\to 0\,.     \hspace*{-3mm}
    \end{aligned}\label{eq:minty4}
  \end{align}
  Apart from that, appealing to Corollary~\ref{cor:stab}, we have that
  \begin{align}
  {\sup_{n\in \mathbb{N}}{h_n\,\rho_{\phi^*,\smash{\Gamma_h^{i}}\cup\Gamma_D}
    \big(\boldsymbol{\mathcal{A}}_{\smash{\aaan}}
    (h_n^{-1}\jump{(\bu_n-\bu_D^*)\otimes
    \bn} )\big)}} \le c\,.\label{eq:ah}
  \end{align}
  Since $\bz_n\in U_{h_n}^k$ is an admissible test function in \eqref{eq:DG}, for every $n\in \mathbb{N}$,  we have
  that 
  \begin{align}
      \big(\bA_n,\boldsymbol{\mathcal{G}}_{h_n}^k\bz_n\big)
      &=(\bfg,\bz_n)+\big(\bfG,\boldsymbol{\mathcal{G}}_{h_n}^k\bz_n\big)+\langle
      \ba_N,\bz_n\rangle_{\Gamma_N}\label{eq:minty5}
      \\
      &\quad-\alpha\big\langle
      \boldsymbol{\mathcal{A}}_{\aaan}(h_n^{-1}\jump{(\bu_n-\bu_D^*)\otimes
        \bn}), \jump{(\bz_n-\bz)\otimes
        \bn}\big\rangle_{\smash{\Gamma_h^{i}}\cup\Gamma_D}\,,\notag 
  \end{align}
  where we used that $\jump{\bz\otimes \bn}=\bfzero $ on
  ${\smash{\Gamma_h^{i}}\cup\Gamma_D}$.
  This together with \eqref{eq:minty4} and \eqref{eq:ah} yields for
  $n\to \infty$ that for every
  $\smash{\bz\in W^{1,\phi}_{\Gamma_D}(\Omega)}$
  \begin{align}
    (\bA,\nabla\bz)=(\bfg,\bz)+(\bfG,\nb\bz)+\langle \ba_N,\bz\rangle_{\Gamma_N}\,.\label{eq:minty6}
  \end{align}
  For arbitrary $\bz\in
  W^{1,\phi}(\Omega)$, we set
  $\bz_n\vcentcolon=\Pi_{h_n}^k\bz\in V_{h_n}^k$, $n\in
  \mathbb{N}$. Then,~recalling~that
  $\bA_n=\Uppi_{h_n}^k\boldsymbol{\mathcal{A}}(\bL_n)$ in $\smash{X_{h_n}^k}$,
  $n\in \mathbb{N}$, the monotonicity of 
  $\boldsymbol{\mathcal{A}}$, the self-adjointness of
  $\Uppi_{h_n}^k$, and \eqref{eq:minty5}, for every $n\in \mathbb{N}$ and
  $\bz\in W^{1,\phi}(\Omega)$, further yield that
  \begin{align}
      0&\leq\big(\AAA(\bL_n)-\AAA(\nabla_{h_n}\bz_n),\bL_n-\nabla_{h_n}\bz_n\big) \notag 
      \\
      &=\big(\bA_n-\AAA(\nabla_{h_n}\bz_n),\bL_n-\nabla_{h_n}\bz_n\big) \notag 
      \\
      & =\big(\bA_n,\boldsymbol{\mathcal{G}}_{h_n}^k(\bu_n-\Uppi_{h_n}^k\bu_D^*)\big)
 +\big(\bA_n,\boldsymbol{\mathcal{G}}_{h_n}^k(\Uppi_{h_n}^k\bu_D^*-\bu_D^*)
 +\nb\bu_D^*-\nb_{h_n}\bz_n\big) \notag 
      \\
      &\quad+\big(\AAA(\nabla_{h_n}\bz_n),\nabla_{h_n}\bz_n-\bL_n\big) \label{eq:minty7}
      \\
      &
      =\big(\bfg,\bu_n-\Uppi_{h_n}^k\bu_D^*\big)+\big(\bfG,
      \boldsymbol{\mathcal{G}}_{h_n}^k(\bu_n-\Uppi_{h_n}^k\bu_D^*)\big)+\big\langle
      \ba_N,\bu_n-\Uppi_{h_n}^k\bu_D^*\big\rangle_{\Gamma_N}\notag 
      \\
      &\quad  -\alpha\,\big\langle
      \AAA_{\aaan}(h_n^{-1}\jump{(\bu_n-\bu_D^*)\otimes
        \bn}),\jump{(\bu_n-\Uppi_{h_n}^k\bu_D^*\pm\bu_D^*)\otimes
        \bn}\big\rangle_{\smash{\Gamma_h^{i}}\cup\Gamma_D}\notag 
      \\
      &\quad+\big(\bA_n,\boldsymbol{\mathcal{G}}_{h_n}^k(\Uppi_{h_n}^k\bu_D^*-\bu_D^*)
      +\nb\bu_D^*-\nb_{h_n}\bz\big)
+\big(\AAA(\nb_{h_n}\bz_n),\nb_{h_n}\bz_n-\bL_n\big)\,.
    \hspace*{-5mm}\notag 
  \end{align}
  Moreover, using $\AAA_{\aaa}(\bP):\bP\ge 0$ for every
  ${\bP\in \setR^{d\times n}}$, we conclude from~\eqref{eq:minty7} that
  for every $n\in \mathbb{N}$, there holds
  \begin{align}
    \begin{aligned}
      0&\leq\big(\bA_n-\AAA(\nb_{h_n}\bz_n),\bL_n-\nb_{h_n}\bz_n\big)
      \\
      &  \leq \big(\bfg,\bu_n-\Uppi_{h_n}^k\bu_D^*\big)+
      \big(\bfG,\boldsymbol{\mathcal{G}}_{h_n}^k(\bu_n-\Uppi_{h_n}^k\bu_D^*)\big)+\big\langle
      \ba_N,\bu_n-\Uppi_{h_n}^k\bu_D^*\big\rangle_{\Gamma_N}
      \\
      &\quad  +\alpha\big\langle
      \AAA_{\aaan}(h_n^{-1}\jump{(\bu_n-\bu_D^*)\otimes
        \bn}),\jump{(\Uppi_{h_n}^k\bu_D^*-\bu_D^*)\otimes
        \bn}\big\rangle_{\smash{\Gamma_h^{i}}\cup\Gamma_D}
      \\
      &\quad+\big(\bA_n,\boldsymbol{\mathcal{G}}_{h_n}^k
     (\Uppi_{h_n}^k\bu_D^*-\bu_D^*)+\nb\bu_D^*-\nb_{h_n}\bz_n\big)
      +\big(\AAA(\nb_{h_n}\bz_n),\nb_{h_n}\bz_n-\bL_n\big)\,.
    \end{aligned}\hspace*{-8mm}\label{eq:minty8}
  \end{align}
  Hence, by passing for $n\to \infty$ in \eqref{eq:minty8}, taking into
  account \eqref{eq:minty1}--\eqref{eq:minty4},~the~conver-gence
  properties of $\Uppi_{h_n}^k\!$ in Lemma \ref{lem:conv},
  Corollary~\ref{cor:stabpi}, the trace~\mbox{inequality}~in~\eqref{eq:trace}, the estimate \eqref{eq:ah}, and the fact that
  $\AAA$ generates a Nemyckii operator, we conclude for every
  $\bz\in W^{1,\phi}(\Omega)$ that 
  \begin{align}
    \begin{aligned}
      0&\leq (\bfg,\bu-\bu_D^*)+ (\bfG,\nb(\bu-\bu_D^*) )+\langle
      \ba_N,\bu-\bu_D^*\rangle_{\Gamma_N}
      \\
      &\quad
      + (\bA,\nb(\bu_D^*-\bz) 
      )+(\AAA(\nabla\bz),\nabla\bz-\nb\bu)
      \\
      &= (\bA,\nabla(\bu-\bu_D^*) )
      +  (\bA,\nb(\bu_D^*-\bz) )+  (\AAA(\nabla\bz),\nabla\bz-\nb\bu)
      \\
      &=(\bA-\AAA(\nabla\bz),\nb\bu-\nabla\bz)\,,
    \end{aligned}\label{eq:minty9}
  \end{align}
  where we used for the first equality sign that
  $\bu-\bu_D^*=\tilde{\bu}\in \smash{W^{1,\phi}_{\Gamma_D}(\Omega)}$
  and,~thus,~\eqref{eq:minty6} applies with
  $\bz=\bu-\bu_D^*\in \smash{W^{1,\phi}_{\Gamma_D}(\Omega)}$. Eventually, choosing
  $\bz\vcentcolon=\bu\pm \tau \tilde{\bz}\in \smash{W^{1,\phi}(\Omega)}$ in
  \eqref{eq:minty9} for arbitrary $\tau\!\in\! (0,1)$ and
  $\tilde{\bz}\!\in\! \smash{W^{1,\phi}_{\Gamma_D}(\Omega)}$, diving
  by $\tau\!>\!0$ and passing~for~${\tau\!\to\! 0}$, for every
  $\tilde{\bz}\in \smash{W^{1,\phi}_{\Gamma_D}(\Omega)}$, we conclude
  that $ (\bA-\AAA(\nb\bu),\nabla\tilde{\bz})=0$,~which,~due~to~\eqref{eq:minty6}, implies that $\bu\in W^{1,\phi}(\Omega)$
  satisfies \eqref{eq:minty0.2}.
\end{proof}

\begin{remark}
	\label{rem:zerobnd}
	There are only few numerical investigations for nonlinear
        problems with Orlicz-structure showing the convergence of
        discrete solutions to a weak solution. We are only aware of
        the studies \cite{dr-interpol,DKS13,EW13,BC15,Ruf17}. Non of
        these contributions uses DG methods. For the subclass of
        nonlinear problems of $p$-Laplace~type~DG~methods have been
        used in \cite{ern-p-laplace,BufOrt09,dkrt-ldg,CS16,sip,QS19}. In all
        these investigations,~only~the~case $\bfG=\bfzero$ is
        treated. Thus, to the best of the author's knowledge Theorem
        \ref{thm:minty} is the first convergence result for a DG
        scheme in an Orlicz-setting for general~\mbox{right-hand}~sides
$\bfg -\divo \bfG$.
\end{remark}

\section{Error estimates}
\label{ssec:primalerror}

In order to establish error estimates, we need to find a system similar~to~\eqref{eq:DG},~which is
satisfied by a solution of our original problem \eqref{eq:p-lap} in the case
$\bfG=\bfzero$.  Using the notation ${\bL=\nabla \bu}$, ${\bA
  =\AAA(\bL)}$, we~find~that
$(\bfu, \bfL, \bfA)^\top \in W^{1,\phi}(\Omega)\times L^\phi(\Omega) \times
L^{\phi^*}(\Omega)$. If, in addition, $\bA \in W^{1,1}(\Omega)$, we~observe~as~in~\cite{dkrt-ldg}, i.e., using integration-by-parts, the boundary
conditions, the properties~of~$\PiDG$, the definition of the discrete
gradient, and of the jump functional, that 
\begin{align}\label{eq:cont}
  \begin{aligned}
    \hskp{\bfL}{\bfX_h} &= \hskp{\nabla \bfu}{ \bfX_h}\,,
    \\
    \hskp{\bfA}{\bfY_h} &= \bighskp{\AAA(\bfL)}{ \bfY_h}\,,
    \\
    \bighskp{\bfA}{\nablaDG \bfz_h} &= \hskp{\bfg}{\bfz_h} +
    \skp{\bfa_N}{\bfz_h}_{\Gamma_N} { +
      \bigskp{\avg{\bA}-\bigavg{\PiDG\bA}}{\jump{\bz_h\otimes
          \bn}}_{\smash{\Gamma_h^{i}}\cup \Gamma_D}}
  \end{aligned}
\end{align}  
is satisfied for all $(\bfX_h,\bfY_h,\bfz_h)^\top \in \Xhk \times \Xhk \times \Vhk$.
Using this and \eqref{eq:primal}, we arrive at
\begin{align}
  \label{eq:errorprimal}
  \begin{aligned}
    & \bighskp{\AAA(\nablaDG \bfu_h +\Rhk \bfu_D^*) - \AAA(\nabla
      \bfu)}{\nablaDG \bfz_h}
    \\
    &\quad + \alpha \bigskp{\AAA_{\aaal}(h^{-1} \jump{(\bfu_h -\bu_D^*)\otimes
        \bfn})}{ \jump{\bfz_h \otimes \bfn}}_{\smash{\Gamma_h^{i}}\cup \Gamma_D}
    \\
    &={\bigskp{\bigavg{\PiDG\bA}-\avg{\bA}}{\jump{\bz_h\otimes
          \bn}}_{\smash{\Gamma_h^{i}}\cup \Gamma_D}}\,,
  \end{aligned}
\end{align}
which is satisfied for all $\bz_h \in \Vhk$. 

Before we prove the convergence rate in Theorem \ref{thm:error}, we derive some estimates.

\begin{lemma}\label{lem:e4}
  Let $\AAA$ satisfy Assumption~\ref{ass:1} for a balanced
   {N-func\-tion} ${\phi}$ and $k\in \mathbb{N}$. Moreover, let $\bu \in W^{1,\phi}(\Omega)$ satisfy
   $\bF(\nabla \bu) \in W^{1,2}(\Omega)$. Then, we have that
   \begin{align}
    \label{eq:e4}
    \rho_{\phi_{\abs{\nabla \bu}},\Omega}\big({ \Rhk(\bfu-\PiDG\bu)}\big)
    &\le  c\, h^2\, \norm{\nabla \bF(\nabla  \bu)}_2^2  
  \end{align}
with a constant $c>0$ depending only on the
characteristics of ${\AAA }$ and $\phi$, and the chunkiness
$\omega_0>0$. The same assertion is valid for the Scott--Zhang interpolation operator
$\PiSZ$, defined in \cite{zhang-scott}, instead of $\PiDG$. 
\end{lemma}

\begin{proof}
Since for every $\bfv \in \WDG$ we have $\smash{\Rhk}\bfv=\sum_{\gamma \in
\smash{\Gamma_h^{i}}\cup \Gamma_D} \smash{\boldsymbol{\mathcal
R}^k_{h,\gamma}}\bfv$ with $\textup{supp}\smash{\boldsymbol{\mathcal
R}^k_{h,\gamma}}\bfv\subseteq S_\gamma$, where for each face ${\gamma\in \smash{\Gamma_h^{i}} \cup
\Gamma_D}$ the set $S_\gamma$ consist of at most two elements
$K \in \mathcal T_h$, and the fact that $\smash{\Omega
=\bigcup_{\gamma\in \smash{\Gamma_h^{i}}\cup \Gamma_D}{S_\gamma}}$, it is
sufficient to treat $\smash{\boldsymbol{\mathcal
R}^k_{h,\gamma}}(\bfu-\PiDG\bu)$ on $S_\gamma$~to~obtain~the global
result \eqref{eq:e4} by summation. The shift
  change in Lemma \ref{lem:change2}, the local stability properties of
  $\boldsymbol{\mathcal R}^k_{h,\gamma}$~in~Lemma~\ref{lem:stabRhk},
  Proposition~\ref{lem:hammer},
  ${\bfu\!-\!\PiDG \bu\!=\!\bfu\!-\!\PiDG \bu\!+\!\PiDG(\bfu\!-\!\PiDG \bu)}$, the
  approximation property of $\PiDG$ in \eqref{eq:PiDGapproxmlocal} for
  $\bfu-\PiDG \bu$, Poincar\'e's inequality~on~$S_\gamma $together
  with \cite[Lemma A.12]{bdr-phi-stokes}, and again
  a shift change~in~Lemma~\ref{lem:change2} together with Poincar\'e's
  inequality on $S_\gamma $, yield
  \allowdisplaybreaks
  \begin{align*}
    &\int_{S_\gamma}\phi_{\abs{{\nabla
      \bu}}} \big(|{ \boldsymbol{\mathcal R}^k_{h,\gamma}
      (\bfu-\PiDG \bu )}|\big)\, \textrm{d}x
    \\
    &\le c\, \int_{S_\gamma}\phi_{\abs{\mean{\nabla
      \bu}_{S_\gamma}}} \big(|{ \boldsymbol{\mathcal R}^k_{h,\gamma}
      (\bfu-\PiDG \bu)}|\big) +\phi_{\abs{{\nabla \bu}}} (|\nabla
      \bfu-\mean{\nabla\bu}_{S_\gamma}|)\, \textrm{d}x
    \\
    &\le c\, h\int_{\gamma}\phi_{\abs{\mean{\nabla
      \bu}_{S_\gamma}}} \big( h^{-1} |\jump{(\bfu-\PiDG \bu)\otimes
  \bn}|\big )\,
      \textrm{d}s+\int_{S_\gamma}\abs{\bF(\nabla 
      \bfu)-\bF(\mean{\nabla\bu}_{S_\gamma})}^2\, \textrm{d}x
    \\
    &\le c\,\int_{S_\gamma} \phi_{\abs{\mean{\nabla
      \bu}_{S_\gamma}}} \big(|{ \nabla _h (\bfu-\PiDG \bu)}|\big) \, \textrm{d}x+c\,\int_{S_\gamma} 
      h^2\,\abs{\nabla \bF(\nabla \bu)}^2   \, \textrm{d}x
    \\
    &\le c\,\int_{S_\gamma} \phi_{\abs{{\nabla
      \bu}}} \big(|{ \nabla _h (\bfu-\PiDG \bu)}|\big) \, \textrm{d}x+c\,\int_{S_\gamma}
      h^2\,\abs{\nabla \bF(\nabla \bu)}^2   \, \textrm{d}x\,.
  \end{align*}
  By summation over $\gamma \in \smash{\Gamma_h^{i}}\cup \Gamma_D$, using
  Proposition \ref{lem:hammer} and~Lemma~\ref{pro:SZnablaF}~in~doing~so, we
  conclude that
  \begin{align*}    
     \rho_{\phi_{\abs{\nabla \bu}},\Omega}\big({ \Rhk
      (\bfu-\PiDG\bu )}\big)
    &\le c\,\rho_{\phi_{\abs{\nabla \bu}},\Omega}\big({ \nabla
      \bfu-\nabla _h\PiDG\bu}\big)+ c\, h^2\,\norm{\nabla \bF(\nabla
      \bu)}_2^2   \notag
    \\
    &\le c\, \bignorm{\bF(\nabla \bu) -\bF (\nabla _h\PiDG\bu )}_2^2 + c\, h^2\, \norm{\nabla \bF(\nabla
      \bu)}_2^2   \notag
    \\
    &\le  c\, h^2\, \norm{\nabla \bF(\nabla
      \bu)}_2^2   \,.
  \end{align*}
  This proves the assertion for $\PiDG$. Since the Scott--Zhang
  interpolation operator has the same properties
  (cf.~\cite{dr-interpol,dkrt-ldg}), the assertion for $\PiSZ$ follows analogously.
\end{proof}
\begin{lemma}\label{lem:e7}
  Let $\AAA$ satisfy Assumption~\ref{ass:1} for a balanced
  {N-func\-tion} ${\phi}$. Moreover, let $\bX \in L^\phi(\Omega)$,
  $\bY \in L^{\phi^*}(\Omega)$, and let $\bu \in W^{1,\phi}(\Omega)$
  satisfy $\bF(\nabla \bu) \in W^{1,2}(\Omega)$. Then,
  for any ${j\in \setN_0} $, it~holds
  \begin{align*}
      \bignorm{\bF(\nabla \bu) -\bF\big(\Uppi^j_h\bX\big)}_2^2 &\le
      c\, h^2 \norm{\nabla \bF(\nabla \bu) }_2^2 + c\,
      \norm{\bF(\nabla \bu) -\bF(\bX)}_2^2\,,
      \\
      \bignorm{\bF^*(\AAA(\nabla \bu))\! -\!\bF^*\big(\Uppi^j_h\AAA(\bY)\big)}_2^2 &\le
      c\, h^2 \norm{\nabla \bF(\nabla \bu) }_2^2 + c\,
      \norm{\bF^*(\AAA(\nabla \bu)) \!-\!\bF^*(\AAA(\bY))}_2^2
  \end{align*}
  with a  constant $c>0$ depending only on the characteristics of ${\AAA}$ and
  $\phi$, and the chunkiness~$\omega_0>0$.
\end{lemma}
\begin{proof}\let\qed\relax
  Resorting to \Poincare's inequality on each $K\in \mathcal T_h$ and Proposition
  \ref{lem:hammer}, we find that
  \begin{align}
    \bignorm{\bF(\nabla \bu) -\bF\big(\Uppi^j_h\bX\big)}_2^2
    &\le 2\,\bignorm{\bF(\nabla \bu) -\bF\big(\Pia \nabla \bu\big)}_2^2 + 2\,
      \bignorm{\bF\big(\Pia \nabla \bu\big) -\bF\big(\Uppi^j_h\bX\big)}_2^2\notag
    \\
    &\le c\, h^2 \norm{\nabla \bF(\nabla \bu) }_2^2 + c\,
      \smash{\rho_{\phi_{\abs{\Pia \nabla \bu}},\Omega}\big( \Pia \nabla \bfu
      - \Uppi^j_h \bX\big)}\,.\label{eq:e7.1}
  \end{align}
  The $L^\phi$-stability of $\Uppi^k_h$
  in \eqref{eq:PiDGLpsistable}, Proposition~\ref{lem:hammer}, and
  again \Poincare's inequality on each $K\in \mathcal T_h$ yield
  \begin{align}
    \rho_{\phi_{\abs{\Pia \nabla \bu}},\Omega}\big( \Pia \nabla \bfu
    - \Uppi^j_h \bX\big)\label{eq:e7.2}
    &=       \rho_{\phi_{\abs{\Pia \nabla \bu}},\Omega}\big( \Uppi^j_h\big(\Pia \nabla \bfu
      - \bX\big)\big)
    \\[-1mm]
    &\le c\,      \rho_{\phi_{\abs{\Pia \nabla \bu}},\Omega}\big( \Pia \nabla \bfu
      - \bX\big)\notag
    \\[-1mm]
    &\le c\, \bignorm{\bF\big(\Pia\nabla \bu\big) -\bF(\nb\bu) }_2^2 +c\,
      \norm{ \bF(\nabla \bu) -\bF(\bX) }_2^2  \notag
    \\
    &\le c\, h^2 \norm{\nabla \bF(\nabla \bu) }_2^2 +c\, \norm{
      \bF(\nabla \bu) -\bF(\bX) }_2^2 \,.\notag
  \end{align}
  Combining \eqref{eq:e7.1} and \eqref{eq:e7.2}, we conclude the first
  assertion. The second assertion follows analogously, if we
  additionally use Lemma \ref{lem:F-F^*-diff} and (cf.~\cite[Lemma 4.4]{dkrt-ldg})
  \begin{align}
    \bignorm{\bF(\nabla \bu) -\bF\big(\Pia \nabla \bu\big)}_2^2 \sim 
    \bignorm{\bF^*(\AAA(\nabla \bu)) -\bF^*\big(\Pia \AAA(\nabla \bu)\big)}_2^2\,. \tag*{$\square$}
  \end{align}
\end{proof}

\begin{lemma}\label{lem:e5}
	Let $\AAA$ satisfy Assumption~\ref{ass:1} for a balanced
	{N-func\-tion} ${\phi}$ and $k\in \mathbb{N}_0$. Moreover, let  $\bfX\in L^\varphi(\Omega)$ and let $\bu \in W^{1,\phi}(\Omega)$ satisfy
	$\bF(\nabla \bu) \in W^{1,2}(\Omega)$. Then, we have that
	\begin{align}
		\label{eq:e5}
		h\,\rho_{\phi_{\abs{\nabla \bu}},\smash{\Gamma_h^{i}}\cup\Gamma_D}\big(\vert \vert \nabla\bu\vert -\avg{\abs{\Pia \bfX}}\vert\big)
		&\le  c\, h^2\, \norm{\nabla \bF(\nabla  \bu)}_2^2  +c\,\norm{\bF(\nabla \bu)-\bF(\bX)}^2_2
	\end{align}
	with a constant $c>0$ depending only on the
	characteristics of ${\AAA }$ and $\phi$, and the chunkiness
	$\omega_0>0$. 
\end{lemma}

\begin{proof}
	Using $\vert \vert \nabla\bu\vert -\avg{\abs{\Pia \bfX}}\vert\leq\avg{\vert\nabla\bu -\Pia \bfX\vert}$ on $\smash{\Gamma_h^{i}}\cup\Gamma_D$, 
	the convexity of $\phi_{\abs{\nabla \bu}}$, \eqref{eq:hammera} and the trace inequality \eqref{eq:emb}, for every $\gamma\in \smash{\Gamma_h^{i}}\cup\Gamma_D$, we find that
	\begin{align*}
		&h\int_{\gamma}{\phi_{\abs{\nabla \bu}}\big(\vert \vert \nabla\bu\vert -\avg{\abs{\Pia \bfX}}\vert\big)\,\mathrm{d}s}\leq c\,h\int_{\gamma}{\phi_{\abs{\nabla \bu}}\big(\avg{\vert\nabla\bu -\Pia \bfX\vert}\big)\,\mathrm{d}s}\\&\leq 
		 c\,h\sum_{K\in \mathcal{T}_h;K\subseteq S_\gamma}{\int_{\gamma}{\phi_{\abs{\nabla \bu}}\big(\vert\nabla\bu -(\Pia \bfX)|_K\vert\big)	\,\mathrm{d}s}}
		 \\&\leq
		 c\,h\sum_{K\in \mathcal{T}_h;K\subseteq S_\gamma}{\int_{\gamma}{\abs{\bF(\nabla\bu) -\bF((\Pia \bfX)|_K)}^2	\,\mathrm{d}s}} 
		 \\&\leq
		 c\,h^2 \,\int_{S_\gamma}{\abs{\nabla\bF(\nabla\bu)}^2	\,\mathrm{d}x}+c\,\int_{S_\gamma}{\abs{\bF(\nabla\bu) -\bF(\Pia \bfX)}^2	\,\mathrm{d}x}\,. 
	\end{align*}
	Then, the assertion  follows by summing with respect to $\gamma\in \smash{\Gamma_h^{i}}\cup\Gamma_D$ and resorting to Lemma \ref{lem:e7} with $j=0$.
\end{proof}
\begin{theorem}
  \label{thm:error}
  Let $\AAA$ satisfy Assumption~\ref{ass:1} for a balanced
   {N-func\-tion} ${\phi}$.~\mbox{Moreover}, let $\bu \in W^{1,\phi}(\Omega)$ be a solution of \eqref{eq:p-lap} which satisfies
  $\bF(\nabla \bu) \in W^{1,2}(\Omega)$ and let $\bu_h \in \Vhk $ be a
  solution of \eqref{eq:primal} for $\alpha>0$ and $k\in \mathbb{N}$.  Then, we have
  that
  \begin{align*}
    \begin{split}
       \bignorm{\bfF\big(\nablaDG \bfu_h \!+\!\Rhk \bfu_D^*\big) \!-\! \bfF(\nabla
        \bfu)}_2^2 \!+\! \alpha\,m_{\phi_{\aaal},h } (\bfu_h\!-\!\bu)
      \!\le\! c\, h^2 \norm{\nabla \bF(\nabla \bu) }_2^2 
    \end{split}
  \end{align*}
  with a constant $c\!>\!0$ depending only on the characteristics of
  $\AAA$~and~$\phi$, $k \in \setN$, the~chunkiness $\omega_0>0$, and
  $\alpha^{-1}>0$.
\end{theorem}
\begin{proof}
  To shorten the notation, we use again $\bL_h$ instead of
  $\nablaDG \bfu_h +\Rhk \bfu_D^*$~in~the~shifts.  Let
  $\bfz_h = \bfu_h - \PiDG\bfu\in \Vhk$, then we have that
  \begin{align*}
    \nablaDG \bfz_h = \big(\nablaDG \bfu_h + \Rhk \bfu_D^* - \nabla \bfu\big) -
                      \big(\nablaDG \PiDG\bfu + \Rhk \bfu_D^* - \nabla \bfu\big)\,,
  \end{align*}
  so the error equation~\eqref{eq:errorprimal} gives us
  \begin{align}
  	\begin{aligned}
    &\bighskp{\AAA(\nablaDG \bfu_h +\Rhk \bfu_D^*) - \AAA(\nabla
      \bfu)}{(\nablaDG \bfu_h + \Rhk \bfu_D^*) - \nabla \bfu}
    \\
    &\quad + \alpha \bigskp{\AAA_{\aaa} (h^{-1} \jump{(\bfu_h -\bu_D^*)\otimes
      \bfn})}{ \jump{(\bfu_h -\bu_D^*) \otimes \bfn}}_{\smash{\Gamma_h^{i}}\cup
      \Gamma_D}
    \\
    &= \bighskp{\AAA(\nablaDG \bfu_h +\Rhk \bfu_D^*) - \AAA(\nabla
      \bfu)}{(\nablaDG \PiDG\bu + \Rhk \bfu_D^*) - \nabla \bfu}
    \\
    &\quad + \alpha \bigskp{\AAA_{\aaa} (h^{-1} \jump{(\bfu_h -\bu_D^*)\otimes
      \bfn})}{ \jump{(\PiDG\bfu -\bu_D^*) \otimes \bfn}}_{\smash{\Gamma_h^{i}}\cup
      \Gamma_D}
    \\
    &\quad +{\bigskp{\bigavg{\PiDG\bA}-\avg{\bA}}{\jump{(\bu_h-\PiDG\bu)\otimes
      \bn}}_{\smash{\Gamma_h^{i}}\cup \Gamma_D}}
    \\
    &=:K_1 +\alpha\, K_2+K_3\,. 
\end{aligned}\label{eq:e0}
  \end{align}
  Due to Proposition \ref{lem:hammer}, the first term on the
  left-hand side of \eqref{eq:e0} is equivalent~to
  \begin{align}
    \smash{\rho_{\phi,\Omega}\big({\nablaDG \bfu_h + \Rhk \bfu_D^*-\nabla \bu}\big)
    \sim \bignorm{\bfF\big(\nablaDG \bfu_h +\Rhk \bfu_D^*\big) - \bfF(\nabla
    \bfu)}_2^2\,,}\label{eq:e0.1}
  \end{align}
  while \eqref{eq:equi2} yields that the second term
  on the left-hand side~of~\eqref{eq:e0}  is equivalent~to
  \begin{align}
    \smash{\alpha\,m_{\phi_{\aaa},h } (\bfu_h-\bu_D^*) \,.}\label{eq:e0.2}
  \end{align}
  It is important to use that $\jump{\bu_D^*\otimes \bn} =\jump{\bu\otimes \bn}$ on
  $\smash{\Gamma_h^{i}}\cup \Gamma_D$, which allows to replace~in \eqref{eq:e0.2} $\bu_h-\bu_D^*$ by
  $\bu_h- \bu$.  Thus, \eqref{eq:e0.1} and \eqref{eq:e0.2} yield that the left-hand side of
\eqref{eq:e0} is bounded from~below~by
  \begin{align}\label{eq:lhs}
    c\, \bignorm{\bfF\big(\nablaDG \bfu_h +\Rhk \bfu_D^*\big) - \bfF(\nabla  \bfu)}_2^2
    + \alpha\,c\,m_{\phi_{\aaa},h } (\bfu_h-\bu)\,.
  \end{align}
  To treat the term $K_1$, we exploit that, owing to $\Rhk\bu =\Rhk \bu_D$, it holds
  \begin{align}
    \nablaDG \PiDG \bfu + \Rhk \bfu_D^* -\nabla \bu \!=\! (\nabla_h
    \PiDG \bfu -\nabla \bu )+ \Rhk(\bfu- \PiDG\bfu)\,.\label{eq:fe}   
  \end{align}
  Thus,
  Proposition~\ref{lem:hammer}, the $\varepsilon$-Young inequality \eqref{ineq:young}
  with $\psi=\phi_{\abs{\nabla \bu}}$,  Lemma~\ref{pro:SZnablaF},
  and Lemma \ref{lem:e4} yield that for all $\vep >0$, there exists
  $c_\vep>0$ such that
  \begin{align}\label{eq:e2}
    \begin{aligned}
      \abs{K_1} &\le \vep \,\rho_{\phi_{\abs{\nabla
            \bu}},\Omega}\big({\nablaDG \bfu_h + \Rhk \bfu_D^*-\nabla
        \bu}\big) +c_\vep \, \rho_{\phi_{\abs{\nabla
            \bu}},\Omega}\big({\nabla_h \PiDG\bfu -\nabla \bu}\big)
      \\
      &\quad +c_\vep \, \rho_{\phi_{\abs{\nabla \bu}},\Omega}\big({
        \Rhk (\bfu-\PiDG \bu)}\big) 
      \\[-0.5mm]
      &\le \vep \, c\, \bignorm{\bfF\big(\nablaDG \bfu_h +\Rhk \bfu_D^*\big) -
        \bfF(\nabla \bfu)}_2^2 + c_\vep\, \bignorm{\bfF\big(\nabla_h
        \PiDG\bfu \big) - \bfF(\nabla \bfu)}_2^2
      \\
      &\quad +c_\vep \, \rho_{\phi_{\abs{\nabla \bu}},\Omega}\big({
        \Rhk (\bfu-\PiDG \bu)}\big) 
      \\[-0.5mm]
      &\le \vep \, c\, \bignorm{\bfF\big(\nablaDG \bfu_h +\Rhk \bfu_D^*\big) -
        \bfF(\nabla  \bfu)}_2^2 + c_\vep \, h^2\, \norm{\nabla \bF(\nabla
        \bu)}_2^2  \,.
    \end{aligned}       
  \end{align}
  The identity $\jump{\bu_D^*\otimes \bn}=\jump{\bu\otimes \bn}$ on $\smash{\Gamma_h^{i}}\cup\Gamma_D$,
  \eqref{eq:flux}, the $\varepsilon$-Young inequality~\eqref{ineq:young}~with $\psi=\phi_{\aaa}$, a shift change in Lemma \ref{lem:change2},  Corollary \ref{cor:PiDGapproxmspecial}~with~$\bfw_h=\PiDG\bfu -\bu$, Lemma \ref{lem:e5}, \eqref{eq:hammera} and Lemma \ref{pro:SZnablaF}  yield
  \begin{align}\label{eq:e6}
    \begin{aligned}
      \abs{K_2} &= \bigabs{\bigskp{\AAA_{\aaa} (h^{-1} \jump{(\bfu_h
            -\bu)\otimes \bfn})}{ \jump{(\PiDG\bfu -\bu) \otimes
            \bfn}}_{\smash{\Gamma_h^{i}}\cup \Gamma_D}}
      \\
      &\le \vep\, \smash{m_{\phi_{\aaa},h}(\bfu_h -\bu)}
      +c_\vep \, \smash{m_{\phi_{\aaa},h}\big(\PiDG\bfu -\bu\big)}
        \\
      &\le \vep\, \smash{m_{\phi_{\aaa},h}(\bfu_h -\bu)}
      +c_\kappa\,c_\vep \, \smash{m_{\phi_{\abs{\nabla\bfu}},h}\big(\PiDG\bfu -\bu\big)}
      \\
      &\quad+\kappa\, c_\vep \, h\,{\rho _{\phi_{\abs{\nabla\bfu}},\smash{\Gamma_h^{i}}\cup\Gamma_D}\big(\abs{\abs{\nabla\bfu}-\aaal}\big)}
       \\
      &\le \vep\, \smash{m_{\phi_{\aaa},h}(\bfu_h -\bu)}+c_\kappa\,c_\vep \,
      \bignorm{\bF(\nabla \bu) -\bF\big(\nabla \PiDG\bfu \big)}_2^2
      \\
      &\quad+
      c_\kappa\, c_\vep \,  h^2 \norm{\nabla \bF(\nabla \bu) }_2^2 + \kappa\,c_\vep \, 
      \bignorm{\bF(\nabla \bu) -\bF\big(\Lh\big)}_2^2\,.
      \\
      &\le \vep\, \smash{m_{\phi_{\aaa},h}(\bfu_h -\bu)} +
      c_\kappa\,c_\vep \, h^2 \norm{\nabla \bF(\nabla \bu) }_2^2 
      \\
      &\quad+ \kappa\,c_\vep \, 
      \bignorm{\bF(\nabla \bu) -\bF\big(\Lh\big)}_2^2\,.
    \end{aligned}       
  \end{align} 
  To treat $K_3$, we use the $\varepsilon$-Young inequality~\eqref{ineq:young} for
  $\psi =\phi_{\abs{\nabla \bu}}$ 
  to get\footnote{We
    use $\bF(\nb\bu)\in W^{1,2}(\Omega)$, the embedding
    $W^{1,2}(\Omega)\vnor L^2(\smash{\Gamma_h^{i}}\cup\Gamma_D)$ and
    \eqref{eq:hammerg} to conclude that
    $\nb \bu \in L^\phi(\smash{\Gamma_h^{i}}\cup\Gamma_D)$ and
    $ \bA \in L^{\phi^*}(\smash{\Gamma_h^{i}}\cup\Gamma_D)$ and, thus, have
    well-defined traces on $\smash{\Gamma_h^{i}}\cup\Gamma_D$.\vspace*{-1cm}}
  \begin{align}
    \label{eq:e9}
    \abs{K_3}
    &\le c_\vep\,h\,\int_{\smash{\Gamma_h^{i}}\cup \Gamma_D} (\phi_{\abs{\nabla
      \bu}})^* \big (\big|\avg{\bA} -\bigavg{\PiDG\bA}\big|\big ) \, \textrm{d}s
      +\vep \,m_{\phi_{\abs{\nabla \bu}},h}\big (\bfu_h -\PiDG \bu\big ) 
    \\
    &=: c_\vep \, \sum_{\gamma \in \smash{\Gamma_h^{i}}\cup \Gamma_D} K^\gamma_{3,1} +\vep\, K_{3,2}\,. \notag 
  \end{align}
	Using a~shift~change~in~Lemma~\ref{lem:change2},  Corollary
        \ref{cor:PiDGapproxmspecial} with $\bfw_h=\bfu -\PiDG \bu$,
        Proposition~\ref{lem:hammer}, Lemma \ref{pro:SZnablaF} and Lemma \ref{lem:e5}, we find that
  \begin{align}
    \label{eq:e10}
    \begin{aligned}
      \abs{K_{3,2}}  &\le   c\, m_{\phi_{\abs{\nabla \bu}},h} \big(\bfu
      -\PiDG \bu\big)+c\,m_{\phi_{\abs{\nabla \bu}},h} \big(\bfu_h - \bu\big)
      \\
      &\le   c\, m_{\phi_{\abs{\nabla \bu}},h} \big(\bfu -\PiDG
      \bu\big)+c\,m_{\phi_{\aaa},h} \big(\bfu_h - \bu\big)
      \\[-1mm]
      &\quad +c\, h
      \,\rho _{\phi_{\abs{\nabla\bfu}},\smash{\Gamma_h^{i}}\cup\Gamma_D}\big(\abs{\abs{\nabla\bfu}-\aaal}\big)
      \\
      &\leq c\, h^2 \norm{\nabla \bF(\nabla \bu) }_2^2 + c\,
      \bignorm{\bF(\nabla \bu) -\bF\big(\Lh\big)}_2^2\,.
    \end{aligned}
  \end{align}
  From 
  Proposition \ref{lem:hammer}, Lemma \ref{lem:F-F^*-diff},
  choosing some $K\in \mathcal T_h$
  such~that~${\gamma \subseteq \pa K}$, using $\nb \bu, \bA \in W^{1,1}(\Omega)$, the trace inequality
  \eqref{eq:pol-trace}, \Poincare's inequality on $K$, the
  stability~of~$\PiDG$~in \eqref{eq:PiDGLpsistable}, \eqref{eq:F-F*2},
  \eqref{eq:hammera}, and again \Poincare's inequality on $K$, it~follows~that
  \begin{align}
    \begin{aligned}
    \abs{K_{3,1}^\gamma}
    &\le c\, h\int_{\gamma} \bigabs{\bF^*
      ({\AAA(\nabla \bu)} ) -\bF^*\big({\PiDG \AAA(\nabla \bu) }\big) }^2\,
      \textrm{d}s
    \\[-0.5mm]
    &\le c\, h\int_{\gamma} \bigabs{\bF^*
      (\AAA(\nabla \bu) ) -\bF^*\big(\AAA(\mean{\nabla \bu}_K)\big) }^2 \, \textrm{d}s 
    \\[-1mm]
    &\quad +c\, h\int_{\gamma} \bigabs{\bF^*
      (\AAA(\mean{\nabla \bu}_K)) -\bF^*\big(\PiDG \AAA(\nabla \bu)\big) }^2\, \textrm{d}s 
    \\[-0.5mm]
    &\le c\, h\int_{\gamma} \bigabs{\bF(\nabla \bu )
      -\bF(\mean{\nabla \bu}_K) }^2 \, \textrm{d}s
    \\[-1mm]
    &\quad +c\, h\int_{\gamma} 
      (\phi^*)_{\abs{\AAA(\mean{\nabla \bu}_K)}} \big (\PiDG (\AAA
      (\nabla \bu)-\AAA(\mean{\nabla \bu}_K))\big ) \, \textrm{d}s 
    \\[-0.5mm]
    &\le c\int_{K} \bigabs{\bF(\nabla \bu )
      -\bF(\mean{\nabla \bu}_K) }^2 +h^2\, \abs{\nabla \bF(\nabla
      \bu ) }^2 \, \textrm{d}x
    \\[-1mm]
    &\quad + c\int_K  (\phi^*)_{\abs{\AAA(\mean{\nabla \bu}_K)}} \big (\PiDG (\AAA
      (\nabla \bu)-\AAA(\mean{\nabla \bu}_K))\big )\, \textrm{d}x
    \\[-0.5mm]
    &\le c\int_{K} h^2\, \abs{\nabla \bF(\nabla
      \bu ) }^2 \, \textrm{d}x + c\int_K  (\phi^*)_{\abs{\AAA(\mean{\nabla \bu}_K)}} \big (\AAA
      (\nabla \bu)-\AAA(\mean{\nabla \bu}_K)\big )\, \textrm{d}x
     \\[-0.5mm]
    &\le c\int_{K} h^2\, \abs{\nabla \bF(\nabla
      \bu ) }^2 \, \textrm{d}x+c\int_{K} \abs{\bF(\nabla \bu )
      -\bF(\mean{\nabla \bu}_K) }^2 \, \textrm{d}x 
    \\[-0.5mm]
    &\le c\int_{K} h^2\, \abs{\nabla \bF(\nabla
      \bu ) }^2 \, \textrm{d}x\,.
   \end{aligned}\hspace*{-7mm} \label{eq:e11}
  \end{align}
  Eventually \eqref{eq:e9}--\eqref{eq:e11} imply that 
  \begin{align}
    \abs{K_3} & \le c_\vep\,  h^2 \norm{\nabla \bF(\nabla \bu) }_2^2 + \vep\,c\,
		\bignorm{\bF(\nabla \bu) -\bF\big(\Lh\big)}_2^2\,. 
  \end{align}
  All together, choosing first $\vep>0$ and than $\kappa>0$
  small enough, and absorbing the terms
  with $\vep $ and $\kappa$ in the left-hand side, we conclude the assertion.
\end{proof}
\enlargethispage{3mm}
\begin{corollary}
  \label{cor:error}
  Let $\AAA$ satisfy Assumption~\ref{ass:1} for a balanced
  {N-func\-tion}~${\phi}$.~Moreover, let $\bu \in W^{1,\phi}(\Omega)$ be a solution of
  \eqref{eq:p-lap} which satisfies
  $\bF(\nabla \bu) \in W^{1,2}(\Omega)$ and let
  $(\bu_h,\bL_h,\bA_h)^\top \in \Vhk \times \Xhk \times \Xhk$ be a solution
  of \eqref{eq:DG} for $\alpha>0$, $h>0$ and $k\in \mathbb{N}$. Then,~it~holds
  \begin{align}
    \begin{split}
      \norm{\bfF^*(\AAA(\nabla \bfu)) - \bfF^*(\bA_h)}_2^2 \le c\,
      h^2\, \norm{\nabla \bF(\nabla \bu) }_2^2 \,,
    \end{split}\label{eq:ba1}
  \end{align}
  with a constant $c>0$ depending only on the characteristics of
  $\AAA$~and~$\phi$, $k \in \setN$, the chunkiness $\omega_0>0$, and
  $\alpha^{-1}>0$.
\end{corollary}
\begin{proof} 
  Lemma \ref{lem:e7} with
  $\bY=\bL_h$, $j=k$, $\bA_h = \PiDG \AAA(\bL_h)$, \eqref{eq:F-F*3},
  \eqref{eq:Lh} imply that
  \begin{align*}
    \norm{\bfF^*(\AAA(\nabla \bfu)) - \bfF^*(\bA_h)}_2^2
    &\le c\, \norm{\bfF^*(\AAA(\nabla \bfu)) - \bfF^*(\AAA(\bL_h))}_2^2 
    \\&\quad+ c\,  h^2\, \norm{\nabla \bF(\nabla \bu) }_2^2
    \\
    &\le c\, \bignorm{\bfF(\nabla \bfu) - \bfF\big(\Lh\big)}_2^2 
    \\&\quad+ c\,  h^2\, \norm{\nabla \bF(\nabla \bu) }_2^2\,,
  \end{align*}
  which together with Theorem \ref{thm:error} yields the assertion.
  Alternatively, we~could~use \mbox{\cite[Proposition 4.9]{dkrt-ldg}} and Theorem \ref{thm:error}.
\end{proof}

\begin{remark}\label{rem:dis}
  Let us compare our results in the special case that $\phi$ possesses
  $(p,\delta)$-structure and $k=1$ with the corresponding ones in \cite{dkrt-ldg}.
  \begin{itemize}[noitemsep,topsep=1pt,leftmargin=\widthof{(iiiv)},labelwidth=\widthof{(iii)}] 
  \item [(i)] \cite[Theorem 4.8, Corollary 4.10 (ii)]{dkrt-ldg} provides
    sub-optimal convergence rates for $p\le 2$ as well as for $p\ge 2$, while
    Theorem \ref{thm:error} and Corollary \ref{cor:error} prove
    optimal ones for all $p\in (1,\infty)$. We emphasize that in
    \cite{dkrt-ldg}, the cases $p\le 2$ and $p\ge 2$ are treated
    differently, while our approach provides a unified treatment. The
    reason for these differences are the different fluxes $\widehat \bA$. Due to the shift in~our~new~flux, it fits perfectly with the structure of the problem
    (cf.~Proposition~\ref{lem:hammer} and treatment of the term $K_3$
    in the proof of Theorem \ref{thm:error}). This is in analogy to
    the gradient shift in the natural distance (cf.~Remark
    \ref{rem:natural_dist}).
  \item [(ii)] \cite[Theorem 4.3, Theorem 4.5, Corollary 4.10
    (i)]{dkrt-ldg} treat the case
    $\bu_D^*=\PiSZ \bu\in \Vhk\cap W^{1,p}(\Omega)$, where
    $\bu\in W^{1,p}(\Omega)$ is the solution of \eqref{eq:p-lap}. In
    this case, an optimal convergence rate is proved for $p\le 2$, while
    for $p\ge 2$ only sub-optimal results are provided. Inasmuch as the solution
    $\bu$ is a priori unknown~and,~in~general,
    $\PiSZ \bu \neq \bu =\bu_D$ on $\Gamma_D$, these results are of
    rather theoretical interest than of interest from the practical
    computational point of view.
  \item [(iii)]   The assertions of Theorem \ref{thm:error} and Corollary
  \ref{cor:error} also hold if we replace the extension $\bu_D^*\in W^{1,\phi}(\Omega)$ of
  the boundary datum $\bu_D$ by the approximation $\PiSZ
  \bu\in \Vhk\cap W^{1,\phi}(\Omega)$ of the solution $\bu\in W^{1,\phi}(\Omega)$, i.e., we define ${\bu_D^* \vcentcolon=\PiSZ \bu}$. 
  In this case, we obtain
  $$
    \bu_h-\bu_D^*= (\bu_h-\bu) +\big(\bu-\PiSZ\bu\big)\quad\text{ in }\Vhk\,.
  $$
  Thus, we can, once more, replace $\bu_h-\bu_D^*\in \Vhk$ by $\bu_h-
  \bu\in W^{1,\phi}(\Omega)$ in the modular 
  $\smash{m_{\phi_{\aaa},h } (\bfu_h-\bu_D^*)}$ if we add on the right-hand
  side~of~\eqref{eq:e0}~the~term
  \begin{align}
    \label{eq:e1}
  \alpha\,c\, m_{\phi_{\aaa},h } \big(\bfu-\PiSZ\bu\big)\,.
  \end{align}
Similarly, we get instead of \eqref{eq:fe} 
  \begin{align*}
  	\nablaDG \PiDG \bfu + \Rhk \bfu_D^* -\nabla \bu
    &= (\nabla_h \PiDG \bfu
      -\nabla \bu )+ \Rhk(\bfu- \PiDG\bfu)\\&\quad + \Rhk\big(\PiSZ\bfu- \bfu\big),
  \end{align*}
  which amounts in an additional term
  \begin{align}\label{eq:e1a}
    \rho_{\phi_{\abs{\nb \bu}}}\big (\Rhk (\PiSZ\bfu- \bfu)\big)\,.
  \end{align}
  Consequently, in the case
  $\bu_D^* =\PiSZ \bu\in\Vhk\cap W^{1,\phi}(\Omega)$, we have to treat
  additionally the terms in \eqref{eq:e1} and
  \eqref{eq:e1a}. \hspace{-0.2em}However, the Scott--Zhang
  interpolation operator~$\PiSZ$~has~\mbox{similar} properties as the
  projection operator $\PiDG$. Thus, we can handle these terms in the
  same way as the corresponding terms in the proof of Theorem
  \ref{thm:error}. Thus, our approach also gives for ${\bu_D^* =\PiSZ \bu}$ optimal
  convergences rates for all balanced N-functions $\phi$.
  \end{itemize}
\end{remark}

\section{Numerical experiments}
\label{sec:experiments}

In this section, we apply the LDG scheme \eqref{eq:DG} (or \eqref{eq:primal}), described above, to solve
numerically the system \eqref{eq:p-lap} with balanced Orlicz-structure with the nonlinear operator $\AAA:\setR^{d\times d}\to\setR^{d\times d}$, for every $\bP\in\setR^{d\times d}$ defined by 
\begin{align*}
	\AAA(\bP) \vcentcolon=(\delta+\vert \bP\vert)^{p-2}\ln(1+\delta+\vert \bP\vert)\bP\,,
\end{align*}
where $\delta\vcentcolon=1\textrm{e}{-}3$ and $ p\in (1,\infty)$, i.e., the operator $\bfF:\setR^{d\times d}\to\setR^{d\times d}$ for every $\bP\in\setR^{d\times d}$ is defined by  
\begin{align*}
	\bfF(\bP) \vcentcolon=(\delta+\vert \bP\vert)^{\frac{p-2}{2}}\sqrt{\ln(1+\delta+\vert \bP\vert)}\,\bP\,.
\end{align*}
We approximate the discrete solution~${\bu_h\in U^k_h}$ of the nonlinear problem \eqref{eq:DG} deploying the Newton line search algorithm of \textsf{PETSc} (version 3.16.1),~cf.~\cite{PETSc19}, with an absolute tolerance of $\tau_{abs}=1\textrm{e}{-}8$ and a relative tolerance of $\tau_{rel}=1\textrm{e}{-}10$.~The~linear system emerging in each Newton step is solved deploying \textsf{PETSc}'s preconditioned biconjgate gradient stabilized method (BCGSTAB) with an~incomplete~LU~factorizaton. For~the numerical flux \eqref{def:flux-A}, we choose the parameter $\alpha>0$ according to Table
\ref{tb0} as a function of~${p\in (1,\infty)}$. This choice is in accordance~with~the~choice~in~\mbox{\cite[Table~1]{dkrt-ldg}}.

\begin{table}[H]
	\centering 
	\begin{tabular}{|c| |c|c|c|c|c|c|c|c|c|c|} 
		\hline 
		&\multicolumn{10}{|c|}{$p$ } \\ [0.5ex] \hline 
		--         & 1.25 &4/3& 1.5 & 5/3 &1.8 &  2  &2.25&  2.5 & 3   &  4   \\ \hline\hline
		$\alpha$  & 0.06 &0.1 &0.2  &0.5 &1 &    2  &2   &  2.5 & 2.5 &  2.5    \\ \hline
	\end{tabular}
	
	\caption{Choice of the stabilization parameter
		$\alpha>0$.\vspace*{-0.5cm}
	}
	\label{tb0}
\end{table}

All experiments were carried out using the finite element software package~\mbox{\textsf{FEniCS}} (version 2019.1.0), cf.~\cite{LW10}. All graphics are generated using the \textsf{Matplotlib} library (version 3.5.1), cf.~\cite{Hun07}.

For our numerical experiments, we choose $\Omega= (-2,2)^2$, $\Gamma_D=\pa\Omega$, ${\Gamma_N=\emptyset}$,~and linear elements, i.e., $k\hspace*{-0.1em}=\hspace*{-0.1em}1$. For $\beta\hspace*{-0.1em}=\hspace*{-0.1em}0.01$, we choose $\smash{\bfg\!\in\! L^{p'}(\Omega)}$ and $\smash{\bu_D\!\in\! W^{1-\frac{1}{p},p}(\Gamma_D)}$ such that $\bu\in W^{1,p}_{\Gamma_D}(\Omega)$, for every $x\vcentcolon=(x_1,x_2)^\top\in \Omega$ defined by
\begin{align}
	\bu(x)\vcentcolon=\vert x\vert^\beta(x_2,-x_1)^\top,
\end{align}
is a solution of  \eqref{eq:p-lap} and satisfies $\bF(\nb\bu)\!\in\! W^{1,2}(\Omega)$. \!We construct~a~\mbox{starting}~\mbox{triangula}-tion $\mathcal
T_{h_0}$ with $h_0\!=\!1$ by subdividing a rectangular cartesian~grid~into~regular~triangles with different orientations.  Finer triangulations $\mathcal T_{h_i}$, $i=1,\dots,5$, with $h_{i+1}=\frac{h_i}{2}$ for all $i=1,\dots,5$, are 
obtained by
regular subdivision of the previous~grid:~Each~\mbox{triangle} is subdivided
into four equal triangles~by~connecting~the~midpoints~of~the~edges. 

Then, for the resulting series of triangulations $\mathcal T_{h_i}$, $i\!=\!1,\dots,5$, we apply~the~above Newton scheme to compute the corresponding numerical solutions $(\bu_i,\bL_i,\bA_i)^\top\vcentcolon=\smash{(\bu_{h_i},\bL_{h_i},\bA_{h_i})^\top\in U_{h_i}^k\times X_{h_i}^k\times X_{h_i}^k}$, $i=1,\dots,5$, 
and the error quantities
\begin{align*}
	\left.\begin{aligned}
		e_{\bL,i}&\vcentcolon=\|\bF(\bL_i)-\bF(\bL)\|_2\,,\\
		e_{\jump{},i}&\vcentcolon=m_{\phi_{\avg{\abs{\bL_i}}},h_i}(\bu_i-\bu)^{\smash{\frac{1}{2}}}\,,
	\end{aligned}\quad\right\}\quad i=1,\dots,5\,.
\end{align*}
As estimation of the convergence rates,  the experimental order of convergence~(EOC)
\begin{align*}
	\texttt{EOC}_i(e_i)\vcentcolon=\frac{\log(e_i/e_{i-1})}{\log(h_i/h_{i-1})}, \quad i=1,\dots,5\,,
\end{align*}
where for every $i=1,\dots,5$, we denote by $e_i$ either  
$e_{\bL,i}$
or $e_{\jump{},i}$,~\mbox{resp.}, is recorded.
For different values of $p\in \{1.25, 4/3, 1.5, 5/3, 1.8, 2, 2.25, 2.5, 3, 4\}$~and~for a series of triangulations $\mathcal{T}_{h_i}$, $i=1,\dots,5$,
obtained by regular, global refinement as described above with
$h_0=1$, the EOC is computed
and presented in 
Table~\ref{tab1} 
and
Table~\ref{tab3}, resp.  In each case, we observe a convergence~ratio~of about $\texttt{EOC}_i(e_i)\approx 1$, $i=1,\dots, 5$, as predicted by 
Theorem~\ref{thm:error}.

\begin{table}[H]
	\centering 
	\begin{tabular}{|c| |c|c|c|c|c|c|c|c|c|c|} 
		\hline 
		&\multicolumn{10}{|c|}{$p$} \\ [0.5ex] \hline 
		$\frac{h_0}{2^i}$    & 1.25 &4/3  & 1.5 & 5/3 &1.8   &  2   & 2.25 &  2.5  &  3       &  4   \\ \hline\hline
		$i=1$                 			& 0.90 &0.90  &0.90 &0.90 &0.88  & 0.85    & 0.86  &0.85    &  0.87 &  0.90  \\ \hline
		$i=2$                		   &0.94 &0.94  &0.94 &0.93 &0.91  & 0.88      &0.89  &0.89    &  0.90    & 0.93   \\ \hline
		$i=3$                  			& 0.95&0.95 & 0.94 &0.94 &0.93  & 0.91     &0.91  &0.91   & 0.92    & 0.94 \\ \hline
		$i=4  $               			& 0.95 &0.95 &0.95 & 0.95 &  0.93  & 0.92&0.92  &0.92  &0.93  &  0.94\\ \hline
		$i=5$               			& 0.96  & 0.96 & 0.96 & 0.95  &  0.94& 0.93&  0.93&  0.94 & 0.94  & 0.95  \\ \hline 
	\end{tabular}
	\caption{Experimental order of convergence: $\texttt{EOC}_i(e_{\bL,i})$,~${i=1,\dots,5}$, 
		in the~case~of $k=1$, $\Gamma_D=\partial
		\Omega$, $\bu_D^*=\bu\in W^{1,p}_{\Gamma_D}(\Omega)$
		and $\mathbf{F}(\nabla\mathbf{u})\in
		{W}^{1,2}(\Omega)$.\vspace*{-0.5cm}} 
	\label{tab1}
\end{table}
\begin{table}[h]
	\centering 
	\begin{tabular}{|c| |c|c|c|c|c|c|c|c|c|c|} 
		\hline 
		&\multicolumn{10}{|c|}{$p$ } \\ [0.5ex] \hline 
		$\frac{h_0}{2^m}$   & 1.25 & 4/3 & 1.5  &5/3    &   1.8 &  2  & 2.25 & 2.5 & 3      &  4    \\ \hline\hline
		$i=1$              &1.12  & 1.10 & 1.08 &1.05   & 1.04  & 1.03& 1.03 &1.03 & 1.03  & 1.04   \\ \hline
		$i=2$              &1.04  & 1.03 & 1.03 &1.03   & 1.03   & 1.03& 1.03 &1.03 & 1.03   & 1.04  \\ \hline
		$i=3$               &1.02  & 1.02 & 1.02 &1.02   & 1.02   & 1.02& 1.02 & 1.02 & 1.03   & 1.03 \\ \hline
		$i=4$             &1.01  & 1.01 & 1.012&1.01  & 1.02   & 1.02& 1.02 &1.02 & 1.02  & 1.03  \\ \hline 
		$i=5$               	& 1.01  & 1.01 & 1.01 & 1.01  &  1.01& 1.01 &  1.01&  1.02 & 1.02  & 1.02  \\ \hline 
	\end{tabular}
	\caption{Experimental order of convergence: $\texttt{EOC}_i(e_{\jump{},i})$, $i=1,\dots,5$, 
		in the~case~of $k=1$, $\Gamma_D=\partial
		\Omega$, $\bu_D^*=\bu\in W^{1,p}_{\Gamma_D}(\Omega)$ and $\mathbf{F}(\nabla\mathbf{u})\in {W}^{1,2}(\Omega)$.\vspace*{-0.5cm}}
	\label{tab3}
\end{table}

In the Figures \ref{fig:plot_uh}--\ref{fig:plot_Rh}, for $p=4$ and
$\alpha=2.5$, we display $\Uppi_{h_4}^1(\vert \bu_4\vert )$,
$\Uppi_{h_4}^1(\vert \bu_4 -\bu\vert )$,
$\Uppi_{h_4}^1(\vert \bL_4\vert )$,
$\Uppi_{h_4}^1(\vert \bL_4 -\bL\vert
)$, and $\Uppi_{h_4}^1(\vert \bR_{h_4}^1(\bu_4 -\bu)\vert
)$. In it, we clearly observe~that the major proportion of the error
is located near the singularity of the exact solution $\bu\in
W^{1,4}(\Omega)$.  For all other considered~\mbox{values}~${p\!\in\!
  \{1.25, 4/3, 1.5, 5/3, 1.8, 2,2.25, 2.5, 3,
  4\}}$ with associated
$\alpha>0$, according~to~Table~\ref{tb0},
one~get~similar~\mbox{pictures}.\vspace*{-0.5cm}

\begin{figure}[H]
	\centering
	\hspace*{-0.25cm}\includegraphics[width=6cm]{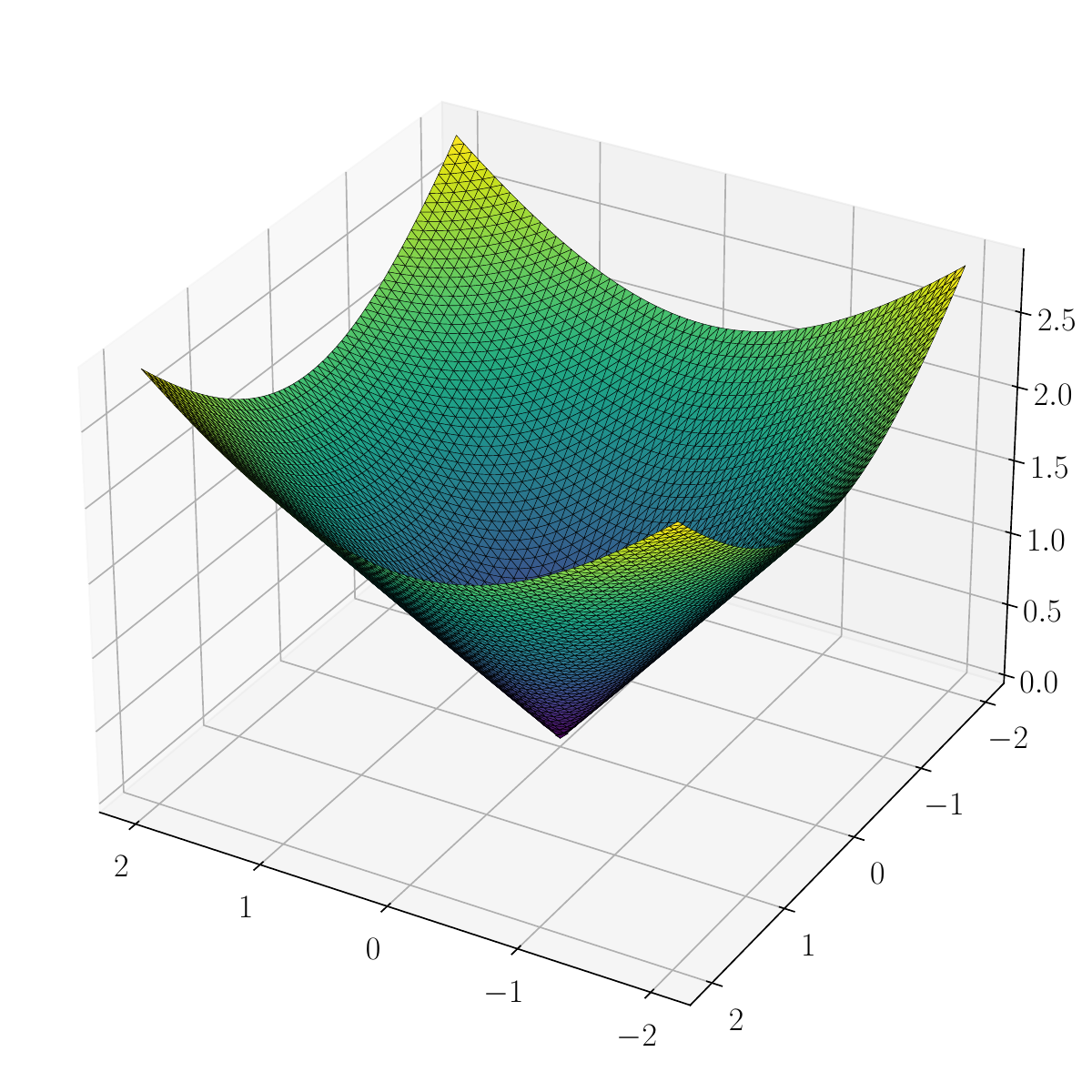}%
	\includegraphics[width=6.25cm]{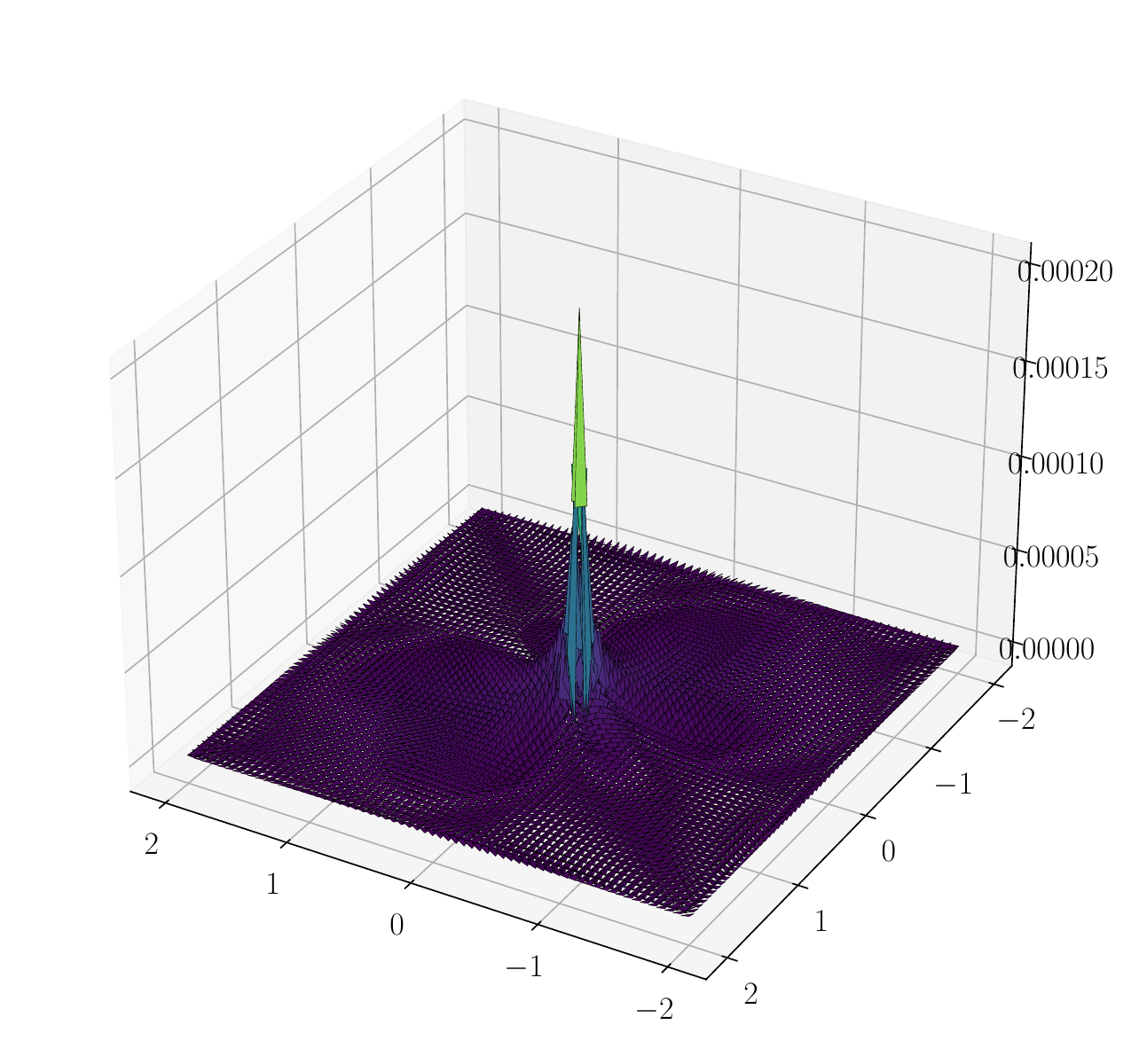}\vspace*{-0.5cm}
	\caption{Plots of  $\Uppi_{h_4}^1(\vert \bu_4\vert )$ (left) and $\Uppi_{h_4}^1(\vert \bu_4 -\bu\vert )$ (right) for ${p=4}$, ${\alpha=2.5}$.}%
	\label{fig:plot_uh}%
\end{figure}\vspace*{-1.25cm}

\begin{figure}[H]
	\centering
	\hspace*{-0.25cm}\includegraphics[width=6cm]{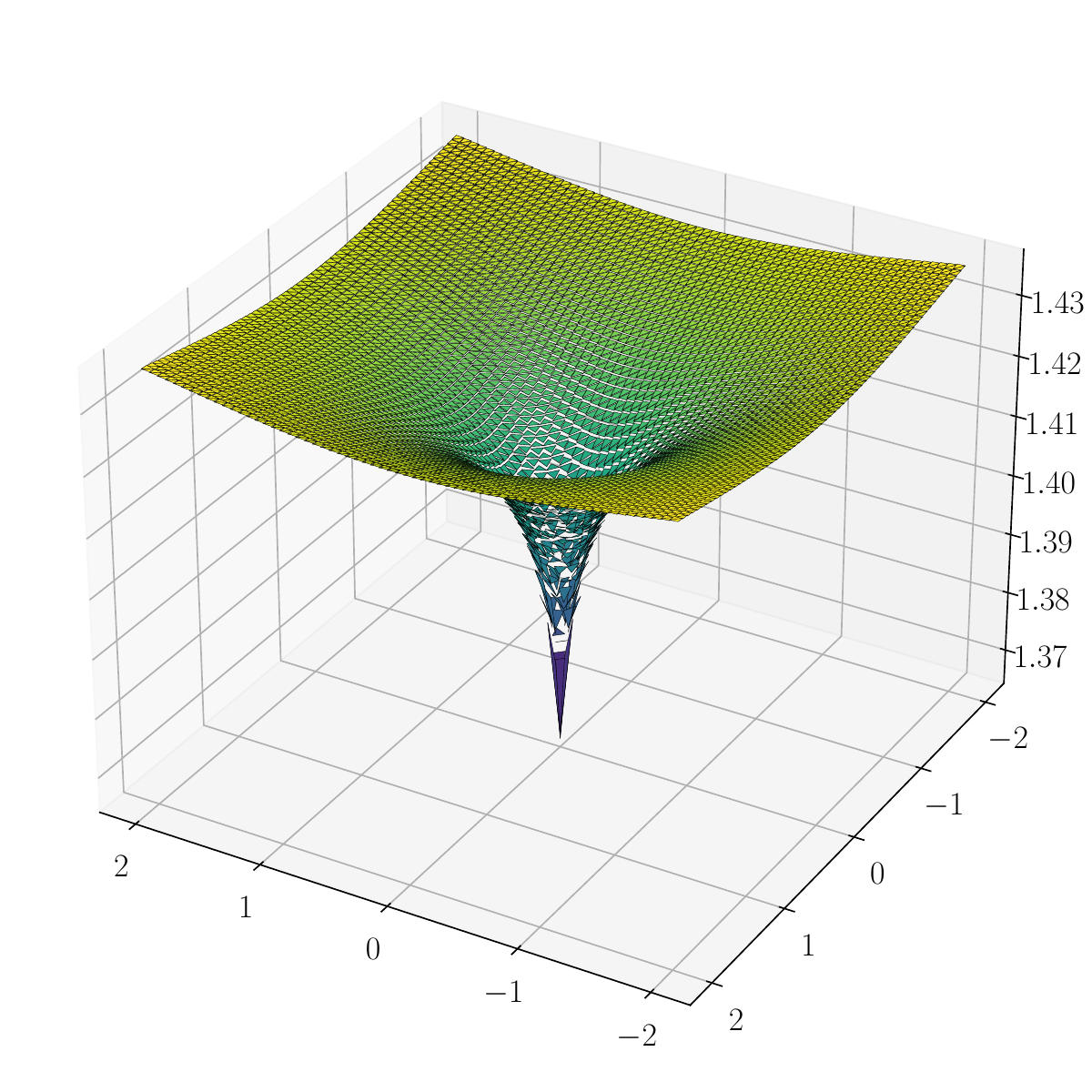}%
	\includegraphics[width=6cm]{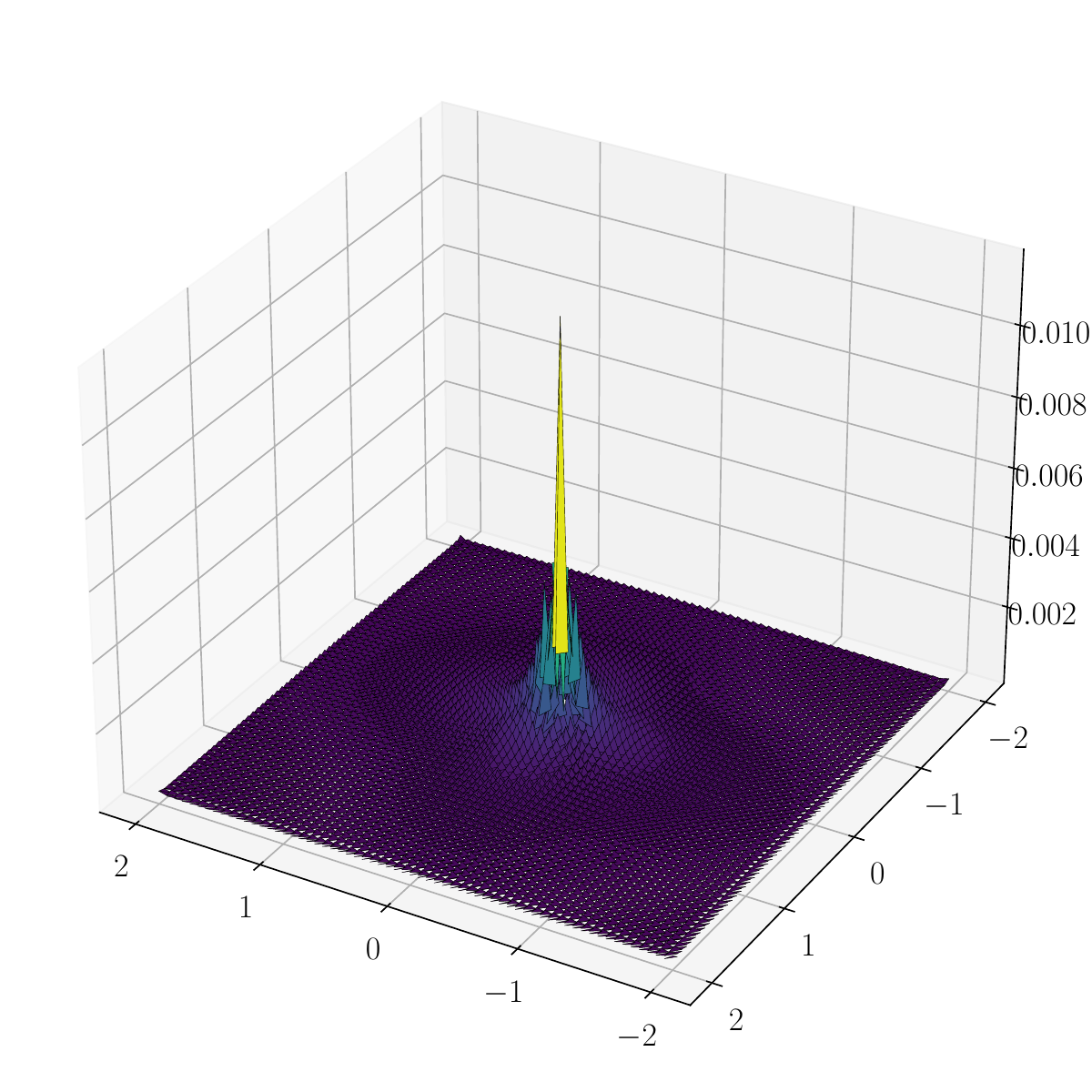}\vspace*{-0.5cm}
	\caption{Plots of $\Uppi_{h_4}^1(\vert \bL_4\vert )$ (left) and  ${\Uppi_{h_4}^1(\vert \bL_4-\bL\vert )}$~(right) for ${p=4}$, ${\alpha=2.5}$.}\vspace*{-1.25cm}
	\label{fig:plot_Lh}%
\end{figure}

\begin{figure}[H]
	\centering
	\hspace*{-0.25cm}\includegraphics[width=6cm]{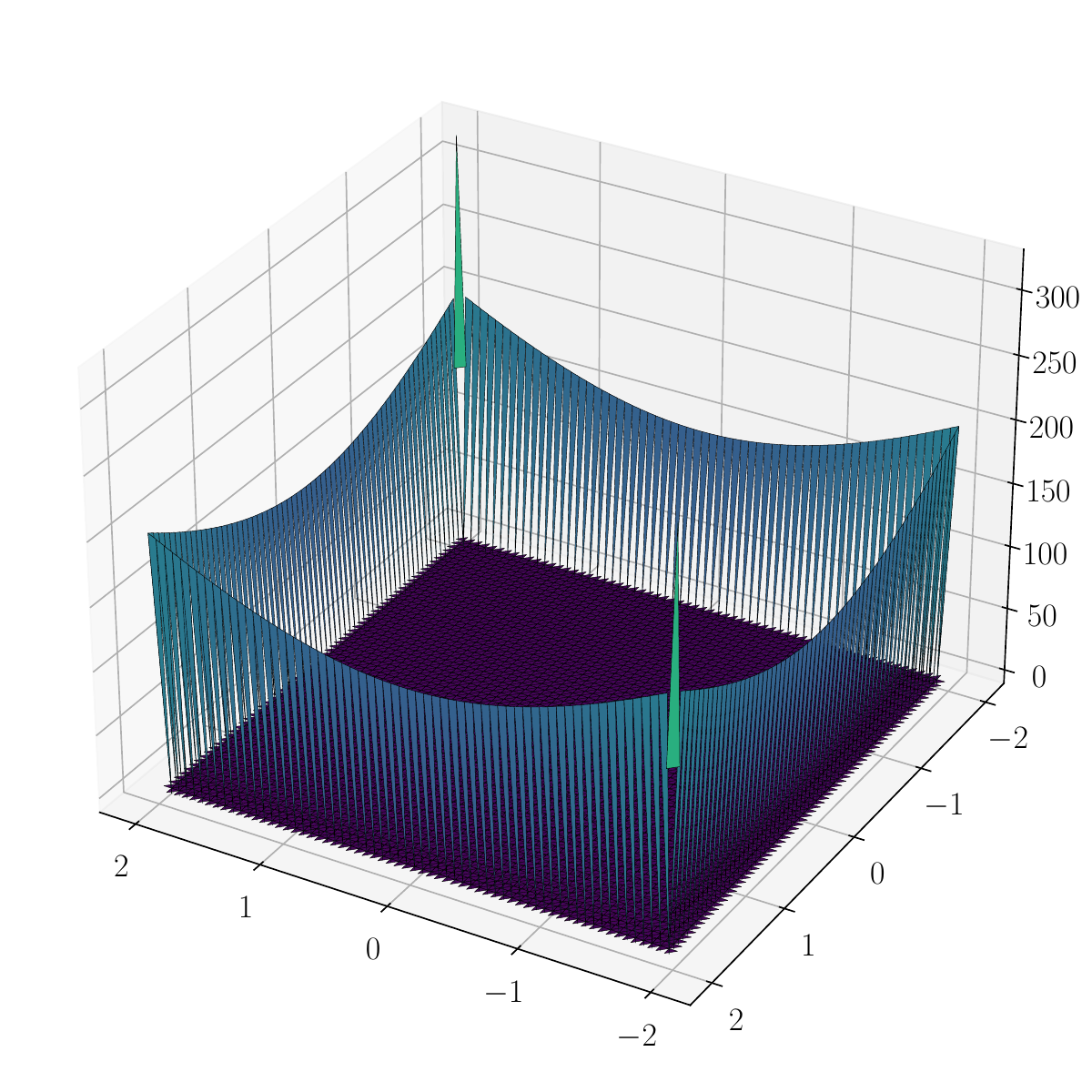}%
	\includegraphics[width=6cm]{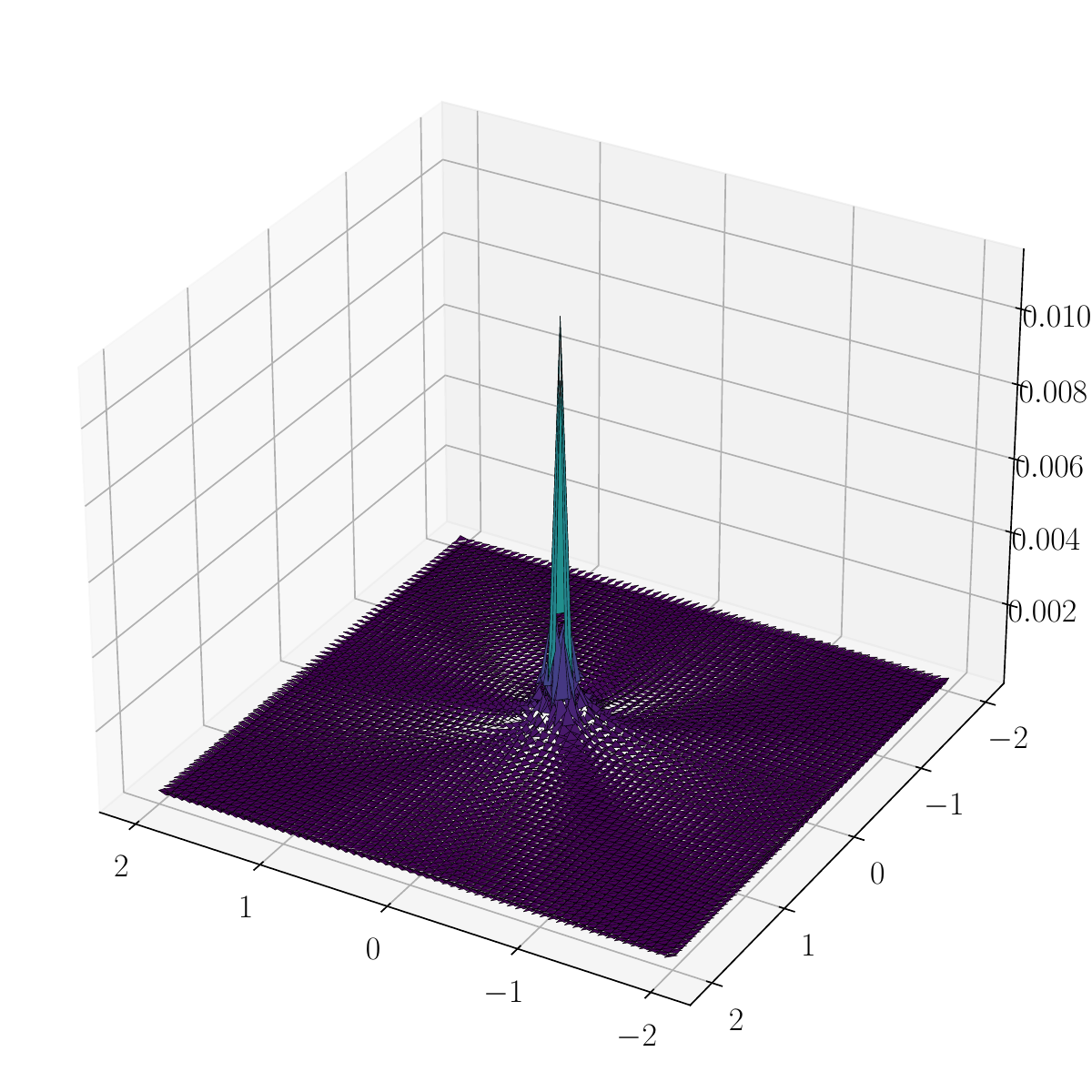}\vspace*{-0.5cm}
	\caption{Plots of  $\Uppi_{h_4}^1(\vert \bR_{h_4}^1\bu_4 \vert )$ (left) and  $\Uppi_{h_4}^1\!(\vert \bR_{h_4}^1(\bu_4 -\bu)\vert )$ (right) for ${p=4}$, ${\alpha=2.5}$.}%
	\label{fig:plot_Rh}%
\end{figure}

\appendix 
\section{}
\label{sec:aux}

This Appendix collects known results, used in the paper,
and proves new results in the DG Orlicz-setting. 

\begin{lemma}\label{lem:stabRhk}
  Let $\psi:\setR^{\ge 0}\to \setR^{\ge 0}$ be an N-function
 satisfying the $\Delta_2$-condition and $k\in \mathbb{N}_0$.  Then, for every
  $\bw_h \in W^{1,\psi}(\mathcal T_h)$, we have that
  \begin{align}
    \label{eq:Rgammaest}
    \begin{aligned}
      \int_{S_\gamma} \psi \big ( \boldsymbol{\mathcal
        R}^k_{h,\gamma}\bw_h\big)\, \textup{d}x
      &\le c\, h\,\int_\gamma \psi  ( h^{-1}\abs{\jump{\bw_h \otimes
          \bn}})\, \textup{d}s\,,
      \\
      \rho_{\psi,\Omega} \big (\Rhk \bw_h\big) &\leq c\,
      m_{\psi,h}(\bw_h)
    \end{aligned}
  \end{align}
  with a constant $c>0$ depending only on $k\in\mathbb{N}_0$, 
  $\Delta_2(\psi)>0$, and $\omega_0>0$.
\end{lemma}
\begin{proof}
	The assertions are proved in \cite[(A.23), (A.25)]{dkrt-ldg}.
\end{proof}
\begin{lemma}
	\label{lem:equiW1psi}
	Let $\psi:\setR^{\ge 0}\to \setR^{\ge 0}$  be an N-function satisfying the $\Delta_2$-condition and $k\in \mathbb{N}_0$.  Then, for every $\bw_h \in \WDGpsi$, we have that 
	\begin{align}\label{eq:equi}
		c^{-1}\,  M_{\psi,h}(\bw_h) \le \rho_{\psi,\Omega}\big(\Ghk \bw_h\big) +
		m_{\psi,h}(\bw_h) \le c\,  M_{\psi,h}(\bw_h)
	\end{align}
	with a constant $c>0$ depending only on $k\in\mathbb{N}_0$, 
        $\Delta_2(\psi)>0$, and $\omega_0>0$.
      \end{lemma}
\begin{proof}
	The assertion follows from \cite[(A.26), (A.28)]{dkrt-ldg} and Lemma
	\ref{lem:stabRhk}.
\end{proof}
\begin{lemma}
	\label{lem:stabpi}
	Let $\psi:\setR^{\ge 0}\to \setR^{\ge 0}$ be an N-function satisfying the $\Delta_2$-condition and $k\in \mathbb{N}_0$.
	Then,
	for every $K \in \mathcal{T}_h$, $\bfu \in W^{\ell,\psi}(K)$ and 
	$0\leq j \leq \ell \leq k+1$,~we~have~that
	\begin{align}
		\label{eq:PiDGapproxpsi}
		\dashint_K \psi\big(h_K^j \vert\nabla^j(\bfu - \PiDG \bfu)\vert\big)\,\textup{d}x &\leq
		c\, \dashint_K \psi\big(h_K^{\ell} \vert \nabla^{\ell} \bfu\vert\big)\,\textup{d}x\,,\\
		\dashint_K \psi\big(h_K^j \vert\nabla^j\PiDG \bfu\vert\big)\,\textup{d}x &\leq
		c\, \dashint_K \psi\big(h_K^{\ell} \vert \nabla^{\ell} \bfu\vert\big)\,\textup{d}x \label{eq:PiDGstabpsi}
	\end{align}
	with a constant $c>0$ depending only on $\ell, k\in
        \mathbb{N}_0$, $\Delta_2(\psi)>0$, and $ \omega_0>0$.
\end{lemma}
\begin{proof}
	The first assertion is shown in \cite[(A.7)]{dkrt-ldg}, and
        the second one follows from the triangle inequality. 
\end{proof}
\begin{corollary}\label{cor:stabpi}
	Let $\psi:\setR^{\ge 0}\to \setR^{\ge 0}$ be an N-function satisfying the $\Delta_2$-condition and $k\in \mathbb{N}_0$. Then,
	for every $\bu \in L^{\psi}(\Omega)$ and $\bw_h \in W^{1,\psi}(\mathcal T_h)$, we have that
	\begin{align}
		\label{eq:PiDGapprox0}
		\rho_{\psi,\Omega}\big(\bu - \PiDG \bu\big) &\leq c\,
		\rho_{\psi,\Omega}(\bu)\,,
		\\
		\label{eq:PiDGapprox1a}
		\rho_{\psi,\Omega}\big(\bw_h - \PiDG \bw_h\big) &\leq c\,
	\rho_{\psi,\Omega}(h\, {\nabla_h \bw_h})\,,
		\\
		\label{eq:PiDGapprox2}
		\rho_{\psi,\Omega}\big(\nabla_h \bw_h - \nabla_h \PiDG \bw_h\big) &\leq c\,
		\rho_{\psi,\Omega}({\nabla_h \bw_h})
	\end{align}
	with a constant $c>0$ depending only on $k\in\mathbb{N}_0$, 
  $\Delta_2(\psi)>0$, and $\omega_0>0$. In~particular, this implies for every $k\in \setN_0$, $\bu \in L^{\psi}(\Omega)$ and $\bw_h \in W^{1,\psi}(\mathcal T_h)$ that
	\begin{align}
		\label{eq:PiDGLpsistable}
		\rho_{\psi,\Omega}\big(\PiDG \bu\big) &\leq c\,
	\rho_{\psi,\Omega}(\bu)\,,
		\\
		\label{eq:PiDGLW1psistable}
		\rho_{\psi,\Omega}\big(\nabla_h \PiDG \bw_h\big) &\leq c\,
		\rho_{\psi,\Omega}(\nabla_h \bw_h)\,.
	\end{align} 
	For every $k \in \setN$ and $\bw_h \in W^{2,\psi}(\mathcal T_h)$, we have that
	\begin{align}
		\label{eq:PiDGapprox1}
		\rho_{\psi,\Omega}\big(\nabla_h \bw_h - \nabla_h \PiDG \bw_h\big) &\leq c\,
		\rho_{\psi,\Omega}\big({h\nabla^2_h \bw_h}\big)
	\end{align}
	with a constant $c>0$ depending only on $k\in\mathbb{N}$, 
  $\Delta_2(\psi)>0$, and $\omega_0>0$.
\end{corollary}


\begin{lemma}
  \label{pro:SZnablaF}
    Let $\AAA$ satisfy Assumption~\ref{ass:1} for a balanced
   {N-func\-tion}~${\phi}$~and~${k\in \mathbb{N}}$. Let $\bfF(\nabla \bfu) \in W^{1,2}(\Omega)$, then
  \begin{align*}
    \bignorm{\bfF\big(\nabla_h \PiDG \bfu\big) - \bfF(\nabla \bfu)}_2^2 &\le c\,
    h^2 \, \norm{\nabla \bF(\nabla \bu)}_{2}^2
  \end{align*}
  with $c\!>\!0$ depending only on the characteristics of $\AAA$ and $\phi$,
  and the chunkiness~$\omega_0$. The same
  assertion also holds for the Scott--Zhang interpolation operator~$\PiSZ$.
\end{lemma}

\begin{proof}
	The assertion is proved in \cite[Corollary 5.8]{dr-interpol}.
\end{proof}

\begin{lemma}
	\label{lem:traceW1psi}
	Let $\psi$ be an N-function satisfying the $\Delta_2$-condition and $k\in \mathbb{N}_0$.
	Let $K \in \mathcal{T}_h$~and $\gamma$ be a face of~$K$. Then, for every
	$\bu \in W^{1,\psi}(K)$ and ${\bu_h \in \mathcal
          P_k(K)}$,  it holds
	\begin{align}\label{eq:emb}
		\dashint_\gamma \psi(\abs{\bu})\,\textup{d}s &\leq c \dashint_K
		\psi(\abs{\bu}) \,\textup{d}x + c \dashint_K \psi(h\abs{\nabla
			\bu}) \,\textup{d}x\,,
		\\
		\dashint_\gamma \psi(\abs{\bu_h})\,\textup{d}s &\le c \,
		\dashint_K \psi(\abs{\bu_h})\,\textup{d}x \,, \label{eq:pol-trace}
	\end{align}
	with  constants $c>0$ depending only on $k\in\mathbb{N}_0$, 
 $\Delta_2(\psi)>0$, and $\omega_0>0$. 
\end{lemma}
\begin{proof}
	The assertions are proved in \cite[(A.13), (A.14)]{dkrt-ldg}.
\end{proof}
\begin{corollary}\label{cor:PiDGapproxmlocal}
	Let  $\psi:\setR^{\ge 0}\to \setR^{\ge 0}$ be an N-function satisfying the $\Delta_2$-condition and $k\in \mathbb{N}_0$.
	Let $K \in \mathcal{T}_h$ and $\gamma$ be a face of~$K$. Then, for
	every $\bu \in W^{1,\psi}(K)$, $\bw_h \in \WDGpsi$ and $\bu_h \in \Vhk$, we have that
	\begin{align}
          \label{eq:PiDGapproxmlocal}
          h \int_\gamma \psi\big(h^{-1} \bigabs{\bfu - \PiDG \bfu}\big)\,\textup{d}s 
          &\leq c \int_K \psi(\abs{\nabla \bfu})\,\textup{d}x\,,
          \\[-1.5mm]
		\label{eq:PiDGapproxmglobal}
          m_{\psi,h}\big(\bw_h - \PiDG \bw_h\big) &\leq c\,\rho_{\psi,\Omega}(\nabla_h \bw_h)\,,
          \\
          \label{eq:PiDGapproxmglobal2}
          m_{\psi,h}(\bw_h )
          &\leq c\,\rho_{\psi,\Omega}\big(h^{-1}
            \bw_h\big)+c\,\rho_{\psi,\Omega}(\nabla_h \bw_h)\,, 
          \\
          \label{eq:PiDGapproxmglobal1}
          h\,\rho_{\psi,\smash{\Gamma_h^{i}}\cup\Gamma_D}(\avg{\bu_h}
          ) &\leq c\,\rho_{\psi,\Omega}( \bu_h)\,,
          \\
          \label{eq:PiDGapproxmglobal3}
          h\,\rho_{\psi,\smash{\Gamma_h^{i}}\cup\Gamma_D}(\avg{\bw_h-\PiDG
          \bw_h} ) &\leq c\,\rho_{\psi,\Omega}( h\nabla _h\bw_h)\,,
	\end{align}
	with  constants $c>0$ depending only on $k\in\mathbb{N}_0$, $\Delta_2(\psi)$, and $\omega_0>0$.
\end{corollary}
\begin{proof}
	For the assertions \eqref{eq:PiDGapproxmlocal} and
	\eqref{eq:PiDGapproxmglobal}, we refer to \cite[(A.15) \& (A.16)]{dkrt-ldg}.~The assertions \eqref{eq:PiDGapproxmglobal2} and
	\eqref{eq:PiDGapproxmglobal1} follow from \eqref{eq:emb} and
	\eqref{eq:pol-trace}, resp., by summation, while
        \eqref{eq:PiDGapproxmglobal3} follows from
        \eqref{eq:PiDGapproxmlocal} by summation.
\end{proof}

\begin{corollary}\label{cor:PiDGapproxmspecial}
	Let $\AAA$ satisfy Assumption~\ref{ass:1} for a balanced
	{N-func\-tion} ${\phi}$ and $k\in \mathbb{N}_0$. Moreover, let $\bfu\in W^{1,\phi}(\Omega) $ satisfy $\bfF(\nabla \bfu) \in W^{1,2}(\Omega)$. Then, for every $\bfw_h\in W^{1,\phi}(\mathcal{T}_h) $, it holds
	\begin{align}
		\label{eq:PiDGapproxmspecialglobal}
		m_{\phi_{\abs{\nabla\bfu}},h}\big(\bw_h - \PiDG \bw_h\big) \leq c\,\rho_{\phi_{\abs{\nabla\bu }},\Omega}(\nabla_h\bw_h)+c\,h^2\,\|\nabla\bfF(\nabla\bfu)\|_2^2\,,
	\end{align}
	with  constants $c>0$ depending only on $k\in\mathbb{N}_0$,
        the characteristics of $\AAA$ and $\phi$, and $\omega_0>0$.
\end{corollary}

\begin{proof}
  Using the convexity of $\phi_{\abs{\nabla\bfu}}$, twice a change of
  shift (cf.~Lemma \ref{lem:change2}), Proposition \ref {lem:hammer},
  the local trace inequalities \eqref{eq:PiDGapproxmlocal} and
  \eqref{eq:emb} and \Poincare's inequality on $K$ together with
  \cite[Lemma A.12]{bdr-phi-stokes} we find that
	\begin{align*}
          &h \int_\gamma
            \phi_{\abs{\nabla\bfu}}\big(h^{-1}\abs{\jump{(\bfw_h -
            \PiDG \bfw_h)\otimes\bfn}}\big)\,\textup{d}s
          \\
          &\leq h\sum_{K\in \mathcal{T}_h;K\subseteq
            S_\gamma}{\int_\gamma{\phi_{\abs{\mean{\nabla\bfu}_K}}\big(h^{-1}
            \abs{ \bfw_h - (\PiDG
            \bfw_h)|_K}\big)+\abs{\bfF(\nabla\bfu)
            -\bfF(\mean{\nabla\bfu}_K)}^2\,\textup{d}s}} 
          \\
          & \leq \sum_{K\in \mathcal{T}_h;K\subseteq
            S_\gamma}{\int_K{\phi_{\abs{\mean{\nabla\bfu}_K}}\big(\vert
            \nabla\bfw_h\vert\big)\,\textup{d}s}}+h^2\sum_{K\in
            \mathcal{T}_h;K\subseteq
            S_\gamma}{\int_K{\abs{\nabla\bfF(\nabla\bfu)}^2\,\textup{d}s}} 
          \\
          &\leq \int_{S_\gamma}{\phi_{\abs{\nabla\bfu}}\big(\vert
            \nabla_h\bfw_h\vert\big)\,\textup{d}s}+h^2\int_{S_\gamma}{\abs{\nabla\bfF(\nabla\bfu)}^2\,\textup{d}s}\,.  
	\end{align*}
	Then, the assertion follows by summing with respect to $\gamma\in \smash{\Gamma_h^{i}}\cup\Gamma_D$.
\end{proof}

\begin{lemma}\label{lem:conv}
	Let $\psi:\setR^{\ge 0}\to \setR^{\ge 0}$ be an N-function satisfying the
	$\Delta_2$-condition and $k\in \mathbb{N}$. Then, for every $\bu \in W^{1,\psi}(\Omega)$, we have that
	\begin{alignat}{2}
		\label{eq:grad}
		\rho_{\psi,\Omega}\big(\nabla_h(\PiDG \bu -\bu)\big)&\to 0 &&\quad (h\to0)\,,
		\\
		\label{eq:mod}
		m_{\psi,\Omega}\big(\PiDG \bu -\bu\big)&\to 0 &&\quad (h\to0)\,,
		\\
		\label{eq:R}
		\rho_{\psi,\Omega}\big(\Rhk(\PiDG \bu -\bu)\big)&\to 0 &&\quad (h\to0)\,,
		\\
		\label{eq:G}
		\rho_{\psi,\Omega}\big(\Ghk(\PiDG \bu -\bu)\big)&\to 0 &&\quad (h\to0)\,.
	\end{alignat}
\end{lemma}

\begin{proof}
	For every $\bu \in W^{1,\psi}(\Omega)$, there exists a sequence
	$(\bu^n)_{n\in \mathbb{N}} \subseteq C^{\infty}(\overline {\Omega})$ such that
	\begin{align}
		\rho_{\psi, \Omega}\big(\nabla (\bu^n-\bu)\big)\to 0 \quad  (n \to
		\infty)\,.\label{eq:1b}
	\end{align}
	Thus, using \eqref{eq:PiDGLW1psistable},
	\eqref{eq:PiDGapprox1} and the properties of the N-function $\psi$,
	we find that
	\begin{align}
          \label{eq:1a}
          \begin{aligned}
            \rho_{\psi,\Omega}\big(\nabla_h(\PiDG \bu -\bu)\big) &\le
            c\,\rho_{\psi,\Omega}\big(\nabla_h\PiDG (\bu -\bu^n)\big)
            + c\, \rho_{\psi,\Omega}\big(\nabla_h(\PiDG \bu^n
            -\bu^n)\big)
            \\
            &\quad + c\, \rho_{\psi,\Omega}\big(\nabla(\bu^n
            -\bu)\big)
            \\
            &\le c\, \rho_{\psi,\Omega}\big(\nabla(\bu^n -\bu)\big)
            +c\, \rho_{\psi,\Omega}\big(h\,\nabla^2\bu^n\big)
            \,.
		\end{aligned}
	\end{align}
	Due to \eqref{eq:1b}, for every $\vep >0$, there
	exists $n_0\in \setN$ such that
	${c\rho_{\psi,\Omega}(\nabla(\bu^n -\bu)) \!\le\! \vep}$~for all $n\in\mathbb{N}$ with $n\ge n_0$. 
	Therefore,  choosing $n=n_0\in \mathbb{N}$ in \eqref{eq:1a}, we~\mbox{conclude}~that $ \lim_{h\to 0} \rho_{\psi,\Omega}(\nabla_h(\PiDG \bu -\bu))
	\!\le\! \vep$,
	which yields \eqref{eq:grad}, since $\vep\!>\!0$~was~chosen~\mbox{arbitrarily}.
	
	Choosing $\PiDG \bu -\bu= \PiDG \bu -\bu- \PiDG (\PiDG \bu -\bu)$
	in \eqref{eq:PiDGapproxmglobal}, since ${W^{1,\psi}(\Omega) \subseteq \WDGpsi}$,
	also using \eqref{eq:grad}, we find that $  m_{\psi,\Omega}(\PiDG \bu -\bu) \le c\,
	\rho_{\psi,\Omega}(\nabla_h(\PiDG \bu -\bu))\to 0$~${(h\to 0)}$,
	which is \eqref{eq:mod}.
	
	Next, \eqref{eq:Rgammaest} and \eqref{eq:mod} yield $\rho_{\psi,\Omega}(\Rhk(\PiDG \bu -\bu)) \le c\,m_{\psi,\Omega}(\PiDG \bu -\bu) \to 0$~$(h\to 0)$,
	which is \eqref{eq:R}.
	
	Finally, \eqref{eq:G} follows from the definition of
	$\Ghk$ in \eqref{eq:DGnablaR}, \eqref{eq:grad} and \eqref{eq:R}. 
\end{proof}
\begin{lemma}
	\label{lem:poincareWDGz1psi}
	Let $\psi:\setR^{\ge 0}\to \setR^{\ge 0}$ be an N-function
        such that $\psi $ and $\psi^*$ satisfy the
        $\Delta_2$-condition. Then, for~every~$\bw_h \in \WDGpsi$, we have that
	\begin{align}
		\rho_{\psi,\Omega}(\bw_h) &\leq c\,
		M_{\psi,h}(\diameter(\Omega)\,  \bw_h)\,,\label{eq:poincare}
		\\
		\rho_{\psi, \Gamma_N}(\bw_h) &\leq c\, \diameter(\Omega)^{-1}
		M_{\psi,h}(\diameter(\Omega)\, \bw_h)\,,\label{eq:trace}
	\end{align}
	where $c>0$ only depends on $\Omega, \Omega'\subseteq\setR^n$, $n\ge 2$, and $\Delta_2(\psi), \Delta_2(\psi^*),\omega_0>0$.
\end{lemma}
\begin{proof}
	This is proved in \cite[Lemma A.9, Lemma A.10]{dkrt-ldg}.
\end{proof}

\begin{lemma}\label{lem:wc}
	Let $\psi:\setR^{\ge 0}\to \setR^{\ge 0}$ be an N-function 
        such that $\psi $ and $\psi^*$ satisfy the
        $\Delta_2$-condition and $k\!\in\! \mathbb{N}$. Let $\bw_h\! \in\! \WDGpsi$, $h\!>\!0$,~be~such~that~${\sup_{h>0}M_{\psi,h}(\bw_h) \!\le\! c}$. Then, for the sequence~$\bw_n\vcentcolon=\bw_{h_n}\in W^{1,\psi}(\mathcal{T}_{h_n})$, where $h_n \to 0$  $(n\to \infty)$, there exists a function
	$\bw \in W^{1,\psi}_{\Gamma_D}(\Omega)$ such that, up to subsequences, we have that
	\begin{alignat}{3}
		\label{eq:weakpsi}
		\bw_n &\weakto \bw &&\quad\text { in } L^\psi(\Omega)&&\quad(n\to \infty)\,,
		\\
		\label{eq:weaknabla}
		\boldsymbol{\mathcal{G}}_{h_n}^k\bw_n &\weakto \nabla\bw &&\quad\text { in } L^\psi(\Omega)&&\quad(n\to \infty)\,,
		\\
		\label{eq:weaktrace}
		\bw_n &\weakto \bw &&\quad\text { in } L^\psi(\Gamma_N)&&\quad(n\to \infty)\,.
	\end{alignat}
\end{lemma}
\begin{proof}
	The proof is a straightforward adaptation of the proof of \cite[Theorem
	5.7]{ern-book}.  In fact, from Poincar\'e's inequality
	\eqref{eq:poincare} and the reflexivity of $L^\psi(\Omega)$, it  follows
	that there exists $\bw \in L^\psi (\Omega)$ such that for a not
	relabeled subsequence, it holds \eqref{eq:weakpsi}. We extend both
	$\bw$ and $\bw_n$ by zero to $\Omega'\setminus \Omega$ and denote
	the~extensions~again~by~$\bw$~and~$\bw_n$, respectively. Moreover,
	we extend $\smash{\boldsymbol{\mathcal{R}}_{h_n}^k \bw_n}$ by zero to $\Omega'\setminus \Omega$ and denote
	the extension again by $\smash{\boldsymbol{\mathcal{R}}_{h_n}^k \bw_n}$. Using these
	extensions and \eqref{eq:equi}, we obtain a not relabeled~sub-sequence and a function $\bG \in L^\psi (\Omega')$ such that
	\begin{align}\label{eq:4b}
		\boldsymbol{\mathcal G}_{h_n}^k\bw_n &\weakto \bfH \quad\text { in } L^\psi(\Omega')\quad(n\to\infty)\,.
	\end{align}
	We have to show that $\bfH =\nabla \bw$ holds in $L^\psi(\Omega')$. To this end, we observe that~for~every $\bX \in C^{\infty}_0 (\Omega')$, there holds
	\begin{align}
		\label{eq:4a}
		\begin{aligned}
			\big (\boldsymbol{\mathcal G}_{h_n}^k\bw_n, \bX\big )_{\Omega'} &= \big( \nabla _{h_n} \bw_n,
			\bX\big )_{\Omega'} - \big (\boldsymbol{\mathcal R}_{h_n}^k\bw_n ,\Uppi_{h_n}^k \bX\big
			)_{\Omega'}
			\\
			&= -  ( \bw_n , \divo \bX )_{\Omega'} + \big\langle
			\llbracket\bfw_n\otimes\bfn\rrbracket,\big\{\bfX-\Uppi_{h_n}^k
			\bX\big\}\big\rangle_{\smash{\smash{\Gamma_h^{i}}'\cup\Gamma_D'}}\,.
		\end{aligned}
	\end{align}
	Using Young's inequality, $\sup_{n\in \mathbb{N}}{m_{\psi,h_n}(\bw_n)}\leq \sup_{n\in \mathbb{N}}{M_{\psi,h_n}(\bw_n)}<\infty$ and \eqref{eq:PiDGapproxmglobal} for $\psi^*$, by passing for $n\to \infty$, for every $\bX \in C^{\infty}_0 (\Omega')$, we find that
	\begin{align*}
		\big\langle
		\llbracket\bfw_n\otimes\bfn\rrbracket,\big\{\bfX-\Uppi_{h_n}^k
		\bX\big\}\big\rangle_{\smash{\smash{\Gamma_h^{i}}'\cup\Gamma_D'}} \le h_n c \big (m_{\psi^*,h_n}\big(\bX -\Uppi_{h_n}^k \bX\big) +
		m_{\psi,h_n}(\bw_n)\big )  
		\to 0\,.
	\end{align*}
	Thus, by passing for $n\!\to\! \infty$ in \eqref{eq:4a}, using
	\eqref{eq:4b} and \eqref{eq:weakpsi},~for~any~${\bX \in C^{\infty}_0 (\Omega')}$, we arrive at $(\bfH , \bX)_{\Omega'} = -  ( \bw , \divo \bX)_{\Omega'}$,
	i.e., $\bfH  =\nabla \bw$ in $\Omega'$ and, thus,
	$\bw\hspace*{-0.1em} \in\hspace*{-0.1em} W^{1,\psi}(\Omega')$. Since $\bw =\bfzero$ in
	$\Omega' \setminus \Omega$,~we~get~${\bw \hspace*{-0.1em}\in\hspace*{-0.1em}
		W^{1,\psi}_{\Gamma_D}(\Omega)}$.
	
	Inequality \eqref{eq:trace} and the reflexivity of
	$L^\psi(\Gamma_N)$ yield a not relabeled subsequence  and a function $\bg \in L^\psi (\Gamma_N)$
	such that 
	\begin{align}\label{eq:4c}
		\bw_n &\weakto \bfh \quad\text { in } L^\psi(\Gamma_N)\quad(n\to \infty)\,.
	\end{align}
	Similar arguments as above yield that for every $\bX \in C^{\infty} (\overline\Omega)$, we have that
	\begin{align*}
		\big (\boldsymbol{\mathcal G}_{h_n}^k\bw_n, \bX\big )_{\smash{\Omega}} &\!=\! - (
		\bw_n , \divo \bX)_{\Omega} \!+\! \big\langle
		\llbracket\bfw_n\otimes\bfn\rrbracket,\!\big\{\bfX-\Uppi_{h_n}^k
		\bX\big\}\big\rangle_{\smash{\smash{\Gamma_h^{i}} \cup \Gamma_D}} \!+\!\langle
		\bfw_n\otimes\bfn, \bX\rangle_{\Gamma_N} .
	\end{align*}
	Taking the limit with respect to $n\to \infty$ in this equality, we find that
	\begin{align*}
		(\nabla \bw, \bX )_{\Omega}
		&= - \big (
		\bw , \divo \bX )_{\Omega} +\langle
		\bfh\otimes\bfn, \bX\rangle_{\Gamma_N}
		\\
		&
		=(\nabla \bw, \bX)_{\Omega} +\langle
		(\bfh-\bw)\otimes\bfn, \bX\rangle_{\Gamma_N} \,.
	\end{align*}
	Choosing $\bX = \bz \otimes \bn$ for arbitrary $\bz \in C^\infty
	_0(\Gamma_N)$, we conclude that $\bfh=\bw$ in $L^{\psi}(\Gamma_N)$, which together with
	\eqref{eq:4c} proves \eqref{eq:weaktrace}.  
\end{proof}

\begin{samepage}
\def\cprime{$'$} \def\cprime{$'$} \def\cprime{$'$}
\ifx\undefined\bysame
\newcommand{\bysame}{\leavevmode\hbox to3em{\hrulefill}\,}
\fi

\end{samepage}

\end{document}